%% file: main.tex
\documentclass[10pt,reqno,a4paper]{amsart}
\usepackage[scale=0.75, centering, headheight=14pt]{geometry}

\usepackage[T1]{fontenc}
\usepackage{lmodern}
\usepackage[english]{babel}
\usepackage{stmaryrd}
\usepackage{tikz}
\usepackage{bbm}
\usepackage{amsmath,amssymb,amsfonts,amsthm}
\usepackage{mathtools,accents}
\usepackage{mathrsfs}
\usepackage{xfrac}
\usepackage{array} 
\usepackage{aliascnt}
\usepackage{booktabs} 
\usepackage{array} 

\usepackage{verbatim} 
\usepackage{mathrsfs}
\usepackage{mathrsfs, dsfont}
\usepackage{amssymb}
\usepackage{amsthm}
\usepackage{amsmath,amsfonts,amssymb,esint}
\usepackage{graphics,color}
\usepackage{enumerate}
\usepackage{mathtools,centernot}

\usepackage{caption}
\usepackage{subcaption}
\usepackage{graphicx}

\usepackage{microtype}
\usepackage{paralist} 
\usepackage{cases}
\usepackage[initials,alphabetic]{amsrefs}
\allowdisplaybreaks

\usepackage{braket}
\usepackage{bm}

\usepackage[citecolor=blue,colorlinks]{hyperref}
\addto\extrasenglish{}

\usepackage{enumerate}
\usepackage{xcolor}

\usepackage{aliascnt}
\newtheorem{theorem}{\bf Theorem}[section]
\newtheorem{remark}[theorem]{\bf{Remark}}

\DeclareMathOperator{\diver}{div}

\numberwithin{equation}{section}

\title[Non-uniqueness for advection-diffusion equation]{Non-uniqueness of parabolic solutions for advection-diffusion equation}

\author{Th\'er\`ese Moerschell}
\address{Institute of Mathematics, ETH, Z\"urich, Switzerland}
\email{therese.moerschell@math.ethz.ch}

\author{Massimo Sorella}
\address{Department of Mathematics, Imperial College London, London, SW7 2AZ, UK}
\email{m.sorella@imperial.ac.uk}

\input{analysis_preamble}

\input{nonuniADE_preamble}

\usepackage{xcolor}
\usetikzlibrary{math}
\usepackage{pgfplots}

\begin{document}

\definecolor{vert}{rgb}{0.03,0.4,0.03}
\newcommand{\com}[1]{\textcolor{blue}{#1}}
\newcommand{\blue}[1]{\textcolor{blue}{#1}}
\newcommand{\red}[1]{\textcolor{red}{#1}}

\begin{abstract}
We present a novel example of a divergence--free velocity field $b \in L^\infty  ((0,1);  L^p (\T^2))$ for $p<2$ arbitrary but fixed which leads to non-unique solutions of advection--diffusion in the class $L^\infty_{t,x} \cap L^2_t H^1_x$ while satisfying the local energy inequality. This result complements the known uniqueness result of bounded solutions for divergence-free and $L^2_{t,x}$ integrable velocity fields. 
Additionally, we also prove the necessity of time integrability of the velocity field for the uniqueness result. More precisely, we  construct another divergence--free velocity field $b \in L^p  ((0,1);  L^\infty (\T^2))$, for $p< 2$ fixed, but arbitrary, with non--unique aforementioned solutions.  
Our contribution closes the gap between the regime of uniqueness and non-uniqueness in this context. Previously, it was shown with the convex integration technique that for $d\geq 3$ divergence--free velocity fields
 $ b \in L^\infty ((0,1) ;L^p (\T^d))$ with $p < \frac{2d}{d+2}$ could lead to non--unique solutions in 
the space $L^\infty_t L^{\sfrac{2d}{d-2}}_x \cap L^2_t H^1_x$.
 Our proof is based on a  stochastic Lagrangian approach and {does} not rely on convex integration.
\end{abstract}

\maketitle
 
\section{Introduction}

In this paper we study the Cauchy problem of the advection{--}diffusion equation with divergence free velocity fields $b : [0,T] \times \T^d \to \R^d$, where $\T^d \cong \R^d / \Z^d$ is the $d$ dimensional torus {with $d \ge 2$}, more precisely 
\begin{equation}\label{eq:ADE}\tag{ADV-DIFF}
    \begin{cases}
    \partial_t \solAD  + b \cdot \nabla \solAD  = \Delta \solAD & \text{ on } [0,T] \times \T^d,\\
    \solAD(0, \cdot ) = \ADin (\cdot ) \in L^\infty & \text{ on }\T^d\,,
    \end{cases}
\end{equation}
where $\ADin : \T^d \to \R$ is a given bounded initial datum and $\solAD : [0,T] \times \T^d \to \R$ is the unknown. If $b \in C^\infty$ then it is classical to prove existence and uniqueness of solutions  satisfying the energy balance, see e.g. the standard PDEs textbook \cite{E98}
$$ \frac{1}{2} \int_{\T^d } | \solAD (t,x)|^2 dx + \int_0^t \int_{\T^d} | \nabla \solAD (s,x)|^2 dx ds =  \frac{1}{2} \int_{\T^d } | \ADin (x)|^2 dx \,.  $$

It is known that for divergence free velocity fields $b \in L^2 ({[0,T]} \times \T^d)$ in any $d \geq 2${,} bounded distributional solutions to the advection{--}diffusion equation are unique and parabolic, see  
\cite{BCC24} and previous results related to this equation with rough velocity fields in \cites{F08,LL19,LSU68}.
 We give here the definition of parabolic solutions. 
\begin{df}[Parabolic solutions] \label{d:parabolic}
Let $b \in L^1_{t,x}$ be a divergence-free vector field. We say that $\solAD$ {$ \in L^\infty_{t,x}$} is a parabolic solution to the \eqref{eq:ADE} if $\solAD$ is a distributional solution, i.e. it holds
\begin{equation*} 
\int_0^T \int_{\T^{{d}} } \solAD (t,x) ( \partial_t \varphi(t,x) + b\cdot \nabla \varphi (t,x) + \Delta \varphi (t,x) )  dx dt =  {-}\int_{\T^d }  \ADin (x) \varphi(0, x) dx \,,
\end{equation*}  
for any $\varphi \in C^\infty_c ([0,T) \times \T^d)$ and $\solAD \in L^\infty_{t,x} \cap L^2_t H^1_x$   satisfies the local energy inequality, i.e. it holds
\begin{equation} \label{local:energy}
\partial_t \frac{|\theta|^2}{2} + \diver \left (b \frac{|\theta|^2}{2} \right ) \leq \Delta \frac{|\theta|^2}{2} - |\nabla \theta|^2 \,,
\end{equation}
for any non-negative $\varphi \in C^\infty_c ([0,T)  \times \T^d)$.
\end{df} 
Integrating {\eqref{local:energy}} in space--time{,} it is possible to show  the global energy inequality
$$ \frac{1}{2} \int_{\T^d } | \solAD (t,x)|^2 dx + \int_0^t \int_{\T^d} | \nabla \solAD (s,x)|^2 dx ds \leq  \frac{1}{2} \int_{\T^d } | \ADin (x)|^2 dx \,.  $$
\begin{remark}
{In \cite{BCC24}, given $d \geq 2$ and $b \in L^2 ({[0,T]} \times \T^d)$ divergence free, Bonicatto, Ciampa and Crippa  use a commutator estimate to prove the existence of a unique bounded solution satisfying the global energy equality. The same argument can be used to show the existence of a unique parabolic solution (according to Definition \ref{d:parabolic}) which moreover satisfies the local energy equality 
\[\partial_t \frac{|\theta|^2}{2} + \diver \left (b \frac{|\theta|^2}{2} \right ) = \Delta \frac{|\theta|^2}{2} - |\nabla \theta|^2 \,.\]
The non-unique solutions in Theorem \ref{thm:NU_ADE_loops} and Theorem \ref{thm:NU_ADE_chess} are obtained as the limit of a sequence of solutions to \eqref{eq:ADE} with regularized  velocity fields. The strong convergence of this sequence in the $L^2_t H^1_x$ topology is not guaranteed. Although the  solutions with the regularized velocity fields satisfy the local energy equality, we can only guarantee an inequality for the limit points.  This loss of strong convergence in the $L^2_t H^1_x$ topology is analogous to the phenomenon encountered in the proof of existence of suitable Leray solutions for the 3D Navier–Stokes equations.
It remains an open problem to prove non-uniqueness of parabolic solutions to \eqref{eq:ADE} with a integrable divergence-free velocity field satisfying the local energy equality. }
\end{remark}

With the convex integration technique, introduced by De Lellis and  Sz\'ekelyhidi  in \cite{DLL09,DLL13}, it is possible to show the existence of a divergence free velocity field $b \in L^\infty ((0,1);  L^p (\T^d)) $, with $p < \frac{2d}{d+2}$ and $d \geq 3$ arbitrary but fixed, for which solutions  for the advection{--}diffusion equation in  $L^\infty_t L^{\sfrac{2d}{d-2}}_x \cap L^2_t H^1_x$ are non--unique, see this result by Modena and Sattig in \cite{MSa20}. {Their result covers also Sobolev vector fields and other regularity ranges when one does not require a solution in $L_t^2H_x^1$, and they achieve a similar result for the advection equation.} 
 The convex integration technique for the advection{--}diffusion equation has been introduced  in \cite{MS18}  by Modena and  Sz\'ekelyhidi. We also refer to \cites{BDLC20,PS23,SG21,CL22}  for convex integration results for the advection equation. 
 
  However, there is a significant gap between uniqueness and non-uniqueness of parabolic solutions in terms of the regularity of the velocity field, as we summarize here:
\begin{itemize}
\item Uniqueness for divergence free $b \in L^2_{t,x} (\T^d)$ and $d\geq 2$: it is known that parabolic solutions, as in Definition \ref{d:parabolic}, are unique via commutator estimates or maximal parabolic regularity methods, see for instance \cite{BCC24}. 
\item Non-uniqueness for divergence free $b \in L^\infty ((0,1); L^p (\T^d))$ for $p < \frac{2d}{d+2}$ and $d \geq 3$: it is known that solutions in $L^\infty_t L^{\sfrac{2d}{d-2}}_x \cap L^2_t H^1_x$ are non-unique thanks to the convex integration method, see  \cite[Theorem 1.4]{MSa20} applied with parameters $m=1, \tilde{m} =0, p=2, k=2$. We highlight that in this paper it is not proved any local energy inequality \eqref{local:energy}, so the solutions are not even parabolic.
\end{itemize}


In this context we prove the following results, proving that the $L^2_{t,x}$ integrability condition on the divergence--free velocity field is necessary to have uniqueness of parabolic solutions.

\begin{thm}\label{thm:NU_ADE_loops}
    For every $p < 2$ there exists a divergence free velocity field $\bb \in L^\infty({[0,T]}; L^p (\T^2))$ and an initial data $\solAD_{in} \in L^\infty(\T^2)$ such that the advection--diffusion equation \eqref{eq:ADE} admits at least two distinct parabolic solutions as in Definition \ref{d:parabolic}. 
\end{thm}

\begin{thm}\label{thm:NU_ADE_chess}
    For every $p < 2$ there exists a divergence free velocity field $\bb \in L^p({[0,T]}; L^{\infty} (\T^2))$ and some initial data $\solAD_{in} \in L^\infty(\T^2)$ such that the advection--diffusion equation \eqref{eq:ADE} admits at least two distinct  parabolic solutions as in Definition \ref{d:parabolic}.
\end{thm}
\begin{remark}
Our result applies to any dimension $d \geq 2$, lifting the velocity fields and the initial datum  as in the statement to be independent of the last $d-2$ variables. 
\end{remark}

These theorems prove  non-uniqueness of parabolic solutions, in particular these solutions satisfy the local energy inequality. It is clear that distributional solutions are more flexible, therefore it is harder to prove non-uniqueness of parabolic solutions and very few examples of non--uniqueness are known with an additional local energy inequality. The definition of  parabolic solutions reminds the one of {\em suitable} Leray-Hopf for the 3D Navier--Stokes equations \cites{leray,H51,CKN82}.  The local energy inequality has been exploited 
  in \cite{CKN82} to prove that the singular set of  3D Navier--Stokes equations solutions has parabolic Hausdorff dimension  less or equal than one, improving the previous result by Scheffer \cite{S80}. There is an extensive literature on local and global energy conservation for advection and advection--diffusion equations, Navier--Stokes and Euler equations, see for instance \cites{DPL89,CET,A04,CCFS08,LL19,BCC23,BCC24}. 
We also refer to \cites{BCC23,Gautam,BrueN21,LL08,LLuo19,NavS22} for some recent results with vanishing diffusivity coefficient for advection{--}diffusion equation and references therein.
Here we provide a brief overview of some techniques for the non-uniqueness of solutions in various contexts.
\\
\\
The convex integration technique is a very general technique that is used to prove non-uniqueness of solutions in mild spaces for several equations. This technique, introduced in \cite{DLL09,DLL13} as cited above, has been intensively studied to solve important problems regarding non-uniqueness and non conservation of energy for solutions in fluid dynamics, see for instance some important results for 3D Euler equations \cites{Is18,NV22}. Particularly, we point out that very recently in \cite{GKN23} Giri, Kwon and Novack  prove, with the convex integration technique, that there are 3D Euler equations solutions which dissipate the energy and satisfy the local energy inequality. Taking inspiration from this, it might be possible to impose a local energy inequality on the  non-unique solutions for the advection--diffusion equation given by Modena and Sattig \cite{MSa20}. Finally, 
 in the pioneering paper  \cite{BV18} Buckmaster and Vicol prove the non--uniqueness of weak solutions for the 3D Navier--Stokes equations, paving the way for several results on non--uniqueness with Laplacian (or fractional Laplacian) regularization, see for instance \cite{BMS21,DR19,CDRS22,MS18}.
  
  Very recently, Vishik in \cite{Vishik1,Vishik2} introduce the idea of proving non--uniqueness via spectral method, see also  \cite{ABCDL+,CFMS24}. Vishik proves the existence of a  self--similar profile which admits an eigenvalue with positive real part 
  for the linearized  operator of 2D Euler equations. Eventually, Vishik proves the non--uniqueness of 2D Euler equations solutions  with a suitable body force  in sharp classes of regularity predicted by Yudovich \cite{Yudovich1963}. These solutions also enjoy the local energy equality taking into account also the force in the energy balance. 
   Albritton, Bru\'e and Colombo use this technique in \cite{ABC22} to prove examples of non--uniqueness of  suitable Leray--Hopf solutions for 3D Navier--Stokes with a body force, which in particular satisfy  the local energy inequality, see also \cites{ABC23,AC23} for related results. This technique is tailored to specific non--linear problems that admit an eigenvalue with positive real part. 
   
   Another technique is to find explicit examples, not relying on convex integration, where non--uniqueness of solutions follows more from a Lagrangian perspective. This technique is still the optimal one in some contexts and recently has led to the solution of some open problems in the advection equation setting.
   In \cite{K23} Kumar proves that there are H\"older continuous Sobolev velocity fields for which non-uniqueness of trajectories holds on the full torus. The H\"older continuity of the velocity field is sharp and completely new with respect to previous works obtained in convex integration. In \cite{BCK24tr} Bru\'e, Colombo and Kumar {further developed} this construction introducing new ideas to obtain non-unique positive solutions to the advection equation with Sobolev velocity fields in sharp range of regularity.   This explicit velocity field has also inspired the same authors in \cite{BCK24Eu} to introduce a new ``asynchronization'' idea for the building blocks  in the convex integration technique to  solve a longstanding open problem of non-uniqueness of 2D Euler equations solutions with integrable vorticity. The full answer to non-uniqueness of 2D Euler equations solutions with $L^p$ vorticity ($p< \infty$ arbitrary but fixed) is still open, see the following results going into this direction \cites{BS21,BM20,Vishik1,Vishik2,ABCDLGJH}.
\\
\\
The strategies of the proofs of Theorem \ref{thm:NU_ADE_loops}  and Theorem \ref{thm:NU_ADE_chess}  are similar and closely align to the last aforementioned technique.  The proofs are based on the study of the Lagrangian flow and Stochastic flow of explicit velocity fields. The construction of this velocity field does not rely on convex integration, though it is inspired by the use of super-exponential frequency separation as in convex integration. 
We highlight that our building block for the construction of the velocity field in Theorem \ref{thm:NU_ADE_loops} can be made a stationary solution of 2D Euler equations, making it radial and acting only on the angular component.

The construction of our  velocity field  to prove  Theorem \ref{thm:NU_ADE_loops} is completely new whereas the one  to prove Theorem \ref{thm:NU_ADE_chess}  is a rescaled version  of the velocity field constructed in \cite{CCS23} where it is used to prove anomalous dissipation and lack of selection via vanishing diffusivity for the 2D advection equation under H\"older regular velocity fields. { This velocity field takes inspirations from previous well known examples, see \cite{Depauw,A78}.}

\subsection{Heuristics and idea of the proofs} \label{subsec:heuristic-idea}

Here we primarily discuss the idea behind the proof of Theorem \ref{thm:NU_ADE_loops}, even although the approach for the other theorem is similar. { For more details on the main ideas behind the construction of the velocity field for Theorem \ref{thm:NU_ADE_loops}, we refer to Subsection \ref{subsec:idea-loop}, and for the construction in Theorem \ref{thm:NU_ADE_chess}, we refer to Subsection \ref{subsec:idea-chess}.}

This heuristics is inspired by anomalous dissipation phenomena, see for instance \cites{DEIJ22,CCS23,AV23,JS23,EL23,R23,BBS23,JS24}, but it is even more pathological, since here we fix the diffusivity parameter $\kappa=1$.

The proof of this theorem requires a novel  construction of a divergence free velocity field $b$. The construction is completely explicit and the non-uniqueness is naturally coming from the advection equation and the  study of a system with high  P\'eclet number 
\begin{equation} \label{Peclet}
{\rm  Pe =  \frac{LV}{\kappa} = \frac{\sfrac{L^2}{\kappa}}{\sfrac{L}{V}} =  \frac{ diffusion \,  \, time}{ advection \, \, time}  \gg 1 } \,,
\end{equation}
where $L$ and $V$ are respectively the characteristic length and velocity of the system and $\kappa >0 $ is the diffusivity parameter that in our case is $\kappa=1$. 
In this case, heuristically, one expects the solution to the  advection{--}diffusion equation to  behave similarly to the solution to the advection equation. More precisely, we introduce the advection equation as the following Cauchy problem
\begin{equation}\label{eq:TE}\tag{TE}
    \begin{cases}
        \partial_t \solTE + b  \cdot \nabla  \solTE  = 0 & \text{ on } [0,T] \times \T^2,\\
        \solTE(0, \cdot ) = \TEin(\cdot ) \in L^\infty & \text{ on }\T^2.
    \end{cases}
\end{equation} 
with a given velocity field $b : [0,T] \times \T^2 \to \R^2$. The main strategy is to approximate the velocity field $b$ with a sequence of divergence free velocity fields $b_n \to b $ as $n \to \infty$ in $L^\infty_t L^p_x$, $b_n \in L^\infty_t BV_x $ for a given $p< 2$ arbitrary but fixed.
We represent the solution to the advection{--}diffusion equation with  the Feynman Kac formula with backward Stochastic flow $\bY^n_{{t}, 0}$ of $b_n$  and the solution to the advection equation with backward regular Lagrangian flow $\bX^n_{t,0}$ of $b_n$ as 
\begin{equation}\label{eq:back_repres}
	\solAD_n (t,x) =  \E \big[ \ADin\big( \bY^n_{{t}, 0} (x,\cdot ) \big) \big] \qquad \text{and} \qquad \solTE_n (t,x) = \TEin\big( \bX^n_{t,0} (x)\big)\,, 
\end{equation}
and we impose the same initial datum $\TEin = \ADin$. The following conditions are sufficient to prove the existence of two distinct solutions to the advection{--}diffusion equation: 
\begin{itemize}
\item The sequences of solutions $\{\rho_{2n} \}_{n \in \N}$, $\{\rho_{2n+1} \}_{n \in \N}$  to the advection equation, with velocity fields $\{ b_{2n}\}_{n \in  \N}$, $\{ b_{2n +1}\}_{n \in  \N}$ respectively, satisfy the following quantitative non--uniqueness condition:
$$\rho_{2n_k} \overset{*}{\rightharpoonup} \rho_{\text{even}}  \neq \rho_{\text{odd}}  \overset{*}{\leftharpoonup} \rho_{2n_k+1} \,, \qquad \text{as } n_k \to \infty$$
and 
\begin{equation} \label{intro:1}
\| \rho_{\text{even}}  (T, \cdot ) - \rho_{\text{odd}} (T, \cdot ) \|_{L^1 (\T^2)} \geq 2\varepsilon >0 \,, \quad \text{ for some } \, T, \varepsilon >0\,.
\end{equation}
\item The initial data are Lipschitz and the backward Regular Lagrangian flow and the backward Stochastic Flow satisfy the quantitative stability inequality with the same constant $\varepsilon >0$ given above
\begin{equation} \label{intro:2}
\| \rho_{n}  (T, \cdot ) - \theta_{n} (T, \cdot ) \|_{L^1 (\T^2)} \leq \| \ADin \|_{W^{1, \infty}} \mathbb{E} \| \bY^n_{{T}, 0} - \bX^n_{{T}, 0}   \|_{L^1 (\T^2)} < \varepsilon  \,.
\end{equation} 
\end{itemize}
Indeed, thanks to \eqref{intro:1} and \eqref{intro:2}, up to not relabelled even and odd subsequences, we have that any $L^\infty$--weak* limit solution of   the odd subsequence is different to any $L^\infty$-weak* limit solution of the even subsequence, namely
$$ \theta_{2 n_k} \overset{*}{\rightharpoonup} \theta_{\text{even}}  \neq \theta_{\text{odd}}  \overset{*}{\leftharpoonup} \theta_{2n_k +1} \,, \qquad \text{as } n_k \to \infty\,.$$

 We  construct the velocity field $b$ for Theorem \ref{thm:NU_ADE_loops} as a limit   of a sequence of velocity fields $\{ b_n \}_{n \in \N}$.
 { The non--uniqueness of solutions to the advection equation is given by the fact that Lagrangian trajectories of $b$ are pushed in a very fast way towards the center, that is the unique singularity of the velocity field, i.e. $b \in L^\infty ([0,1]; BV_{\text{loc}} (\T^2 \setminus \{ 0 \} )$, but $b \notin L^\infty ([0,1]; BV (\T^2))$.
To make the time of Lagrangian trajectories approaching the origin very short (in a super-exponential way) while maintaining the divergence-free condition, we construct the sequence of velocity fields 
$b_n$ as a finite sum of counter-clockwise rotating velocity fields with support on a finite union of pipes, as shown in Figure \ref{fig:loopstructure}.}
 \begin{figure}[htbp]
		 \begin{tikzpicture}[scale=2.14]
        \draw[->] (-1,0)--(1,0);
        \draw[->] (0,-1)--(0,1);
        \filldraw[fill=red!30!white, fill opacity=0.5, draw=red] (-0.7,0.7) rectangle (-0.8,0.8);
        \filldraw[fill=vert!30!white, fill opacity=0.5, draw=vert] (0.7,0.7) rectangle (0.8,0.8);
        \filldraw[fill=blue!30!white, fill opacity=0.5, draw=blue] (-0.7,-0.7) rectangle (-0.8,-0.8);
        \draw[dotted,->](-0.7,0.7)--(-0.3,0.3);
        \draw[dotted,->](-0.3,0.3)--(-0.01,0.01);
        \draw[dotted,->](0,0)--(0.3,0.3);
        \draw[dotted,->](0.3,0.3)--(0.7,0.7);
        \draw[dotted,->](0,0)--(-0.3,-0.3);
        \draw[dotted,->](0.3,0.3)--(-0.7,-0.7);
        \draw node at (0.9,-0.1) {$x$};
        \draw node at (0.1,0.9) {$y$};
    \end{tikzpicture}
    \caption{{In red we represent an approximation of the support of the initial datum to the \eqref{eq:ADE} we use for Theorem \ref{thm:NU_ADE_loops}.  We depict in green and in blue the supports of two distinct solutions to the advection equation at a suitable time $T_n > 0$.  The dotted lines show an approximation of Lagrangian trajectories of the velocity field $b$. The bifurcation (i.e. non-uniqueness) of Lagrangian trajectories happens only at $x=y=0$. Indeed,  the origin is the only singularity of the velocity field, i.e. $b \in L^\infty ((0,1) ; BV_{\text{loc}}( \T^2 \setminus \{0\} )$. This figure is qualitatively preserved by considering advection--diffusion equation instead of advection equation thanks to the high P\'eclet number of the system.}}
    \label{fig:nu_trajectories}
\end{figure}

\begin{figure}[!htb]
    \centering
   \begin{subfigure}{0.45\textwidth}
    \begin{tikzpicture}[scale=1]
        \draw[->] (-3,0) -- (3,0);
        \draw[->] (0,-3) -- (0,3);
	\def \i {0.3}
	\draw(-8*\i,6*\i)--(-8*\i,4*\i) arc(-180:-90:2*\i)--(-4*\i,2*\i) arc(-90:0:2*\i)--(-2*\i,6*\i)arc(0:90:2*\i) -- (-6*\i,8*\i) arc(90:180:2*\i);
	\draw(-7*\i,6*\i)--(-7*\i,4*\i) arc(-180:-90:\i)--(-4*\i,3*\i) arc(-90:0:\i)--(-3*\i,6*\i)arc(0:90:\i) -- (-6*\i,7*\i) arc(90:180:\i);
	%
	\begin{scope}[xscale=-1]
	\draw(-8*\i,6*\i)--(-8*\i,4*\i) arc(-180:-90:2*\i)--(-4*\i,2*\i) arc(-90:0:2*\i)--(-2*\i,6*\i)arc(0:90:2*\i) -- (-6*\i,8*\i) arc(90:180:2*\i);
	\draw(-7*\i,6*\i)--(-7*\i,4*\i) arc(-180:-90:\i)--(-4*\i,3*\i) arc(-90:0:\i)--(-3*\i,6*\i)arc(0:90:\i) -- (-6*\i,7*\i) arc(90:180:\i);
	\end{scope}
	%
	\begin{scope}[yscale=-1]
	\draw(-8*\i,6*\i)--(-8*\i,4*\i) arc(-180:-90:2*\i)--(-4*\i,2*\i) arc(-90:0:2*\i)--(-2*\i,6*\i)arc(0:90:2*\i) -- (-6*\i,8*\i) arc(90:180:2*\i);
	\draw(-7*\i,6*\i)--(-7*\i,4*\i) arc(-180:-90:\i)--(-4*\i,3*\i) arc(-90:0:\i)--(-3*\i,6*\i)arc(0:90:\i) -- (-6*\i,7*\i) arc(90:180:\i);
	\end{scope}
	%
	\begin{scope}[xscale=-1, yscale=-1]
	\draw(-8*\i,6*\i)--(-8*\i,4*\i) arc(-180:-90:2*\i)--(-4*\i,2*\i) arc(-90:0:2*\i)--(-2*\i,6*\i)arc(0:90:2*\i) -- (-6*\i,8*\i) arc(90:180:2*\i);
	\draw(-7*\i,6*\i)--(-7*\i,4*\i) arc(-180:-90:\i)--(-4*\i,3*\i) arc(-90:0:\i)--(-3*\i,6*\i)arc(0:90:\i) -- (-6*\i,7*\i) arc(90:180:\i);
	\end{scope}
	%
	\foreach \g in {0.1}
	\draw[red] (-16*\g,14*\g)--(-16*\g,4*\g) arc(-180:-90:2*\g)--(14*\g,2*\g) arc(-90:0:2*\g)--(16*\g,14*\g)arc(0:90:2*\g) -- (-14*\g,16*\g) arc(90:180:2*\g);
	\foreach \g in {0.1}
	\draw[red] (-15*\g,14*\g)--(-15*\g,4*\g) arc(-180:-90:\g)--(14*\g,3*\g) arc(-90:0:\g)--(15*\g,14*\g)arc(0:90:\g) -- (-14*\g,15*\g) arc(90:180:\g);
    \end{tikzpicture}
    \caption{$n=0$}
    \end{subfigure}
\hfill
  \begin{subfigure}{0.45\textwidth}
    \begin{tikzpicture}[scale=1]
        \draw[->] (-3,0) -- (3,0);
        \draw[->] (0,-3) -- (0,3);
	\def \i {0.3}
	\draw(-8*\i,6*\i)--(-8*\i,4*\i) arc(-180:-90:2*\i)--(-4*\i,2*\i) arc(-90:0:2*\i)--(-2*\i,6*\i)arc(0:90:2*\i) -- (-6*\i,8*\i) arc(90:180:2*\i);
	\draw(-7*\i,6*\i)--(-7*\i,4*\i) arc(-180:-90:\i)--(-4*\i,3*\i) arc(-90:0:\i)--(-3*\i,6*\i)arc(0:90:\i) -- (-6*\i,7*\i) arc(90:180:\i);
	\foreach \a in {0.1 }
	\draw(-16*\a,14*\a)--(-16*\a,4*\a) arc(-180:-90:2*\a)--(-4*\a,2*\a) arc(-90:0:2*\a)--(-2*\a,14*\a)arc(0:90:2*\a) -- (-14*\a,16*\a) arc(90:180:2*\a);
	\foreach \a in {0.1}
	\draw(-15*\a,14*\a)--(-15*\a,4*\a) arc(-180:-90:\a)--(-4*\a,3*\a) arc(-90:0:\a)--(-3*\a,14*\a)arc(0:90:\a) -- (-14*\a,15*\a) arc(90:180:\a);
	%
	\begin{scope}[xscale=-1]
	\def \i {0.3}
	\draw(-8*\i,6*\i)--(-8*\i,4*\i) arc(-180:-90:2*\i)--(-4*\i,2*\i) arc(-90:0:2*\i)--(-2*\i,6*\i)arc(0:90:2*\i) -- (-6*\i,8*\i) arc(90:180:2*\i);
	\draw(-7*\i,6*\i)--(-7*\i,4*\i) arc(-180:-90:\i)--(-4*\i,3*\i) arc(-90:0:\i)--(-3*\i,6*\i)arc(0:90:\i) -- (-6*\i,7*\i) arc(90:180:\i);
	\foreach \a in {0.1 }
	\draw(-16*\a,14*\a)--(-16*\a,4*\a) arc(-180:-90:2*\a)--(-4*\a,2*\a) arc(-90:0:2*\a)--(-2*\a,14*\a)arc(0:90:2*\a) -- (-14*\a,16*\a) arc(90:180:2*\a);
	\foreach \a in {0.1 }
	\draw(-15*\a,14*\a)--(-15*\a,4*\a) arc(-180:-90:\a)--(-4*\a,3*\a) arc(-90:0:\a)--(-3*\a,14*\a)arc(0:90:\a) -- (-14*\a,15*\a) arc(90:180:\a);
	\end{scope}
	%
	\begin{scope}[yscale=-1]
	\def \i {0.3}
	\draw(-8*\i,6*\i)--(-8*\i,4*\i) arc(-180:-90:2*\i)--(-4*\i,2*\i) arc(-90:0:2*\i)--(-2*\i,6*\i)arc(0:90:2*\i) -- (-6*\i,8*\i) arc(90:180:2*\i);
	\draw(-7*\i,6*\i)--(-7*\i,4*\i) arc(-180:-90:\i)--(-4*\i,3*\i) arc(-90:0:\i)--(-3*\i,6*\i)arc(0:90:\i) -- (-6*\i,7*\i) arc(90:180:\i);
	\foreach \a in {0.1 }
	\draw(-16*\a,14*\a)--(-16*\a,4*\a) arc(-180:-90:2*\a)--(-4*\a,2*\a) arc(-90:0:2*\a)--(-2*\a,14*\a)arc(0:90:2*\a) -- (-14*\a,16*\a) arc(90:180:2*\a);
	\foreach \a in {0.1 }
	\draw(-15*\a,14*\a)--(-15*\a,4*\a) arc(-180:-90:\a)--(-4*\a,3*\a) arc(-90:0:\a)--(-3*\a,14*\a)arc(0:90:\a) -- (-14*\a,15*\a) arc(90:180:\a);
	\end{scope}
	%
	\begin{scope}[xscale=-1, yscale=-1]
	\def \i {0.3}
	\draw(-8*\i,6*\i)--(-8*\i,4*\i) arc(-180:-90:2*\i)--(-4*\i,2*\i) arc(-90:0:2*\i)--(-2*\i,6*\i)arc(0:90:2*\i) -- (-6*\i,8*\i) arc(90:180:2*\i);
	\draw(-7*\i,6*\i)--(-7*\i,4*\i) arc(-180:-90:\i)--(-4*\i,3*\i) arc(-90:0:\i)--(-3*\i,6*\i)arc(0:90:\i) -- (-6*\i,7*\i) arc(90:180:\i);
	\foreach \a in {0.1 }
	\draw(-16*\a,14*\a)--(-16*\a,4*\a) arc(-180:-90:2*\a)--(-4*\a,2*\a) arc(-90:0:2*\a)--(-2*\a,14*\a)arc(0:90:2*\a) -- (-14*\a,16*\a) arc(90:180:2*\a);
	\foreach \a in {0.1 }
	\draw(-15*\a,14*\a)--(-15*\a,4*\a) arc(-180:-90:\a)--(-4*\a,3*\a) arc(-90:0:\a)--(-3*\a,14*\a)arc(0:90:\a) -- (-14*\a,15*\a) arc(90:180:\a);
	\end{scope}
	%
	\foreach \g in {0.03}
	\draw[red] (-16*\g,14*\g)--(-16*\g,-14*\g) arc(-180:-90:2*\g)--(-4*\g,-16*\g) arc(-90:0:2*\g)--(-2*\g,14*\g)arc(0:90:2*\g) -- (-14*\g,16*\g) arc(90:180:2*\g);
	\foreach \g in {0.03}
	\draw[red] (-15*\g,14*\g)--(-15*\g,-14*\g) arc(-180:-90:\g)--(-4*\g,-15*\g) arc(-90:0:\g)--(-3*\g,14*\g)arc(0:90:\g) -- (-14*\g,15*\g) arc(90:180:\g);
    \end{tikzpicture}
    \caption{$n=1$}
    \end{subfigure}
    \begin{subfigure}{0.45\textwidth}
    \begin{tikzpicture}[scale=1]
        \draw[->] (-3,0) -- (3,0);
        \draw[->] (0,-3) -- (0,3);
	\def \i {0.3}
	\draw(-8*\i,6*\i)--(-8*\i,4*\i) arc(-180:-90:2*\i)--(-4*\i,2*\i) arc(-90:0:2*\i)--(-2*\i,6*\i)arc(0:90:2*\i) -- (-6*\i,8*\i) arc(90:180:2*\i);
	\draw(-7*\i,6*\i)--(-7*\i,4*\i) arc(-180:-90:\i)--(-4*\i,3*\i) arc(-90:0:\i)--(-3*\i,6*\i)arc(0:90:\i) -- (-6*\i,7*\i) arc(90:180:\i);
	\foreach \a in {0.1, 0.03 }
	\draw(-16*\a,14*\a)--(-16*\a,4*\a) arc(-180:-90:2*\a)--(-4*\a,2*\a) arc(-90:0:2*\a)--(-2*\a,14*\a)arc(0:90:2*\a) -- (-14*\a,16*\a) arc(90:180:2*\a);
	\foreach \a in {0.1, 0.03 }
	\draw(-15*\a,14*\a)--(-15*\a,4*\a) arc(-180:-90:\a)--(-4*\a,3*\a) arc(-90:0:\a)--(-3*\a,14*\a)arc(0:90:\a) -- (-14*\a,15*\a) arc(90:180:\a);
	%
	\begin{scope}[xscale=-1]
	\def \i {0.3}
	\draw(-8*\i,6*\i)--(-8*\i,4*\i) arc(-180:-90:2*\i)--(-4*\i,2*\i) arc(-90:0:2*\i)--(-2*\i,6*\i)arc(0:90:2*\i) -- (-6*\i,8*\i) arc(90:180:2*\i);
	\draw(-7*\i,6*\i)--(-7*\i,4*\i) arc(-180:-90:\i)--(-4*\i,3*\i) arc(-90:0:\i)--(-3*\i,6*\i)arc(0:90:\i) -- (-6*\i,7*\i) arc(90:180:\i);
	\foreach \a in {0.1, 0.03 }
	\draw(-16*\a,14*\a)--(-16*\a,4*\a) arc(-180:-90:2*\a)--(-4*\a,2*\a) arc(-90:0:2*\a)--(-2*\a,14*\a)arc(0:90:2*\a) -- (-14*\a,16*\a) arc(90:180:2*\a);
	\foreach \a in {0.1, 0.03 }
	\draw(-15*\a,14*\a)--(-15*\a,4*\a) arc(-180:-90:\a)--(-4*\a,3*\a) arc(-90:0:\a)--(-3*\a,14*\a)arc(0:90:\a) -- (-14*\a,15*\a) arc(90:180:\a);
	\end{scope}
	%
	\begin{scope}[yscale=-1]
	\def \i {0.3}
	\draw(-8*\i,6*\i)--(-8*\i,4*\i) arc(-180:-90:2*\i)--(-4*\i,2*\i) arc(-90:0:2*\i)--(-2*\i,6*\i)arc(0:90:2*\i) -- (-6*\i,8*\i) arc(90:180:2*\i);
	\draw(-7*\i,6*\i)--(-7*\i,4*\i) arc(-180:-90:\i)--(-4*\i,3*\i) arc(-90:0:\i)--(-3*\i,6*\i)arc(0:90:\i) -- (-6*\i,7*\i) arc(90:180:\i);
	\foreach \a in {0.1, 0.03 }
	\draw(-16*\a,14*\a)--(-16*\a,4*\a) arc(-180:-90:2*\a)--(-4*\a,2*\a) arc(-90:0:2*\a)--(-2*\a,14*\a)arc(0:90:2*\a) -- (-14*\a,16*\a) arc(90:180:2*\a);
	\foreach \a in {0.1, 0.03 }
	\draw(-15*\a,14*\a)--(-15*\a,4*\a) arc(-180:-90:\a)--(-4*\a,3*\a) arc(-90:0:\a)--(-3*\a,14*\a)arc(0:90:\a) -- (-14*\a,15*\a) arc(90:180:\a);
	\end{scope}
	%
	\begin{scope}[xscale=-1, yscale=-1]
	\def \i {0.3}
	\draw(-8*\i,6*\i)--(-8*\i,4*\i) arc(-180:-90:2*\i)--(-4*\i,2*\i) arc(-90:0:2*\i)--(-2*\i,6*\i)arc(0:90:2*\i) -- (-6*\i,8*\i) arc(90:180:2*\i);
	\draw(-7*\i,6*\i)--(-7*\i,4*\i) arc(-180:-90:\i)--(-4*\i,3*\i) arc(-90:0:\i)--(-3*\i,6*\i)arc(0:90:\i) -- (-6*\i,7*\i) arc(90:180:\i);
	\foreach \a in {0.1, 0.03 }
	\draw(-16*\a,14*\a)--(-16*\a,4*\a) arc(-180:-90:2*\a)--(-4*\a,2*\a) arc(-90:0:2*\a)--(-2*\a,14*\a)arc(0:90:2*\a) -- (-14*\a,16*\a) arc(90:180:2*\a);
	\foreach \a in {0.1, 0.03  }
	\draw(-15*\a,14*\a)--(-15*\a,4*\a) arc(-180:-90:\a)--(-4*\a,3*\a) arc(-90:0:\a)--(-3*\a,14*\a)arc(0:90:\a) -- (-14*\a,15*\a) arc(90:180:\a);
	\end{scope}
	%
	\foreach \g in {0.01}
	\draw[red] (-16*\g,14*\g)--(-16*\g,4*\g) arc(-180:-90:2*\g)--(14*\g,2*\g) arc(-90:0:2*\g)--(16*\g,14*\g)arc(0:90:2*\g) -- (-14*\g,16*\g) arc(90:180:2*\g);
	\foreach \g in {0.01}
	\draw[red] (-15*\g,14*\g)--(-15*\g,4*\g) arc(-180:-90:\g)--(14*\g,3*\g) arc(-90:0:\g)--(15*\g,14*\g)arc(0:90:\g) -- (-14*\g,15*\g) arc(90:180:\g);
    \end{tikzpicture}
    \caption{$n=2$}
    \end{subfigure}
\hfill
    \begin{subfigure}{0.45\textwidth}
    \begin{tikzpicture}[scale=1]
        \draw[->] (-3,0) -- (3,0);
        \draw[->] (0,-3) -- (0,3);
	\def \i {0.3}
	\draw(-8*\i,6*\i)--(-8*\i,4*\i) arc(-180:-90:2*\i)--(-4*\i,2*\i) arc(-90:0:2*\i)--(-2*\i,6*\i)arc(0:90:2*\i) -- (-6*\i,8*\i) arc(90:180:2*\i);
	\draw(-7*\i,6*\i)--(-7*\i,4*\i) arc(-180:-90:\i)--(-4*\i,3*\i) arc(-90:0:\i)--(-3*\i,6*\i)arc(0:90:\i) -- (-6*\i,7*\i) arc(90:180:\i);
	\foreach \a in {0.1, 0.03, 0.01 }
	\draw(-16*\a,14*\a)--(-16*\a,4*\a) arc(-180:-90:2*\a)--(-4*\a,2*\a) arc(-90:0:2*\a)--(-2*\a,14*\a)arc(0:90:2*\a) -- (-14*\a,16*\a) arc(90:180:2*\a);
	\foreach \a in {0.1, 0.03, 0.01, 0.003 }
	\draw(-15*\a,14*\a)--(-15*\a,4*\a) arc(-180:-90:\a)--(-4*\a,3*\a) arc(-90:0:\a)--(-3*\a,14*\a)arc(0:90:\a) -- (-14*\a,15*\a) arc(90:180:\a);
	%
	\begin{scope}[xscale=-1]
	\def \i {0.3}
	\draw(-8*\i,6*\i)--(-8*\i,4*\i) arc(-180:-90:2*\i)--(-4*\i,2*\i) arc(-90:0:2*\i)--(-2*\i,6*\i)arc(0:90:2*\i) -- (-6*\i,8*\i) arc(90:180:2*\i);
	\draw(-7*\i,6*\i)--(-7*\i,4*\i) arc(-180:-90:\i)--(-4*\i,3*\i) arc(-90:0:\i)--(-3*\i,6*\i)arc(0:90:\i) -- (-6*\i,7*\i) arc(90:180:\i);
	\foreach \a in {0.1, 0.03, 0.01 }
	\draw(-16*\a,14*\a)--(-16*\a,4*\a) arc(-180:-90:2*\a)--(-4*\a,2*\a) arc(-90:0:2*\a)--(-2*\a,14*\a)arc(0:90:2*\a) -- (-14*\a,16*\a) arc(90:180:2*\a);
	\foreach \a in {0.1, 0.03, 0.01, 0.003 }
	\draw(-15*\a,14*\a)--(-15*\a,4*\a) arc(-180:-90:\a)--(-4*\a,3*\a) arc(-90:0:\a)--(-3*\a,14*\a)arc(0:90:\a) -- (-14*\a,15*\a) arc(90:180:\a);
	\end{scope}
	%
	\begin{scope}[yscale=-1]
	\def \i {0.3}
	\draw(-8*\i,6*\i)--(-8*\i,4*\i) arc(-180:-90:2*\i)--(-4*\i,2*\i) arc(-90:0:2*\i)--(-2*\i,6*\i)arc(0:90:2*\i) -- (-6*\i,8*\i) arc(90:180:2*\i);
	\draw(-7*\i,6*\i)--(-7*\i,4*\i) arc(-180:-90:\i)--(-4*\i,3*\i) arc(-90:0:\i)--(-3*\i,6*\i)arc(0:90:\i) -- (-6*\i,7*\i) arc(90:180:\i);
	\foreach \a in {0.1, 0.03, 0.01  }
	\draw(-16*\a,14*\a)--(-16*\a,4*\a) arc(-180:-90:2*\a)--(-4*\a,2*\a) arc(-90:0:2*\a)--(-2*\a,14*\a)arc(0:90:2*\a) -- (-14*\a,16*\a) arc(90:180:2*\a);
	\foreach \a in {0.1, 0.03, 0.01, 0.003 }
	\draw(-15*\a,14*\a)--(-15*\a,4*\a) arc(-180:-90:\a)--(-4*\a,3*\a) arc(-90:0:\a)--(-3*\a,14*\a)arc(0:90:\a) -- (-14*\a,15*\a) arc(90:180:\a);
	\end{scope}
	%
	\begin{scope}[xscale=-1, yscale=-1]
	\def \i {0.3}
	\draw(-8*\i,6*\i)--(-8*\i,4*\i) arc(-180:-90:2*\i)--(-4*\i,2*\i) arc(-90:0:2*\i)--(-2*\i,6*\i)arc(0:90:2*\i) -- (-6*\i,8*\i) arc(90:180:2*\i);
	\draw(-7*\i,6*\i)--(-7*\i,4*\i) arc(-180:-90:\i)--(-4*\i,3*\i) arc(-90:0:\i)--(-3*\i,6*\i)arc(0:90:\i) -- (-6*\i,7*\i) arc(90:180:\i);
	\foreach \a in {0.1, 0.03, 0.01 }
	\draw(-16*\a,14*\a)--(-16*\a,4*\a) arc(-180:-90:2*\a)--(-4*\a,2*\a) arc(-90:0:2*\a)--(-2*\a,14*\a)arc(0:90:2*\a) -- (-14*\a,16*\a) arc(90:180:2*\a);
	\foreach \a in {0.1, 0.03, 0.01, 0.003 }
	\draw(-15*\a,14*\a)--(-15*\a,4*\a) arc(-180:-90:\a)--(-4*\a,3*\a) arc(-90:0:\a)--(-3*\a,14*\a)arc(0:90:\a) -- (-14*\a,15*\a) arc(90:180:\a);
	\end{scope}
	%
	\foreach \g in {0.003}
	\draw[red] (-15*\g,14*\g)--(-15*\g,-14*\g) arc(-180:-90:\g)--(-4*\g,-15*\g) arc(-90:0:\g)--(-3*\g,14*\g)arc(0:90:\g) -- (-14*\g,15*\g) arc(90:180:\g);
    \end{tikzpicture}
    \caption{$n=3$}
    \end{subfigure}
    \caption{The support of the velocity fields $b_n$ for $n=0,1,2,3$, with the gluing pipe in red.}
    \label{fig:loopstructure}
\end{figure}
{
The aim is to ensure that the same qualitative non-uniqueness of solutions for the advection equation, as described in Figure \ref{fig:nu_trajectories}, also holds true for the advection--diffusion equation. This is made possible by the extremely fast time of approaching the origin, which prevents
$\Delta$ from having enough time to regularize the system, due to the high P\'eclet number of the system. }
{In the limit $n \to \infty$,} some density flowing in from the upper left part could flow out towards the upper right (along the even sequence of velocity fields $\{ b_{2n} \}_{n \in \N}$) or lower left part (along the odd sequence of velocity fields $\{ b_{2n+1} \}_{n \in \N}$), thus yielding non-uniqueness of solutions. This resembles the idea of a very degenerate hyperbolic point.

The proof of \eqref{intro:1} follows straightforwardly from the construction of the velocity field, and it is quite standard to observe such a property. For the proof of \eqref{intro:2}, we require super-exponential separation of scales, ensuring that the smaller scale (smaller pipe) is unaffected by the larger scale (larger pipe). This idea is inspired by the convex integration technique.  We enumerate the pipes in the first quadrant,  $Q_0, Q_1, \ldots, Q_n$, from the  largest pipe $Q_0$ to the smallest $Q_n$ of the approximated velocity field $b_n$ and we define the $v_q = \| b_n \|_{L^\infty (Q_q)}$. We suppose that the width of the pipe $Q_q$ is comparable to length, denoted as $a_q$, with $a_n \to 0$ as $n \to \infty$ so that we have a single critical point, that is in the origin. One can expect that condition \eqref{Peclet} is somehow necessary to have the stability \eqref{intro:2}. Specifically during an advection time interval $\Delta t_q = \frac{a_q}{v_q}$ (time for a particle to travel all around the $q-th$ pipe) a Brownian motion should move less then $a_q$ to preserve  the geometry.  More precisely \eqref{Peclet} can be rewritten as 
\begin{align} \label{eq:intro-peclet}
\sqrt{\Delta t_q } \ll a_q \qquad \Longleftrightarrow \qquad v_q \gg a_q^{-1} \,. 
\end{align}

If  one computes the $L^p$ norm of $b_n$ in the q-th pipe one has
$$\| b_n \|_{L^p (Q_q)}^p  \lesssim a_q^2 v_q^{p} \to 0 \qquad  \Longrightarrow \qquad p< 2\,.$$
This is the constraint on the integrability of our velocity field, and it represents the sharp  integrability achievable to prove non-uniqueness of solutions to the advection--diffusion equation.
\\
\\
The local energy inequality and the regularity of the solutions follow from the fact that we argue using suitable ``regularization'' of the velocity field which provides a uniform control of the approximate weak solutions in $L^2_t H^1_x \cap L^\infty_{t,x}$, which guarantee a weak limit in $L^2_t H^1_x \cap L^\infty_{t,x}$. Similarly as   in the proof of the existence of suitable Leray--Hopf solutions for 3D Navier--Stokes equations, see \cites{leray,LP}, we can prove a local energy inequality thanks to an Aubin--Lions lemma \cites{A63,L69}, see Section \ref{sec:concluding-proof} for details. 

The quantitative stability in equation \eqref{intro:2} is { more achievable} to prove if the stochastic flow and the regular Lagrangian flow transition from ``small scales'' to ``large scales''. {  The reasoning behind this is that it is more feasible to control an error of the same magnitude as the small scale at the large scale. However, an error of the same magnitude as the large scale would represent a significant error at the small scale, unless there is some contraction during the dynamics that causes the error to decrease.}
 
  However in our situation, the forward flow map of the regularized velocity field transitions from large scales to small scales and back to large scales. Therefore, it is important to use both the backward flow (for the part of the proof from large to small scales) see Lemma \ref{lem:loopstabback} and the forward flow (for the part of the proof from small to large scales) see Lemma \ref{lem:loopstabfor}. This is the reason why we also use the other representation formula for such solutions through the pushforward of the forward Stochastic flow $\bY_{0,t}$ and the forward Regular Lagrangian flow $\bX_{0,t}$   respectively, namely 
 \begin{equation}\label{eq:forw_repres}
 	\solAD (t,x)  = \bY_{0,t} (\cdot , \cdot )_{\#}  \ADin \mathcal{L}^2  \otimes \mathbb{P} \qquad \text{and} \qquad \solTE (t,x) = \bX_{0,t} (\cdot )_{\#} \TEin \mathcal{L}^2 \,.
 \end{equation}
 {Notice that for the advection equation, we have two equivalent representation formulas for the solutions with  the backward Lagrangian flow \eqref{eq:back_repres} and the forward Lagrangian flow \eqref{eq:forw_repres} under suitable regularity assumptions on $b$. This is thanks to the fact that the vector field $b$ is divergence--free.}
Passing from the forward representation formula to the backward representation formula is a technical point that is given in Lemma \ref{lem:claims} and allows to combine the backward stability (Lemma \ref{lem:loopstabback}) to the forward stability (Lemma \ref{lem:loopstabfor}).

\section*{Acknowledgements} 
MS has been   supported by  the Swiss State Secretariat for Education, Research and Innovation (SERI) under contract number MB22.00034 during the first part of this work. 
TM is supported by the Swiss National Science Foundation within the scope of the NCCR SwissMAP under the grant number 205607.

\section{Preliminaries}
We will extensively use the description of the advection equation and of the advection--diffusion equation by the Regular Langrangian Flow, respectively the stochastic flow. We present here the main definitions, notations, and results that will be used. {We refer to \cite{DPL89,A04} for some results on regular Lagrangian flows and on the advection equation under low regularity assumptions on the  velocity field.
We also refer to  \cite{K84,OB} for the stochastic part and also to \cite{LL19} for classical results on advection--diffusion equation.}

Given a {divergence-free} vector field $b$ and some initial data $(t_0,x_0) \in [0,T]\times \T^d$, we denote by $t\mapsto \bX_{t_0,t}(x_0)$ the regular Lagrangian flow associated to it by the ordinary differential equation
\begin{equation}\label{eq:ODE}\tag{ODE}
    \begin{cases}
        d \bX_{t_0,t}(x_0) = \bb(t,\bX_{t_0,t}(x_0))dt\,,\\
        \bX_{t_0,t_0}(x_0) = x_0\,,
    \end{cases}
\end{equation}
{defined for $\L^d$-a.e. $x_0 \in \T^d$.}
Consider any divergence-free vector field $\bb \in L^1(0,T; W^{1,q}(\T^d))$ together with some initial data $\TEin \in L^r(\T^d)$ such that $1/q + 1/r\le 1$. Then there exists a weak solution to the advection equation \eqref{eq:TE}.
Moreover it can be represented as the transport of the initial mass along the characteristic lines of the flow:
\begin{equation}\label{eq:TErepresentation}
    \solTE(t,x) = \TEin \big( \bX_{t,t_0}(x)\big) = \TEin\big( \bX_{t_0,t}^{-1}(x)\big)\,, \qquad {\forall t \in [0,T] \text{ and $\L^d$-a.e. } x \in \T^d}.
\end{equation}

Given a probability space $(\Omega,\U,\P)$ and $(x_0,\omega) \in \T^2 \times \Omega$, denote by $t \mapsto \bY_{t_0,t}(x_0,\omega)$ the continuous stochastic process adapted to the filtration $(\F_{s,t_0})_s$ defined by $\F_{s,t_0} = \U(\bW_t : t_0 \le t \le s)$  which solves the \textit{forward} stochastic differential equation
    \begin{equation}\label{eq:SDEforward}
    \begin{cases}
        d \bY_{t_0,t}(x_0,\omega) = \bb(t,\bY_{t_0,t}(x_0,\omega))dt + \sqrt{2}d\bW_t(\omega)\,,\\
        \bY_{t_0,t_0}(x_0,\omega) =x_0\,,
    \end{cases}
    \end{equation}
    which means 
    \begin{equation*}
        \bY_{t_0,t}(x_0,\omega) = x_0 + \int_{t_0}^t \bb(s,\bY_{t_0,s}(x_0,\omega)) ds + \sqrt{2}(\bW_t - \bW_{t_0})\,. 
    \end{equation*}
    In \cite{KR05}, given $b \in L^p(0,T;L^q(\T^d))$ with $p,q \ge 2$ satisfying the Ladyzhenskaya-Prodi-Serrin condition $ 2/p+d/q  < 1$,  Krylov and R\"ockner prove existence and uniqueness {of \eqref{eq:SDEforward}, see also \cite{FF13}. }
    
On the other hand, given some final data $(\bar{t},\bar{x}) \in [0,T] \times \T^d$ denote by $t \mapsto \bY_{\bar{t},t}(\bar{x},\omega)$ the continuous stochastic process adapted to the filtration $(\tilde{\F}_{s,\bar{t}})_s$ defined by $\tilde{\F}_{s,t} = \U(\bW_t -\bW_{\bar{t}}:s \le t \le \bar{t})$  which solves the \textit{backward} stochastic differential equation
    \begin{equation}\label{eq:SDEbackward}
    \begin{cases}
        d \bY_{\bar{t},t}(\bar{x},\omega) = -\bb(t,\bY_{\bar{t},t}(\bar{x},\omega))dt - \sqrt{2}d\bW_t(\omega)\,,\\
        \bY_{\bar{t},\bar{t}}(\bar{x},\omega) = \bar{x}\,,
    \end{cases}
    \end{equation}
    which means 
    \begin{equation*}
        \bY_{\bar{t},t}(\bar{x},\omega) = \bar{x} - \int_t^{\bar{t}} \bb(s,\bY_{\bar{t},s}(\bar{x},\omega)) ds - \sqrt{2}(\bW_{\bar{t}} - \bW_t) \,.
    \end{equation*}
    It is important to distinguish between the forward and the backward stochastic flow as for $t_1,t_2 \in [0,T]$, $\bY_{t_1,t_2}\neq \bY_{t_2,t_1}^{-1}$. Notice that if $0 \le t_1 < t_2 \le T,$ $\bY_{t_1,t_2}$ denotes the forward flow, but if $0 \le t_2 < t_1 \le T,$ $\bY_{t_1,t_2}$ denotes the backward flow.

    The solution to the \eqref{eq:ADE}, given a {divergence-free} velocity field $\bb \in L^2(0,T;L^{{q}}(\T^d))$ and initial datum $\ADin \in {W^{2,\infty}}(\T^d)$ with ${2/p+d/q < 1}$, can be represented by the backward Feynman-Kac Formula
\begin{equation}\label{eq:FKbackward}
    \solAD(\bar{t},x) = \E \big[ \ADin\big( \bY_{\bar{t},t_0}(x,\omega) \big) \big]\,, \qquad { \forall \bar{t} \in [t_0,T] \text{ and $\L^d$-a.e. } x \in \T^d\,.}
\end{equation}
We can also work with the forward stochastic flow : for any $f \in L^\infty(\T^d),$
\begin{equation}\label{eq:FKforward}
    \int_{\T^d} f(x) \solAD(t,x) dx = \int_\Omega \int_{\T^d} f\big(\bY_{t_0,t}(x,\omega)\big) \ADin(x) dx d\P(\omega)\,, \qquad { \forall t \in [t_0,T] \text{ and $\L^d$-a.e. } x \in \T^d\,.}
\end{equation}
Applying this to $\solAD \equiv 1 \equiv \ADin$ and $f(x) = \one_A$ for some Borel set $A\subseteq \T^d$, we find out that the stochastic flow is measure preserving in an averaged sense :
\begin{equation}
    \int_\Omega \int_{\T^d} \one_A\big(\bY_{t_0,t}(x,\omega)\big) dx d\P(\omega) = \L^d(A) \qquad \forall t \ge 0\,.
\end{equation}

To compare the regular Lagrangian flow and the stochastic flow, we will use two tools
\begin{itemize}
    \item The classical Gr\"onwall lemma: for two flows associated to the same velocity field
    \begin{equation}\label{eq:Gronwall}
        |\bX_{t_0,t}(x) - \bY_{t_0,t}(y,\omega)| \le \left(|x-y| + \sqrt{2}\sup_{s\in [t_0,t]}|\bW_s(\omega)| \right) \exp\left(\int_{t_0}^t \|\nabla b(s,\cdot)\|_{L^\infty(\T^d)}ds\right)
    \end{equation}
    for every $t \ge t_0$\,.
    \item Doob's inequality to bound the {d-dimensional} Brownian motion for any $a< b$ 
   \begin{equation}\label{eq:translatedDoob}
        \dP\Big( \sup_{t\in[a,b]} |\bW_t-\bW_a|\ge C \Big) \le {d} \exp \left[-\frac{C^2}{2{d}(b-a)}\right].
   \end{equation}
\end{itemize}

{Finally, we will use the following notation to denote the restriction of a set $S \subset \T^d$:
	\begin{equation}\label{eq:restriction}
		S[\varepsilon] := \{ x \in \T^d \mid \text{d}(x, S^c) > \varepsilon\}.
	\end{equation}
}

\section{The velocity field in $L^\infty_t L^p_x$}

{
\subsection{Main ideas of the construction} \label{subsec:idea-loop}

This Subsection provides a heuristic description of the construction of the velocity field used in Theorem \ref{thm:NU_ADE_loops}, the crucial properties that this velocity field must satisfy and the technically challenging aspects of the  proof of Theorem \ref{thm:NU_ADE_loops}.

The main idea of the construction is to provide a velocity field $b$ with only one singularity at the origin $x=y=0$, where the bifurcation (i.e. the non-uniqueness) of Lagrangian trajectories occurs. Additionally, the velocity field is designed so that trajectories exist and are unique everywhere except at the origin. Clearly,  many examples of such divergence-free velocity fields exist. For instance, we recall an  example of a velocity field provided by Aizenmann \cite{A78} defined  in polar coordinates  with only radial component 
\begin{align} \label{eq:failvelocity}
b (r, \theta)= b_r (r, \theta) \textbf{e}_r \,,
\end{align}
where  for some $\varepsilon >0$ we have
\begin{align*}
b_r (r, \theta) =
\begin{cases}
1/r &  \text{ if } \theta \in [-\varepsilon + \pi/4, \pi/4 + \varepsilon) \cup [-\varepsilon + \pi+ \pi/4, \pi + \pi/4 + \varepsilon) \,,
\\
-1/r & \text{ if } \theta \in [ - \varepsilon + 3/4 \pi, 3/4 \pi + \varepsilon) \cup [-\varepsilon - \pi/4, - \pi/4 + \varepsilon) \,,
\\
0 & \text{ otherwise}\,,
\end{cases}
\end{align*}
see Figure \ref{fig:velocity-fail}.
\begin{figure}[ht]
    \begin{tikzpicture}[scale=1.5]
        \draw (0,0) rectangle (2,2);
        \draw (0,1)--(2,1);
        \draw (1,0)--(1,2);
        \draw (0,1.8)--(2,0.2);
        \draw (0.2,2)--(1.8,0);
        \draw (0.2,0)--(1.8,2);
        \draw (0,0.2)--(2,1.8);
        \draw node at (0,-0.22) {};
        \draw[->] (0.2,1.8)--(0.6,1.4);
        \draw[->] (0.6,0.6)--(0.2,0.2);
        \draw[->] (1.4,1.4)--(1.8,1.8);
        \draw[->] (1.8,0.2)--(1.4,0.6);
    \end{tikzpicture}
    \caption{The velocity field in \eqref{eq:failvelocity}.} \label{fig:velocity-fail}
\end{figure}
 This velocity field is divergence-free in the sense of distributions, thanks to its symmetry. However, the crucial property required for comparing the regular Lagrangian flow and the stochastic Lagrangian flow is not satisfied by this velocity field. This property is the high-P\'eclet number of the system due to the velocity field $b$ described in Subsection \ref{subsec:heuristic-idea}.  We describe this essential property in a more quantitative way.  There exists a set $R_0 \subset \T^2$ with positive measure such that  for any $(x_0, y_0) \in R_0$ 
$$ \Delta t (x_0, y_0) =\inf \{ t  \in (0,1):   X_{t} (x_0, y_0) =(0,0) \} \,,$$
is such that $\sqrt{ \Delta t (x_0, y_0) } \ll |(x_0, y_0)|  $, where by this we mean that 
\begin{align} \label{eq:heuristic-main}
\lim_{(x_0, y_0) \to (0,0)} \frac{\sqrt{ \Delta t (x_0, y_0) }}{ {|(x_0, y_0)|}} =0 \,. 
\end{align}
This heuristic property implies that an error like $\sqrt{t}$ (which is the square root of the variance of the Brownian motion) does not affect the behaviour of the Lagrangian trajectories. This is the main property that allows us to prove that 
 advection equation solution for a suitable regularized version of this velocity field  are qualitatively the same as the advection--diffusion equation solutions. Unfortunately, in the case of the velocity field $b$ defined in \eqref{eq:failvelocity} we have that  $\frac{\sqrt{\Delta t(x_0, y_0)} }{ {|(x_0, y_0)|}} \geq \frac{|x_0 - y_0|}{\sqrt{2}}$. Indeed, the velocity field $b$ is in the weak-$L^2$ space, which is a critical space for uniqueness of bounded solutions to the advection--diffusion equation as discussed in the introduction and it is not clear to  us if for such velocity field uniqueness of bounded solutions holds true. An approach to define a divergence-free velocity field 
	$b \in L^\infty_t L^p_x$, with an arbitrary but fixed 
$p \in [1,2)$, that satisfies the property described above is to overlap several divergence-free, loop-like vector fields 
$w_q$
  at different scales and magnitudes, see \eqref{eq:intro-peclet}.
These vector fields form chains, with each loop centered along the diagonals of the torus. Each loop-like vector field
$w_q$	
  interacts\footnote{More precisely, its support intersects only the supports of $w_{q+1}$ and $w_{q-1}$.}  only with 
$w_{q-1}$
and 
$w_{q+1}$, as illustrated in Figure \ref{fig:loopstructure}. The corresponding Lagrangian trajectories are then driven toward the origin at an arbitrary speed, 
which can be freely chosen and depends only on the magnitude of the velocity fields 
$w_q$.

 We now define a sequence of \emph{regular} velocity fields $\{ b_n \}_n$ that converge to the velocity field $b$ used in Theorem \ref{thm:NU_ADE_loops}. Each $b_n$ is constructed as the sum of $4n$ loop-like vector fields, $n$ in each quadrant (see Figure \ref{fig:loopstructure}).  
By appropriately choosing a gluing velocity field (see Figures \ref{fig:odd-even} and \ref{fig:nu_trajectories_bis} for details), the trajectories of $b_n$ originating in the second quadrant travel to the first quadrant when $n$ is even and to the third quadrant when $n$ is odd. In particular, for $n$ even, all Lagrangian trajectories of $b_n$ starting from a suitable subset $R_0$ will reach the first quadrant at some time $T > 0$, whereas for $n$ odd, they will all end in the third quadrant. This behaviour implies the non-uniqueness of solutions to the advection equation (see Figure \ref{fig:nu_trajectories}).  
We can then extend this result to prove the non-uniqueness of solutions to the advection-diffusion equation, relying on the crucial property \eqref{eq:heuristic-main} of the constructed velocity field. However, some technical challenges arise in the proof.


\begin{itemize}
	\item (Intersection problem).
	At each intersection of the supports of $w_{q-1}$ and $w_q$, the forward Lagrangian trajectories will concentrate on one side of the pipe, as shown in Figure \ref{fig:intersection}, because the magnitude of the velocity $w_q$ is much greater than that of $w_{q-1}$ in order to impose property \eqref{eq:heuristic-main}.  
Due to small errors introduced by the Brownian motion (resulting from the $\Delta$ term in the equation), most stochastic Lagrangian trajectories may exit the support of the velocity field $w_q$, significantly complicating the analysis. To address this technical issue, we introduce a time intermittency in the velocity fields $w_q$ relative to $w_{q-1}$. More precisely, we ensure that $w_q$ is periodically turned on and off in such a way that most trajectories remain centered within the support of $w_q$ rather than accumulating near the boundary, as illustrated in Figure \ref{fig:intersection}.

	\item (Forward-backward stability). 
	We can obtain some stability results between the Stochastic Lagrangian flow and the regular Lagrangian flow (see Lemma \ref{lem:loopstabfor} and Lemma \ref{lem:loopstabback}). However, these results hold true if these flows start from a point $x$ in the small scale and travel towards large scales (see Remark \ref{rmk:large_small} for more details). More precisely, to prove such lemmas we need that the flows start from the support of the smaller loop-like velocity field and travel toward the support of the larger loop-like velocity field, see Figure \ref{fig:loopstructure}.  
Technically, this means that we need to use both the forward Lagrangian trajectories and the backward Lagrangian trajectories with the \emph{regularized} velocity field $b_n$ with $n \in \N$. Indeed, the evolution of the solution goes from large scales (in the second quadrant) to small scales (close to the origin), and then it returns to large scales again (in the first or third quadrant, respectively for $n$ even and $n$ odd), see Figure \ref{fig:Ssets}.
In Lemma \ref{lem:claims}, we use the forward stability from Lemma \ref{lem:loopstabfor} and the backward stability from Lemma \ref{lem:loopstabback} to conclude distinct properties of the solutions to the advection-diffusion equation along the odd and even numbers of the sequence $n \in \mathbb{N}$. This lemma will lead to the non-uniqueness of solutions to the advection-diffusion equation.

\end{itemize}

\begin{figure}[!htpb]
	\begin{tikzpicture}[scale=1]
		\draw (0,0)--(2,0);
		\draw (0,1)--(2,1);
		\draw (1,-1)--(1,1.5);.
		\draw (1.4,-1)--(1.4,1.5);
		\draw[blue] (0,0.25)--(1,0.25)--(1.05,0)--(1.05,-1);
		\draw[blue] (0,0.5)--(1,0.5)--(1.1,0)--(1.1,-1);
		\draw[blue] (0,0.75)--(1,0.75)--(1.15,0)--(1.15,-1);
	\end{tikzpicture}
	\caption{{ The Lagrangian trajectories behaviour going from the support of $w_{q-1}$ to the support of $w_q$ if they were time independent. We will avoid the possibility of the trajectories being pushed close to the boundary by introducing suitable time dependency of $\{w_q\}_q$ that we call the time intermittency. }}
	\label{fig:intersection}
\end{figure}

\begin{figure}[!htpb]
	\newcommand{\loopsAbcd}{
		\draw (-4,0) -- (-1,0) arc(-90:0:1) -- (0,1.5);
		\draw (-4,-1) -- (-1,-1) arc(-90:0:2) -- (1,1.5);
		\draw[gray!40] (-2,1.5) -- (-2,-1.5) arc(180:270:0.5) -- (1.5,-2) arc(-90:0:0.5) -- (2,1.5);
		\draw[gray!40] (-2.5,1.5) -- (-2.5,-1.5) arc(180:270:1) -- (1.5,-2.5) arc(-90:0:1) -- (2.5,1.5);
		\draw[gray!40] (1,-1.2) -- (1, -2.8) arc(180:270:0.2) -- (2.8,-3) arc(-90:0:0.2) -- (3,-1.2) arc(0:90:0.2) -- (1.2,-1) arc(90:180:0.2);
		\draw[gray!40] (0.8,-1.2) -- (0.8, -2.8) arc(180:270:0.4) -- (2.8,-3.2) arc(-90:0:0.4) -- (3.2,-1.2) arc(0:90:0.4) -- (1.2,-0.8) arc(90:180:0.4);
		\draw[gray!40] (2.6,-2.7)--(2.6,-3.3)arc(180:270:0.1) -- (3.3,-3.4) arc(-90:0:0.1) -- (3.4,-2.7) arc(0:90:0.1) -- (2.7,-2.6) arc(90:180:0.1);
		\draw[gray!40] (2.5,-2.7) -- (2.5, -3.3) arc(180:270:0.2) -- (3.3,-3.5) arc(-90:0:0.2) -- (3.5,-2.7) arc(0:90:0.2) -- (2.7,-2.5) arc(90:180:0.2);
	}
	\newcommand{\labelsAbcd}{
		\draw[->] (-1.5,-0.5) -- (-1.2,-0.5);
		\draw (-1.3,0.4) node {$ Q_0 $};
		\draw[gray!40] (-2.5,-2.5) node {$ Q_1 $};
		\draw[gray!40] (0.3,-3) node {$ Q_2 $};
		\draw[gray!40] (3,-2) node {$ Q_3 $};
	}
	\newcommand{\loopsABcd}{
		\draw (-4,0) -- (-1,0) arc(-90:0:1) -- (0,1.5);
		\draw (-4,-1) -- (-1,-1) arc(-90:0:2) -- (1,1.5);
		\draw[cyan] (-2,1.5) -- (-2,-1.5) arc(180:270:0.5) -- (1.5,-2) arc(-90:0:0.5) -- (2,1.5);
		\draw[cyan] (-2.5,1.5) -- (-2.5,-1.5) arc(180:270:1) -- (1.5,-2.5) arc(-90:0:1) -- (2.5,1.5);
		\draw[gray!40] (1,-1.2) -- (1, -2.8) arc(180:270:0.2) -- (2.8,-3) arc(-90:0:0.2) -- (3,-1.2) arc(0:90:0.2) -- (1.2,-1) arc(90:180:0.2);
		\draw[gray!40] (0.8,-1.2) -- (0.8, -2.8) arc(180:270:0.4) -- (2.8,-3.2) arc(-90:0:0.4) -- (3.2,-1.2) arc(0:90:0.4) -- (1.2,-0.8) arc(90:180:0.4);
		\draw[gray!40] (2.6,-2.7)--(2.6,-3.3)arc(180:270:0.1) -- (3.3,-3.4) arc(-90:0:0.1) -- (3.4,-2.7) arc(0:90:0.1) -- (2.7,-2.6) arc(90:180:0.1);
		\draw[gray!40] (2.5,-2.7) -- (2.5, -3.3) arc(180:270:0.2) -- (3.3,-3.5) arc(-90:0:0.2) -- (3.5,-2.7) arc(0:90:0.2) -- (2.7,-2.5) arc(90:180:0.2);
	}
	\newcommand{\labelsABcd}{
		\draw[->] (-1.5,-0.5) -- (-1.2,-0.5);
		\draw[cyan,->] (-1.5,-2.25) -- (-0.8,-2.25);
		\draw (-1.3,0.4) node {$ Q_0 $};
		\draw[cyan] (-2.5,-2.5) node {$ Q_1 $};
		\draw[gray!40] (0.2,-3) node {$ Q_2 $};
		\draw[gray!40] (3,-2) node {$ Q_3 $};
	}
	\newcommand{\loopsABCd}{
		\draw (-4,0) -- (-1,0) arc(-90:0:1) -- (0,1.5);
		\draw (-4,-1) -- (-1,-1) arc(-90:0:2) -- (1,1.5);
		\draw[cyan] (-2,1.5) -- (-2,-1.5) arc(180:270:0.5) -- (1.5,-2) arc(-90:0:0.5) -- (2,1.5);
		\draw[cyan] (-2.5,1.5) -- (-2.5,-1.5) arc(180:270:1) -- (1.5,-2.5) arc(-90:0:1) -- (2.5,1.5);
		\draw[red] (1,-1.2) -- (1, -2.8) arc(180:270:0.2) -- (2.8,-3) arc(-90:0:0.2) -- (3,-1.2) arc(0:90:0.2) -- (1.2,-1) arc(90:180:0.2);
		\draw[red] (0.8,-1.2) -- (0.8, -2.8) arc(180:270:0.4) -- (2.8,-3.2) arc(-90:0:0.4) -- (3.2,-1.2) arc(0:90:0.4) -- (1.2,-0.8) arc(90:180:0.4);
		\draw[gray!40] (2.6,-2.7)--(2.6,-3.3)arc(180:270:0.1) -- (3.3,-3.4) arc(-90:0:0.1) -- (3.4,-2.7) arc(0:90:0.1) -- (2.7,-2.6) arc(90:180:0.1);
		\draw[gray!40] (2.5,-2.7) -- (2.5, -3.3) arc(180:270:0.2) -- (3.3,-3.5) arc(-90:0:0.2) -- (3.5,-2.7) arc(0:90:0.2) -- (2.7,-2.5) arc(90:180:0.2);
	}
	\newcommand{\labelsABCd}{
		\draw[->] (-1.5,-0.5) -- (-1.2,-0.5);
		\draw[cyan,->] (-1.5,-2.25) -- (-0.8,-2.25);
		\draw[red,->] (1.5,-3.1) -- (2.5,-3.1);
		\draw (-1.3,0.4) node {$ Q_0 $};
		\draw[cyan] (-2.5,-2.5) node {$ Q_1 $};
		\draw[red] (0.2,-3) node {$ Q_2 $};
		\draw[gray!40] (3,-2) node {$ Q_3 $};
	}
	\newcommand{\loopsABCD}{
		\draw (-4,0) -- (-1,0) arc(-90:0:1) -- (0,1.5);
		\draw (-4,-1) -- (-1,-1) arc(-90:0:2) -- (1,1.5);
		\draw[cyan] (-2,1.5) -- (-2,-1.5) arc(180:270:0.5) -- (1.5,-2) arc(-90:0:0.5) -- (2,1.5);
		\draw[cyan] (-2.5,1.5) -- (-2.5,-1.5) arc(180:270:1) -- (1.5,-2.5) arc(-90:0:1) -- (2.5,1.5);
		\draw[red] (1,-1.2) -- (1, -2.8) arc(180:270:0.2) -- (2.8,-3) arc(-90:0:0.2) -- (3,-1.2) arc(0:90:0.2) -- (1.2,-1) arc(90:180:0.2);
		\draw[red] (0.8,-1.2) -- (0.8, -2.8) arc(180:270:0.4) -- (2.8,-3.2) arc(-90:0:0.4) -- (3.2,-1.2) arc(0:90:0.4) -- (1.2,-0.8) arc(90:180:0.4);
		\draw[blue] (2.6,-2.7)--(2.6,-3.3)arc(180:270:0.1) -- (3.3,-3.4) arc(-90:0:0.1) -- (3.4,-2.7) arc(0:90:0.1) -- (2.7,-2.6) arc(90:180:0.1);
		\draw[blue] (2.5,-2.7) -- (2.5, -3.3) arc(180:270:0.2) -- (3.3,-3.5) arc(-90:0:0.2) -- (3.5,-2.7) arc(0:90:0.2) -- (2.7,-2.5) arc(90:180:0.2);
	}
	\newcommand{\labelsABCD}{
		\draw[->] (-1.5,-0.5) -- (-1.2,-0.5);
		\draw[cyan,->] (-1.5,-2.25) -- (-0.8,-2.25);
		\draw[red,->] (1.5,-3.1) -- (2.5,-3.1);
		\draw (-1.3,0.4) node {$ Q_0 $};
		\draw[cyan] (-2.5,-2.5) node {$ Q_1 $};
		\draw[red] (0.4,-3) node {$ Q_2 $};
		\draw[blue] (3,-2) node {$ Q_3 $};
	}
	\begin{tikzpicture}[scale=0.5]
		\loopsAbcd
		\labelsAbcd
		\filldraw[fill=blue!30!white, fill opacity=0.5, draw=blue] (-3.3,-1) -- (-2.3,-1) -- (-2.3,0) -- (-3.3,0) --(-3.3,-1);
	\end{tikzpicture}
	\hfill
	\begin{tikzpicture}[scale=0.5]
		\loopsABcd
		\labelsABcd
		\filldraw[fill=blue!30!white, fill opacity=0.5, draw=blue] (-3,-1) -- (-2.5,-1) -- (-2.5,0) -- (-3,0) --(-3,-1);
		\filldraw[fill=blue!30!white, fill opacity=0.5, draw=blue] (-2.5,-1.3) -- (-2,-1.3) -- (-2,-0.3) -- (-2.5,-0.3) --(-2.5,-1.3);
	\end{tikzpicture}
	\hfill
	\begin{tikzpicture}[scale=0.5]
		\loopsABcd
		\labelsABcd
		\filldraw[fill=blue!30!white, fill opacity=0.5, draw=blue] (-3,-1) -- (-2.5,-1) -- (-2.5,0) -- (-3,0) --(-3,-1);
		\filldraw[fill=blue!30!white, fill opacity=0.5, draw=blue] (0.8,-2) -- (0.45,-2.5) -- (1,-2.5) -- (1,-2) --(0.8,-2);
	\end{tikzpicture}
	\hfill
	\begin{tikzpicture}[scale=0.5]
		\loopsAbcd
		\labelsAbcd
		\filldraw[fill=blue!30!white, fill opacity=0.5, draw=blue] (-2.75,-1) -- (-2.25,-1) -- (-2.25,0) -- (-2.75,0) --(-2.75,-1);
	\end{tikzpicture}
	\begin{tikzpicture}[scale = 1.2]
		\tikzset{line width=0.4pt}
		\draw (-1,0) node{$ $};
		\foreach \x in {0,2.5,5,7.5}{
			\draw (\x cm,2pt) -- (\x cm,-2pt);
			\draw[cyan] (\x, 0) -- (\x + 0.2,0);
			\draw (\x + 0.2, 0) -- (\x + 2.5, 0);
		}
		\draw[dotted] (10,0)--(11,0);
		\draw (0,0) node[below=3pt] {$ \bar{t}_{0,(0)} = 0 $} node[above=3pt] {$   $};
		\draw (2.5,0) node[below=3pt] {$ \bar{t}_{0,(1)} $} node[above=3pt] {$  $};
		\draw (5,0) node[below=3pt] {$ \bar{t}_{0,(2)} $} node[above=3pt] {$  $};
		\draw (7.5,0) node[below=3pt] {$ \bar{t}_{0,(4)} $} node[above=3pt] {$  $};
	\end{tikzpicture}
	\begin{tikzpicture}[scale=0.5]
		\loopsABcd
		\labelsABcd
		\filldraw[fill=blue!30!white, fill opacity=0.5, draw=blue] (-3,-1) -- (-2.5,-1) -- (-2.5,0) -- (-3,0) --(-3,-1);
		\filldraw[fill=blue!30!white, fill opacity=0.5, draw=blue] (0.2,-2) -- (-0.15,-2.5) -- (0.9,-2.5) -- (0.9,-2) --(0.2,-2);
	\end{tikzpicture}
	\hfill
	\begin{tikzpicture}[scale=0.5]
		\loopsABCd
		\labelsABCd
		\filldraw[fill=blue!30!white, fill opacity=0.5, draw=blue] (-3,-1) -- (-2.5,-1) -- (-2.5,0) -- (-3,0) --(-3,-1);
		\filldraw[fill=blue!30!white, fill opacity=0.5, draw=blue] (0.35,-2) -- (0,-2.5) -- (0.8,-2.5) -- (0.8,-2) --(0.35,-2);
		\filldraw[fill=blue!30!white, fill opacity=0.5, draw=blue] (0.8,-2.8)--(1,-2.8)--(1,-2.3)--(0.8,-2.3)--(0.8,-2.8);
	\end{tikzpicture}
	\hfill
	\begin{tikzpicture}[scale=0.5]
		\loopsABCd
		\labelsABCd
		\filldraw[fill=blue!30!white, fill opacity=0.5, draw=blue] (-3,-1) -- (-2.5,-1) -- (-2.5,0) -- (-3,0) --(-3,-1);
		\filldraw[fill=blue!30!white, fill opacity=0.5, draw=blue] (0.35,-2) -- (0,-2.5) -- (0.8,-2.5) -- (0.8,-2) --(0.35,-2);
		\filldraw[fill=blue!30!white, fill opacity=0.5, draw=blue] (2.3,-3.2)--(2.6,-3.2)--(2.6,-3)--(2.5,-3)--(2.3,-3.2);
	\end{tikzpicture}
	\hfill
	\begin{tikzpicture}[scale=0.5]
		\loopsABcd
		\labelsABcd
		\filldraw[fill=blue!30!white, fill opacity=0.5, draw=blue] (-3,-1) -- (-2.5,-1) -- (-2.5,0) -- (-3,0) --(-3,-1);
		\filldraw[fill=blue!30!white, fill opacity=0.5, draw=blue] (0.45,-2) -- (0.1,-2.5) -- (0.9,-2.5) -- (0.9,-2) --(0.45,-2);
	\end{tikzpicture}
	\begin{tikzpicture}[scale = 1.2]
		\tikzset{line width=1pt}
		\draw (-1,0) --(0,0);
		\draw[semithick, cyan] (0,0)--(3,0);
		\draw (4.1,0)--(7,0);
		\draw[dotted] (7,0) -- (8,0);
		\draw (8,0) -- (11,0);
		\foreach \x in {0,10}{
			\draw (\x cm,3pt) -- (\x cm,-3pt);}
		\foreach \x in {3,3.4,3.8,4.2,4.6,5}{
			\draw (\x,1pt) -- (\x,-1pt);
			\draw[red] (\x, 0)--(\x + 0.1,0);
			\draw[cyan] (\x + 0.1,0)--(\x+0.4,0);
		}
		\draw (0,0) node[below=3pt] {$ \bar{t}_{0,(1)} $} node[above=3pt] {$   $};
		\draw (10,0) node[below=3pt] {$ \bar{t}_{0,(2)} $} node[above=3pt] {$  $};
		\draw (4,0) node[below=3pt] {$\bar{t}_{1,(1,0)} \, ... \, \bar{t}_{1,(1,\frac{a_0}{a_2})}$} node[above=3pt] {$  $};
	\end{tikzpicture}
	\begin{tikzpicture}[scale=0.5]
		\loopsABCd
		\labelsABCd
		\filldraw[fill=blue!30!white, fill opacity=0.5, draw=blue] (-3,-1) -- (-2.5,-1) -- (-2.5,0) -- (-3,0) --(-3,-1);
		\filldraw[fill=blue!30!white, fill opacity=0.5, draw=blue] (0.35,-2) -- (0,-2.5) -- (0.8,-2.5) -- (0.8,-2) --(0.35,-2);
		\filldraw[fill=blue!30!white, fill opacity=0.5, draw=blue] (2,-3.2)--(2.55,-3.2)--(2.55,-3)--(2.2,-3)--(2,-3.2);
	\end{tikzpicture}
	\hfill
	\begin{tikzpicture}[scale=0.5]
		\loopsABCD
		\labelsABCD
		\filldraw[fill=blue!30!white, fill opacity=0.5, draw=blue] (-3,-1) -- (-2.5,-1) -- (-2.5,0) -- (-3,0) --(-3,-1);
		\filldraw[fill=blue!30!white, fill opacity=0.5, draw=blue] (0.35,-2) -- (0,-2.5) -- (0.8,-2.5) -- (0.8,-2) --(0.35,-2);		
		\filldraw[fill=blue!30!white, fill opacity=0.5, draw=blue] (2.05,-3.2)--(2.5,-3.2)--(2.5,-3)--(2.25,-3)--(2.05,-3.2);
		\filldraw[fill=blue!30!white, fill opacity=0.5, draw=blue] (2.5, -3.3)--(2.6, -3.3)--(2.6,-3.1)--(2.5,-3.1)--(2.5,-3.3);
	\end{tikzpicture}
	\hfill
	\begin{tikzpicture}[scale=0.5]
		\loopsABCD
		\labelsABCD
		\filldraw[fill=blue!30!white, fill opacity=0.5, draw=blue] (-3,-1) -- (-2.5,-1) -- (-2.5,0) -- (-3,0) --(-3,-1);
		\filldraw[fill=blue!30!white, fill opacity=0.5, draw=blue] (0.35,-2) -- (0,-2.5) -- (0.8,-2.5) -- (0.8,-2) --(0.35,-2);
		\filldraw[fill=blue!30!white, fill opacity=0.5, draw=blue] (2.05,-3.2)--(2.5,-3.2)--(2.5,-3)--(2.25,-3)--(2.05,-3.2);
		\filldraw[fill=blue!30!white, fill opacity=0.5, draw=blue] (3.05,-3.5)--(3.2,-3.5)--(3.2,-3.4)--(3.15,-3.4)--(3.05,-3.5);
	\end{tikzpicture}
	\hfill
	\begin{tikzpicture}[scale=0.5]
		\loopsABCd
		\labelsABCd
		\filldraw[fill=blue!30!white, fill opacity=0.5, draw=blue] (-3,-1) -- (-2.5,-1) -- (-2.5,0) -- (-3,0) --(-3,-1);
		\filldraw[fill=blue!30!white, fill opacity=0.5, draw=blue] (0.35,-2) -- (0,-2.5) -- (0.8,-2.5) -- (0.8,-2) --(0.35,-2);
	\end{tikzpicture}
	\begin{tikzpicture}[scale = 1.2]
		\tikzset{line width=1pt}
		\draw[cyan] (-1,0) --(0,0);
		\draw[red] (0,0)--(3,0);
		\draw[cyan] (4.1,0)--(7,0);
		\draw[cyan, dotted] (7,0) -- (8,0);
		\draw[cyan] (8,0) -- (11,0);
		\foreach \x in {0,10}{
			\draw (\x cm,3pt) -- (\x cm,-3pt);}
		\foreach \x in {3,3.4,3.8,4.2,4.6,5}{
			\draw (\x,1pt) -- (\x,-1pt);
			\draw[blue] (\x, 0)--(\x + 0.1,0);
			\draw[red] (\x + 0.1,0)--(\x+0.4,0);
		}
		\draw (0,0) node[below=3pt] {$ \bar{t}_{1,(1,1)} $} node[above=3pt] {$   $};
		\draw (10,0) node[below=3pt] {$ \bar{t}_{1,(2,1)} $} node[above=3pt] {$  $};
		\draw (4,0) node[below=3pt] {$\bar{t}_{2,(1,1,0)} \, ... \, \bar{t}_{2,(1,1,\frac{a_1}{a_3})}$} node[above=3pt] {$  $};
	\end{tikzpicture}
	\caption{{ Time--intermittency of the pipes: the pipes shown in gray are ``turned off'', while the coloured pipes (black, light blue, red, and blue) are ``turned on''. Additionally, we represent the time intervals during which the coloured velocity fields are active on the timeline below the figures, with the corresponding time supports indicated by the coloured subintervals. In particular, smaller pipes are turned on and off multiple times within a subinterval, while a larger pipe remains continuously turned on during this period.
	}}
	\label{fig:stop-motion}
\end{figure}}

\subsection{Building block}

We start with the construction of a general pipe and fix precisely all parameters in the next section
\begin{itemize}
    \item $a $ denotes the width of the pipe where the flow circulates.
    \item $l$ denotes the inner horizontal length of the pipe.
    \item $\bar{l}$ denotes the inner vertical length of the pipe.
    \item $v$ denotes the speed of the flow.
    \item $O= (O_{x},O_{y}) \in \T^2$ give the coordinates of the center of the loop.
\end{itemize}
These parameters together with the building block are indicated in figure \ref{fig:shearloops}.

\begin{figure}[!htbp]
    \centering
    \begin{tikzpicture}
        \filldraw[fill=yellow, fill opacity=0.2, draw=black](-2.6,-1)--(-2.6,1) arc (180:90:0.6) --(1,1.6) arc (90:0:0.6)--(1.6,-1)arc(0:-90:0.6)--(-2,-1.6) arc (-90:-180:0.6);
        \filldraw[fill=blue!50!white, fill opacity=0.7, draw=black](-2.3,-1)--(-2.3,1) arc (180:90:0.3) --(1,1.3) arc (90:0:0.3)--(1.3,-1)arc(0:-90:0.3)--(-2,-1.3) arc (-90:-180:0.3);
        \filldraw[fill=blue!20!white, fill opacity=1, draw=black](-2,-1)--(-2,1)--(1,1)--(1,-1)--(-2,-1);
        \filldraw[black] (-0.5,0) circle (1pt) node[anchor=south]{\small $O$};
        \coordinate (o) at (0,0);
        \coordinate (a) at (0,1);
        \coordinate (b) at (-1,0);
        \coordinate (f) at (1,0);
        \draw[<->] (-1.7,-1.3) -- (-1.7,-1) node[midway, right]{\small $a$};
        \draw[<->] (-1.7,-1.6) -- (-1.7,-1.3) node[midway, right]{\small $a$};
        \draw[<->] (0.3,-1.3)--(0.3,1.3) node[midway, right]{\small $\bar{l}$};
        \draw[<->] (-2.3,-0.7)--(1.3,-0.7) node[midway, above]{\small $l$};
        \draw node at (-2.45,0.5) {\small $v$};
        \draw[dashed,->] (-2.45,0.3)--(-2.45,-0.5);
        \draw[dashed,->] (0.7,-1.45)--(1,-1.45) arc(-90:0:0.45)--(1.45,-0.7);
        \draw[dashed,->] (1.45,0.7)--(1.45,1) arc(0:90:0.45)--(0.7,1.45);
        \filldraw[fill=yellow, fill opacity=0.2, draw=black] (1.9,0.6)--(2.2,0.6)--(2.2,0.3)--(1.9,0.3)--(1.9,0.6);
        \draw node at (2.8,0.45) {Pipe};
        \filldraw[fill=blue!50!white, fill opacity=0.7, draw=black] (1.9,0.15)--(2.2,0.15)--(2.2,-0.15)--(1.9,-0.15)--(1.9,0.15);
        \filldraw[fill=blue!20!white, fill opacity=1, draw=black] (2.2,0.15)--(2.5,0.15)--(2.5,-0.15)--(2.2,-0.15)--(2.2,0.15);
        \draw node at (2.9,0) {$R_a$};
        \filldraw[fill=blue!20!white, fill opacity=1, draw=black] (1.9,-0.6)--(2.2,-0.6)--(2.2,-0.3)--(1.9,-0.3)--(1.9,-0.6);
        \draw node at (2.6,-0.45) {$R$};
    \end{tikzpicture}
    \caption{The building block.}
    \label{fig:shearloops}
\end{figure}

The building block $w_{l, \overline{l},  a, v} : \R^2 \to \R^2 $ is a compactly supported divergence free velocity field defined by the following lemma. We will use the 3 parameters short hand notation $w_{l,  a, v}$ whenever $\overline{l} = l$.

\begin{lem}[Building block] \label{lemma:building-block}
    For any $l, a \in [0,1/8[$, $v \in \R^+$, there exists an autonomous velocity field  $w= w_{l, \overline{l},  a, v} \in BV( \R^2; \R^2)$ which enjoys the following properties:
    \begin{itemize}
    \item $\supp (w) \subset 
    \big([-l/2-a,l/2+a] \times [-\bar{l}/2-a,\bar{l}/2+a]\big) \backslash \big( [-l/2+a,l/2-a] \times [-\bar{l}/2+a,\bar{l}/2-a] \big)\,,$
        \item $\| w \|_{L^\infty} = v \,,$
        \item $\dv (w) =0$, namely $w = \nabla^\perp H $ for some $H: \R^2 \to \R$.
        \item The non-degenerate streamlines of $w$, i.e. the sets $\{ H = h \} \cap \{ \nabla H \neq 0 \}$, are closed smooth curves and the period 
        $T(x)= \inf \{ s \geq 0: X_{s,0}(x) = x \} $
        satisfies 
        $$ T(x ) \leq \frac{2 l + 2\bar{l} + 4 \pi a }{v}\,.$$
        \item Let $Q$ be the interior of $\supp(w)$, then  $w \in C^\infty ({Q})$ and it satisfies
        $$ \| \nabla w \|_{L^\infty (Q)} = \frac{v}{a}\,.$$
        \item There exists a constant $C>0$ independent on $a, v, l$ such that  for any $\gamma : [0, T] \to \R^2$ integral curve of $w$ starting at $x \in Q_\varepsilon = \{ z \in \R^2: \dist (z, Q^c) > \varepsilon  \}$ we have 
        $$\int_0^T \| \nabla w  \|_{L^\infty (B_{\varepsilon} (\gamma (s)))} ds \leq C \,, $$
       where $T$ is the period of the streamline of $x$.
    \end{itemize}
\end{lem}
\begin{proof}
We consider the set
$R =  \{ (x,y) : |x| < \overline{l} /2-a  \,, |y|< l/2-a  \}  $
and define the set 
$$R_a = \{ (x,y): \dist ((x,y), R  ) < a \} \,.$$
We define the hamiltonian as  
\begin{equation}\label{eq:H}
    H(x,y) :=  \min \left\{ \dist\big( (x,y), R_a \big), a\right\}
    \end{equation}
   and finally we define the building block
    \begin{equation}
        w (x,y) = v \nabla^\perp H = v
    \begin{pmatrix}
        -\partial_y H\\
        \partial_x H 
    \end{pmatrix} \,.
    \end{equation}
The first three properties are straightforward from the definition. 
The fourth property follows from the fact that there are straight parts where the velocity field is such that $\nabla w =0$, i.e. it is constant and there are curved parts where (up to change coordinates) the velocity field locally can be written as 
$$w (x) = v \frac{x^\perp}{|x|} \qquad \text{with } |x| \geq a \,,$$
and therefore the fourth property follows. 
For the final property it is enough to see that the measure of $I = \{ t \in (0, T):  \| \nabla w \|_{L^\infty (B_{\varepsilon} (\gamma (t)))} \neq 0 \} $ is bounded by 
$$|I | \leq C \frac{a}{v} \,.$$
Therefore, using also the bound of the fourth point, we have 
$$ \int_0^T \| \nabla w \|_{L^\infty (B_\varepsilon (\gamma (s)))} ds = \int_{I} \| \nabla w \|_{L^\infty (B_\varepsilon (\gamma (s)))} ds \leq |I|  \| \nabla w \|_{L^\infty } \leq C  \,.$$
\end{proof}

\subsection{The choice of the parameters}

Suppose from now on $1\le p < 2$ is fixed as in Theorem \ref{thm:NU_ADE_loops}. First set up a super exponential sequence $(a_q)_{q\ge 0}$ with $a_0> 0$ (that can be made arbitrarily small and it is chosen in the proof to reabsorb constants) and 
\begin{align} \label{d:a-q}
a_{q+1} = a_q^{1+\delta}
\end{align}
and we may use the notation $a_{-1} = a_0 + a_1$.
Here $\delta \in (0, 1/10)$ is chosen sufficiently small so that
\begin{equation}\label{hyp:delta}\tag{H$\delta$}
    p < \frac{(2+\delta)(1- \delta)}{1+\delta}\,.
\end{equation}
Notice this hypothesis is possible only because the right hand side is a continuous  function in $\delta \geq 0$ and for $\delta =0$ the condition reduces to  $p <2$.
Actually, for computational purposes we want $a_q$ to be a multiple of $a_{q+1}$ for each $q\ge 0$. Thus we should choose inductively $a_{q+1}$ such that $a_q/a_{q+1}$ is an integer inside of $[a_q^{-\delta},a_q^{-\delta}+1]$. As this step is only technical and would make the text heavy to read, we avoid it. We define the horizontal and vertical length of the q-th loop as 
$$ l_q = \bar{l}_q = 4a_{q-1}\,.$$
For any $q \ge 0$ we define the speed in the $q$-th loop as
\begin{equation} \label{d:v-q}
    v_q = a_q^{-\alpha}\,,
\end{equation}  
where $\alpha \in (1, 11/10)$ satisfies the two following condition{s}:
\begin{equation}\label{hyp:ha}
   \frac{1}{1- \delta} < \alpha  < \frac{2+\delta}{p(1+\delta)} \tag{H$\alpha$}
\end{equation}
which is possible thanks to the condition on $\delta >0$ in \eqref{hyp:delta}.
Moreover, to close the estimate we use the parameter $\varepsilon >0$ satisfying 
    \begin{equation}\tag{H$\varepsilon$}\label{hyp:epsilon}
        \varepsilon < \min \left  \{ \frac{\delta(\alpha -1)}{1+\delta} , \frac{\delta^2}{8} \right \}\,.
    \end{equation}

    Finally,  $a_0 \in (0, a)$ with $a $ small enough so that  
    \begin{equation} \label{eq:sum-a_k}
    {\sum_{k \ge 0} a_k \le } \sum_{k \geq 0 } a^{ \varepsilon (1+ \delta)^k} \leq 2 \,.
    \end{equation}


\subsection{The construction of the velocity field}

We  define the centers of the $q$-th pipe
 $$ O_q = \left( -  \sum_{k\ge q-1} 2 a_k,  \sum_{k\ge q-1} 2 a_k \right) \,,$$ 
 and we define 
  $$Q_q =  \supp \{  w_{l_q, a_q, v_q} (\cdot - O_q) \}\,,$$
  where the $ w_{l_q, a_q, v_q}$ is the building block constructed in Lemma \ref{lemma:building-block}. We also define the intersection of two consecutive pipes as 
  $$S_q = Q_{q} \cap Q_{q+1} \,,$$
  and the minimum travel time for the backward flow of $w_q$ to move a point from $S_{q}$ to $S_{q-1}$ is denoted by 
  $$ t_{q, q-1} =  { \sup } \left \{ t \in [0,1]: \sup_{x \in S_{q}}  \overline{\bX}^q_{1, t} (x) \in S_{q-1} \right \} \,, $$
  where $\overline{\bX}^n_{1,\cdot}$ is the backward trajectory of $w_{l_q, a_q, v_q}$ which enjoys $t_{q, q-1} \leq 6 \frac{a_{q-1}}{v_q}$.
 Notice that we chose $1 $ as { the} initial time of the backward trajectory, but since the velocity field $w_{l_q, a_q, v_q}$ is autonomous this choice is irrelevant. Moreover,  we  define the gap time to reach the q-th pipe for the subsequent chunks in the $q-1$-th pipe as 
$$\overline{\tau}_{q} = \frac{a_{q}}{v_{q-1}} \,, \qquad \forall q \geq 1\,,$$
and we remark that there exists $M \in \N$ large such that $\overline{\tau}_{q} = M \overline{\tau}_{q+1}  $.
The gap time is represented in figure \ref{fig:times}.

 Next, we construct the intermittent{--}in{--}time $w_q : \R \times \R^2 \to \R^2$ as
 \begin{align} \label{d:w_q}
 w_q (t,x) = w_{l_q, a_q, v_q} (x - O_q) \sum_{j =- \infty }^\infty \mathbbm{1}_{[j \overline{\tau}_q + t_{q-1} \, , \,  j \overline{\tau}_q + t_q + \frac{a_{q-1}}{v_q} ]} (t) \,,
 \end{align}
 where $t_q = \sum_{k=1}^{q} t_{k, k-1}$
 and clearly we have 
 \begin{equation} \label{eq:Linfty-w_q}
  \| w_{q} \|_{L^\infty} \leq v_q \,.
\end{equation} 
 {Moreover,} we define the gluing velocity field {which brings mass from the second quadrant to the first quadrant of the torus when $n$ is even, and to the third quadrant when $n$ is odd:}
       $$w_{n, \text{glue}} (t,x) = \sum_{j = -\infty}^\infty w_{l_{n, \text{glue}}, \overline{l}_{n, \text{glue}}, a_{n+1}, v_{n+1}} (x - O_{n, \text{glue}}) \mathbbm{1}_{[ j \overline{\tau}_{n+1} + t_n \, , j \overline{\tau}_{n+1} + t_n + t_{\text{glue}, n}]} (t).$$
The gluing parameters {differ when $n$ is even or odd, and} are given in the following table :
\begin{equation}
 \label{eq:table}
\begin{tabular}{c | c c c c} 
 & $O_{n, \text{ glue}}$ & $l_{n, \text{glue}}$ &  $\overline{l}_{n, \text{glue}}$ & $t_{\text{glue}, n}$ \\ [0.5ex] 
 \hline 
 $n$ even & $\left(0,O_{n+1, y}\right)$ & $O_{n+1, x} + l_{n+1}$ & $l_{n+1}$ & $\{ t: \gamma (t) = (-x_1, x_2) \}$ \\ [0.5ex] 
 $n$ odd  & $\left(O_{n+1, x},0\right)$ & $l_{n+1}$ & $O_{n+1, y} + l_{n+1}$ & $\{ t: \gamma (t) = (x_1, -x_2) \}$
\end{tabular}
\end{equation}
      where $\gamma : \R \to \R^2$ is the integral curve {induced from $w_{n,\text{glue}}$} starting from $\gamma(0) = (x_1, x_2)$ where $(x_1, x_2) $ is the center of the rectangle $Q_n \cap Q_{n, \text{glue}}$. {The support of $w_{n,\text{glue}}$ is shown as a red pipe in Figure \ref{fig:loopstructure}.}
Next, for any $n \in \N$, we define  the intermittent{--}in{--}time velocity fields  $b_{n,1 } : \R \times \R^2 \to \R^2$ as 
 $$b_{n, 1} (t, x) =  \sum_{q = 0}^n w_q (t,x) \,,$$

Finally, we define 
\begin{align} \label{d:symmetry:VF}
  b_n (t, x) & = b_{n, 1} (t,x)  + \begin{pmatrix}
             b_{n,1}^{(1)} (T_n - t, -x_1, x_2) \\
           - b_{n,1}^{(2)} (T_n - t, -x_1, x_2) 
         \end{pmatrix}  + \begin{pmatrix}
            - b_{n,1}^{(1)} (T_n - t, -x_1, - x_2) \\
            - b_{n,1}^{(2)} (T_n -t ,  -x_1, - x_2) 
         \end{pmatrix}  \notag
         \\
         & \quad +  \begin{pmatrix}
             - b_{n,1}^{(1)} (T_n -t  , x_1, - x_2) \\
            b_{n,1}^{(2)} (T_n - t, x_1, - x_2)
         \end{pmatrix} + w_{n, \text{glue}} (t,x)  \,,
\end{align}  
and $T_n$ is defined accordingly to have {the reflection property $b_n(t,x) = - b_n(T_n -t,x)$} needed in Lemma \ref{lem:loopstabfor}, i.e. we define 
\begin{equation} \label{d:T_n}
T_n =    2 \frac{a_0}{a_1} \overline{\tau}_{1} +2   \sum_{k=1}^{n} t_{k, k-1}   +   \sum_{i=1}^{n+1}  \frac{a_{i-2}}{a_{i}}\overline{\tau}_{i} +  t_{\text{glue}, n} \, .
\end{equation}
{
\begin{remark}
The gluing velocity field is a small perturbation of the natural sequence of velocity fields that one may consider, namely \eqref{d:symmetry:VF} without $w_{n, \text{glue}}$, i.e. it holds that $\| w_{n, glue} \|_{L^\infty_t L^1_x } \leq 10 \| w_{n+1} \|_{ L^\infty_t L^1_x} \to 0$ as $n \to \infty$. 
Thanks to  \eqref{eq:table}, we observe that the glued velocity field will strongly affect the Lagrangian trajectories of $b_n$, having completely different behaviours for $n$ odd and $n$ even. We depict the different behaviours in  Figures \ref{fig:odd-even} and \ref{fig:nu_trajectories_bis}.
\end{remark}

\begin{figure}[htbp]
    \centering
    	\begin{tikzpicture}[scale=3]
			\draw[->] (-1,0)--(1,0);
			\draw[->] (0,-1)--(0,1);
			\filldraw[fill=red!30!white, fill opacity=0.5, draw=red] (-0.7,0.7) rectangle (-0.8,0.8);
			\filldraw[fill=vert!30!white, fill opacity=0.5, draw=vert] (0.7,0.7) rectangle (0.8,0.8);
			\filldraw[fill=blue!30!white, fill opacity=0.5, draw=blue] (-0.7,-0.7) rectangle (-0.8,-0.8);
			\draw[dotted,->](-0.7,0.7)--(-0.3,0.3);
			\draw[dotted,vert,->](-0.3,0.3)--(0.3,0.3);
			\draw[dotted,->](0.3,0.3)--(0.7,0.7);
			\draw[dotted, blue,->](-0.3,0.3)--(-0.3,-0.3);
			\draw[dotted,->](-0.3,-0.3)--(-0.7,-0.7);
			\draw node at (0.9,-0.1) {$x$};
			\draw node at (0.1,0.9) {$y$};
		\end{tikzpicture}
    \caption{{In red we represent an approximation of the support of the initial datum to the \eqref{eq:ADE}.  We depict in green  an approximation of the  support of the unique solution to the advection equation with $b_{2n}$ at time $T>0$; in blue an approximation of the support of the unique solution to the advection equation with $b_{2n+1}$ at time $T >0$.}
  }
    \label{fig:odd-even}
\end{figure}
\begin{figure}[htbp]
	\begin{subfigure}{0.15\textwidth}
		\begin{tikzpicture}[scale=1.3]
			\draw[->] (-1,0)--(1,0);
			\draw[->] (0,-1)--(0,1);
			\filldraw[fill=red!30!white, fill opacity=0.5, draw=red] (-0.7,0.7) rectangle (-0.8,0.8);
			\filldraw[fill=vert!30!white, fill opacity=0.5, draw=vert] (0.7,0.7) rectangle (0.8,0.8);
			\draw[dotted,->](-0.7,0.7)--(-0.6,0.6);
			\draw[dotted,vert,->](-0.6,0.6)--(0.6,0.6);
			\draw[dotted,->](0.6,0.6)--(0.7,0.7);
			\draw node at (0.9,-0.1) {$x$};
			\draw node at (0.1,0.9) {$y$};
		\end{tikzpicture}
		\caption{$b_0$}
	\end{subfigure}
	\hfill
	\begin{subfigure}{0.15\textwidth}
		\begin{tikzpicture}[scale=1.3]
			\draw[->] (-1,0)--(1,0);
			\draw[->] (0,-1)--(0,1);
			\filldraw[fill=red!30!white, fill opacity=0.5, draw=red] (-0.7,0.7) rectangle (-0.8,0.8);
			\filldraw[fill=blue!30!white, fill opacity=0.5, draw=blue] (-0.7,-0.7) rectangle (-0.8,-0.8);
			\draw[dotted,->](-0.7,0.7)--(-0.3,0.3);
			\draw[dotted, blue,->](-0.3,0.3)--(-0.3,-0.3);
			\draw[dotted,->](-0.3,-0.3)--(-0.7,-0.7);
			\draw node at (0.9,-0.1) {$x$};
			\draw node at (0.1,0.9) {$y$};
		\end{tikzpicture}
		\caption{$b_1$}
	\end{subfigure}
	\hfill
	\begin{subfigure}{0.15\textwidth}
		\begin{tikzpicture}[scale=1.3]
			\draw[->] (-1,0)--(1,0);
			\draw[->] (0,-1)--(0,1);
			\filldraw[fill=red!30!white, fill opacity=0.5, draw=red] (-0.7,0.7) rectangle (-0.8,0.8);
			\filldraw[fill=vert!30!white, fill opacity=0.5, draw=vert] (0.7,0.7) rectangle (0.8,0.8);
			\draw[dotted,->](-0.7,0.7)--(-0.2,0.2);
			\draw[dotted,vert,->](-0.2,0.2)--(0.2,0.2);
			\draw[dotted,->](0.2,0.2)--(0.7,0.7);
			\draw node at (0.9,-0.1) {$x$};
			\draw node at (0.1,0.9) {$y$};
		\end{tikzpicture}
		\caption{$b_2$}
	\end{subfigure}
	\hfill
	\begin{subfigure}{0.15\textwidth}
		\begin{tikzpicture}[scale=1.3]
			\draw[->] (-1,0)--(1,0);
			\draw[->] (0,-1)--(0,1);
			\filldraw[fill=red!30!white, fill opacity=0.5, draw=red] (-0.7,0.7) rectangle (-0.8,0.8);
			\filldraw[fill=blue!30!white, fill opacity=0.5, draw=blue] (-0.7,-0.7) rectangle (-0.8,-0.8);
			\draw[dotted,->](-0.7,0.7)--(-0.15,0.15);
			\draw[dotted, blue,->](-0.15,0.15)--(-0.15,-0.15);
			\draw[dotted,->](-0.15,-0.15)--(-0.7,-0.7);
			\draw node at (0.9,-0.1) {$x$};
			\draw node at (0.1,0.9) {$y$};
		\end{tikzpicture}
		\caption{$b_3$}
	\end{subfigure}
	\caption{{A more detailed approximation of Figure \ref{fig:odd-even}. We represent the supports of the unique solutions to the advection equation with velocity fields $b_0, b_1, b_2, b_3 $. The dotted lines are an approximation of unique  Lagrangian trajectories for such velocity fields. The gluing velocity field $w_{n, \text{glue}}$ is responsible for the {\em vertical} dotted lines for $n$ odd; whereas it is responsible for  the {\em horizontal} dotted lines for $n$ even.}}
	\label{fig:nu_trajectories_bis}
\end{figure}
}
For convenience we set $T= \lim_{n \to \infty} T_n + \frac{5a_0}{v_0}$ which ensures $1> T > T_n$ for any $n \in \N$.
       \begin{remark}\label{rmk:time_dependency}
         Notice that we are imposing an even-even symmetry in space at level of the Hamiltonian function in the definition of $b_n$. Moreover the time symmetry $b_n(t,x) = - b_n(T_n -t,x)$ holds for all $n \ge 0$.
       \end{remark}
       
%
%
For the upcoming Proposition \ref{prop:basicprop}, we will need a good control over the mass that arrives up to the gluing pipe while keeping a reasonable distance to the boundary. We control it using the following tools. $\beta= \beta (n+1) = (\beta_1, \ldots, \beta_{{n+1}})$ be a multi-index such that  $\beta_k \leq  \frac{a_{k-2}}{a_{k}}$ for any $k \geq 1$, recalling the notation  {$a_{-1} =a_0$}. We now introduce
the sequence of arrival times { to the middle part of the support of the gluing velocity field}  as
$$t_{n, \beta} = \sum_{k=1}^{n} t_{k, k-1} + \sum_{k=1}^{n+1} \beta_k \overline{\tau}_k + \frac{t_{\text{glue}, n}}{2} \,,$$
{ and for any $q \in \N$ and for any $\beta = \beta(q) = (\beta_1, \ldots, \beta_{{q}})$  multi-index such that  $\beta_k \leq  \frac{a_{k-2}}{a_{k}}$ for any $k \geq 1$, we define the arrival times to the set $Q_q$ as 
$$ \overline{t}_{q, \beta} = \sum_{k=1}^{q} t_{k, k-1} + \sum_{k=1}^{q+1} \beta_k \overline{\tau}_k \,. $$
These times are represented in Figures \ref{fig:times}. }

\begin{figure}
	\newcommand{\loops}{
		\draw (-4,0) -- (-1,0) arc(-90:0:1) -- (0,1.5);
		\draw (-4,-1) -- (-1,-1) arc(-90:0:2) -- (1,1.5);
		\draw (-2,1.5) -- (-2,-1.7) arc(180:270:0.3) -- (1.7,-2) arc(-90:0:0.3) -- (2,1.5);
		\draw (-2.3,1.5) -- (-2.3,-1.7) arc(180:270:0.6) -- (1.7,-2.3) arc(-90:0:0.6) -- (2.3,1.5);
		\draw (1.4,-1.6) -- (1.4, -2.4) arc(180:270:0.2) -- (2.4,-2.6) arc(-90:0:0.2) -- (2.6,-1.6) arc(0:90:0.2) -- (1.6,-1.4) arc(90:180:0.2);
		\draw (1.2,-1.6) -- (1.2, -2.4) arc(180:270:0.4) -- (2.4,-2.8) arc(-90:0:0.4) -- (2.8,-1.6) arc(0:90:0.4) -- (1.6,-1.2) arc(90:180:0.4);
	}
	\newcommand{\labels}{
		\draw[->] (-1.5,-0.5) -- (-1.2,-0.5);
		\draw[->] (-1.5,-2.15) -- (-0.8,-2.15);
		\draw[->] (1.5,-2.7) -- (2.5,-2.7);
		\draw (-1.3,0.4) node {$ Q_0 $};
		\draw (-2.5,-2.5) node {$ Q_1 $};
		\draw (0.7,-3) node {$ Q_2 $};
	}
	\newcommand{\loopsbis}{
		\draw (-4,0) -- (-1,0) arc(-90:0:1) -- (0,1.5);
		\draw (-4,-1) -- (-1,-1) arc(-90:0:2) -- (1,1.5);
		\draw (-2,1.5) -- (-2,-1.5) arc(180:270:0.5) -- (1.5,-2) arc(-90:0:0.5) -- (2,1.5);
		\draw (-2.5,1.5) -- (-2.5,-1.5) arc(180:270:1) -- (1.5,-2.5) arc(-90:0:1) -- (2.5,1.5);
		\draw (1,-1.2) -- (1, -2.8) arc(180:270:0.2) -- (2.8,-3) arc(-90:0:0.2) -- (3,-1.2) arc(0:90:0.2) -- (1.2,-1) arc(90:180:0.2);
		\draw (0.8,-1.2) -- (0.8, -2.8) arc(180:270:0.4) -- (2.8,-3.2) arc(-90:0:0.4) -- (3.2,-1.2) arc(0:90:0.4) -- (1.2,-0.8) arc(90:180:0.4);
	}
	\newcommand{\labelsbis}{
		\draw[->] (-1.5,-0.5) -- (-1.2,-0.5);
		\draw[->] (-1.5,-2.15) -- (-0.8,-2.15);
		\draw[->] (1.5,-2.7) -- (2.5,-2.7);
		\draw (-1.07,0.5) node {$ Q_{q-1} $};
		\draw (-2.5,-2.5) node {$ Q_q $};
		\draw (0,-3) node {$ Q_{q+1} $};
	}
	\newcommand{\gluingloops}{
		\draw (-5,1.7)--(-5,-1.7) arc(180:270:0.3) -- (4.7,-2) arc(-90:0:0.3) -- (5,1.7) arc(0:90:0.3) -- (-4.7,2) arc(90:180:0.3);
		\draw (-5.3,1.7)--(-5.3,-1.7) arc(180:270:0.6) -- (4.7,-2.3) arc(-90:0:0.6) -- (5.3,1.7) arc(0:90:0.6) -- (-4.7,2.3) arc(90:180:0.6);
		\draw (-7,0)--(-4,0)arc(-90:0:1)--(-3,3);
		\draw (-7,-1)--(-4,-1)arc(-90:0:2)--(-2,3);
		\draw (7,0)--(4,0)arc(-90:-180:1)--(3,3);
		\draw (7,-1)--(4,-1)arc(-90:-180:2)--(2,3);
	}
	\newcommand{\gluinglabels}{
		\draw (-6.2,1) node {$Q_n$};
		\draw (7.4,1) node{$P_n(Q_n)$};
		\draw (0,-1) node{$Q_{n,\text{ glue}}$};
		\draw[->] (-4.5,-0.5) -- (-3.5,-0.5);
		\draw[->] (-4,-2.15) -- (-2,-2.15);
		\draw[->] (3.5,-0.5) -- (4.5,-0.5);
	}
	\centering
	\begin{subfigure}{0.24\textwidth}
		\begin{tikzpicture}[scale=0.5]
			\loops
			\labels
			\filldraw[fill=blue!30!white, fill opacity=0.5, draw=blue] (-3.3,-1) -- (-2.3,-1) -- (-2.3,0) -- (-3.3,0) --(-3.3,-1);
		\end{tikzpicture}
		\caption{$\bar{t}_{1,(0)} $}
	\end{subfigure}
	\begin{subfigure}{0.24\textwidth}
		\begin{tikzpicture}[scale=0.5]
			\loops
			\filldraw[fill=blue!30!white, fill opacity=0.5, draw=blue] (-3,-1) -- (-2,-1) -- (-2,0) -- (-3,0) --(-3,-1);
		\end{tikzpicture}
		\caption{$\bar{t}_{1,(1)}$}
	\end{subfigure}
	\begin{subfigure}{0.24\textwidth}
		\begin{tikzpicture}[scale=0.5]
			\loops
			\filldraw[fill=blue!30!white, fill opacity=0.5, draw=blue] (-2.6,-1) -- (-2,-1) -- (-2,0) -- (-2.6,0) --(-2.6,-1);
		\end{tikzpicture}
		\caption{$\bar{t}_{1,(\frac{a_0}{a_1} - 1)}$}
	\end{subfigure}
	\begin{subfigure}{0.24\textwidth}
		\begin{tikzpicture}[scale=0.5]
			\loops
			\filldraw[fill=blue!30!white, fill opacity=0.5, draw=blue] (-2.3,-1) -- (-2,-1) -- (-2,0) -- (-2.3,0) --(-2.3,-1);
		\end{tikzpicture}
		\caption{$\bar{t}_{1,(\frac{a_0}{a_1})} $}
	\end{subfigure}
	\medskip
	\[
	\begin{tikzpicture}[scale = 2.4]
		\draw (0,0) --(2.7,0);
		\draw[dotted] (2.7,0) -- (3.3,0);
		\draw (3.3,0) -- (6,0);
		\foreach \x in {0,6}
		\draw (\x cm,2pt) -- (\x cm,-2pt);
		\foreach \x in {1,2,3.5,4.5}
		\draw (\x cm,1pt) -- (\x cm,-1pt);
		\draw (0,0) node[below=3pt] {$ \bar{t}_{0,(0)} = 0 $} node[above=3pt] {$   $};
		\draw (1,0) node[below=3pt] {$ \bar{t}_{0,(1)} $} node[above=3pt] {$  $};
		\draw (2,0) node[below=3pt] {$ \bar{t}_{0,(2)} $} node[above=3pt] {$  $};
		\draw (3.5,0) node[below=3pt] {$ \bar{t}_{0,(\frac{a_0}{a_1} - 1)} $} node[above=3pt] {$   $};
		\draw (4.5,0) node[below=3pt] {$ \bar{t}_{0,(\frac{a_0}{a_1})} $} node[above=3pt] {$   $};
		\draw (6,0) node[below=3pt] {$ T_n $} node[above=3pt] {$  $};
	\end{tikzpicture}
	\]
	\medskip
	\centering
	\begin{subfigure}{0.24\textwidth}
		\begin{tikzpicture}[scale=0.5]
			\loops
			\labels
			\filldraw[fill=blue!30!white, fill opacity=0.5, draw=blue] (-3,-1) -- (-2.3,-1) -- (-2.3,0) -- (-3,0) --(-3,-1);
			\filldraw[fill=blue!30!white, fill opacity=0.5, draw=blue] (-0.1,-2.3) -- (0.9,-2.3) -- (1.2,-2) -- (0.2,-2) --(-0.1,-2.3);
		\end{tikzpicture}
		\caption{$\bar{t}_{2,(1,0)}$}
	\end{subfigure}
	\begin{subfigure}{0.24\textwidth}
		\begin{tikzpicture}[scale=0.5]
			\loops
			\filldraw[fill=blue!30!white, fill opacity=0.5, draw=blue] (-3,-1) -- (-2.3,-1) -- (-2.3,0) -- (-3,0) --(-3,-1);
			\filldraw[fill=blue!30!white, fill opacity=0.5, draw=blue] (0.2,-2.3) -- (1.1,-2.3) -- (1.4,-2) -- (0.5,-2) --(0.2,-2.3);
		\end{tikzpicture}
		\caption{$\bar{t}_{2,(0,1)}$}
	\end{subfigure}
	\begin{subfigure}{0.24\textwidth}
		\begin{tikzpicture}[scale=0.5]
			\loops
			\filldraw[fill=blue!30!white, fill opacity=0.5, draw=blue] (-3,-1) -- (-2.3,-1) -- (-2.3,0) -- (-3,0) --(-3,-1);
			\filldraw[fill=blue!30!white, fill opacity=0.5, draw=blue] (0.5,-2.3) -- (1.4,-2.3) -- (1.4,-2) -- (0.8,-2) --(0.5,-2.3);
		\end{tikzpicture}
		\caption{$\bar{t}_{2,(0,2)}$}
	\end{subfigure}
	\begin{subfigure}{0.24\textwidth}
		\begin{tikzpicture}[scale=0.5]
			\loops
			\filldraw[fill=blue!30!white, fill opacity=0.5, draw=blue] (-3,-1) -- (-2.3,-1) -- (-2.3,0) -- (-3,0) --(-3,-1);
			\filldraw[fill=blue!30!white, fill opacity=0.5, draw=blue] (0.9,-2.3) -- (1.4,-2.3) -- (1.4,-2) -- (1.1, -2) --(0.9,-2.3);
		\end{tikzpicture}
		\caption{$\bar{t}_{2,(0,\frac{a_0}{a_2})} $}
	\end{subfigure}
	\medskip
	\begin{tikzpicture}[scale = 1.2]
		\draw (-1,0) --(7,0);
		\draw[dotted] (7,0) -- (8,0);
		\draw (8,0) -- (11,0);
		\foreach \x in {0,10}
		\draw (\x cm,3pt) -- (\x cm,-3pt);
		\foreach \x in {3,3.1,3.2,3.3,3.4,3.5,3.6,3.7,3.8,3.9,4}
		\draw (\x cm,1pt) -- (\x cm,-1pt);
		\draw (0,0) node[below=3pt] {$ \bar{t}_{0,(1)} $} node[above=3pt] {$   $};
		\draw (10,0) node[below=3pt] {$ \bar{t}_{0,(2)} $} node[above=3pt] {$  $};
		\draw (3.5,0) node[below=3pt] {$\bar{t}_{1,(1,0)} \, ... \, \bar{t}_{1,(1,\frac{a_0}{a_2})}$} node[above=3pt] {$  $};
	\end{tikzpicture}
	\medskip
	\centering
	\begin{subfigure}{0.24\textwidth}
		\begin{tikzpicture}[scale=0.5]
			\loopsbis
			\labelsbis
			\filldraw[fill=blue!30!white, fill opacity=0.5, draw=blue] (-4,-1) -- (-2,-1) -- (-2,0) -- (-3.5,0) --(-4,-1);
		\end{tikzpicture}
		\caption{ $\bar{t}_{q,(\beta_1,...,\beta_{q})}$}
	\end{subfigure}
	\hfill
	\begin{subfigure}{0.24\textwidth}
		\begin{tikzpicture}[scale=0.5]
			\loopsbis
			\filldraw[fill=blue!30!white, fill opacity=0.5, draw=blue] (-4,-1) -- (-2.5,-1) -- (-2.5,0) -- (-3.5,0) --(-4,-1);
			\filldraw[fill=blue!30!white, fill opacity=0.5, draw=blue] (-0.2,-2) -- (-0.45,-2.5) -- (0.55,-2.5) -- (0.79,-2) --(-0.2,-2);
		\end{tikzpicture}
		\caption{$\bar{t}_{q +1,(\beta_1,...,\beta_{q},0)}$}
	\end{subfigure}
	\hfill
	\begin{subfigure}{0.24\textwidth}
		\begin{tikzpicture}[scale=0.5]
			\loopsbis
			\filldraw[fill=blue!30!white, fill opacity=0.5, draw=blue] (-4,-1) -- (-2.5,-1) -- (-2.5,0) -- (-3.5,0) --(-4,-1);
			\filldraw[fill=blue!30!white, fill opacity=0.5, draw=blue] (0,-2) -- (-0.25,-2.5) -- (0.75,-2.5) -- (1,-2) --(0,-2);
		\end{tikzpicture}
		\caption{$\bar{t}_{q+1,(\beta_1,...,\beta_{q},1)}$}
	\end{subfigure}
	\hfill
	\begin{subfigure}{0.24\textwidth}
		\begin{tikzpicture}[scale=0.5]
			\loopsbis
			\filldraw[fill=blue!30!white, fill opacity=0.5, draw=blue] (-4,-1) -- (-2.5,-1) -- (-2.5,0) -- (-3.5,0) --(-4,-1);
			\filldraw[fill=blue!30!white, fill opacity=0.5, draw=blue] (0.8,-2) -- (0.55,-2.5) -- (1,-2.5) -- (1,-2) --(0.8,-2);
		\end{tikzpicture}
		\caption{$\bar{t}_{q+1,(\beta_1,...,\beta_{q},\frac{a_{q-1}}{a_{q+1}})}$}
	\end{subfigure}
	\medskip
	\begin{tikzpicture}[scale = 1.2]
		\draw (-1,0) --(7,0);
		\draw[dotted] (7,0) -- (8,0);
		\draw (8,0) -- (11,0);
		\foreach \x in {0,10}
		\draw (\x cm,3pt) -- (\x cm,-3pt);
		\foreach \x in {3,3.1,3.2,3.3,3.4,3.5,3.6,3.7,3.8,3.9,4}
		\draw (\x cm,1pt) -- (\x cm,-1pt);
		\draw (0,0) node[below=3pt] {$ \bar{t}_{q-1,(\beta_1,...,\beta_{q})} $} node[above=3pt] {$   $};
		\draw (10,0) node[below=3pt] {$ \bar{t}_{q-1,(\beta_1,...,\beta_{q}+1)} $} node[above=3pt] {$  $};
		\draw (3.5,0) node[below=3pt] {$\bar{t}_{q,(\beta_1,...,\beta_q,0)} \, ... \, \bar{t}_{q,(\beta_1,...,\beta_q,\frac{a_{q-1}}{a_{q+1}})}$} node[above=6pt] {$  $};
	\end{tikzpicture}
	\medskip
	\begin{subfigure}{0.3\textwidth}
		\begin{tikzpicture}[scale=0.3]
			\gluingloops
			\gluinglabels
			\filldraw[fill=blue!30!white, fill opacity=0.5, draw=blue] (-7,-1)--(-5,-1)--(-5,0)--(-6,0)--(-7,-1);
		\end{tikzpicture}
		\caption{$\bar{t}_{n,(\beta_1,...,\beta_{n+1})}$}
	\end{subfigure}
	\hfill
	\begin{subfigure}{0.3\textwidth}
		\begin{tikzpicture}[scale=0.3]
			\gluingloops
			\filldraw[fill=blue!30!white, fill opacity=0.5, draw=blue] (-7,-1)--(-5.3,-1)--(-5.3,0)--(-6,0)--(-7,-1);
			\filldraw[fill=blue!30!white, fill opacity=0.5, draw=blue] (-0.65,-2.3)--(0.35,-2.3)--(0.65,-2)--(-0.35,-2)--(-0.65,-2.3);
		\end{tikzpicture}
		\caption{$t_{n,(\beta_1,...,\beta_{n+1})}$}
	\end{subfigure}
	\hfill
	\begin{subfigure}{0.3\textwidth}
		\begin{tikzpicture}[scale=0.3]
			\gluingloops
			\filldraw[fill=blue!30!white, fill opacity=0.5, draw=blue] (-7,-1)--(-5.3,-1)--(-5.3,0)--(-6,0)--(-7,-1);
			\filldraw[fill=blue!30!white, fill opacity=0.5, draw=blue] (5,-1)--(5.3,-1)--(5.3,0)--(5,0)--(5,-1);
		\end{tikzpicture}
		\caption{$\bar{s}_{n,(\beta_1,...,\beta_{n+1}})$}
	\end{subfigure}
	\medskip
	\begin{tikzpicture}[scale=1.2]
		\draw (-6,0)--(6,0);
		\foreach \x in {-5,-4.5,-4,4,4.5,5}
		\draw (\x cm,1pt) -- (\x cm,-1pt);
		\draw (-4.5,0) node[below=3pt] {$\bar{t}_{n,\beta} \, t_{n,\beta} \, \bar{s}_{n,\beta}$} node[above=3pt] {$  $};
		\draw (4.5,0) node[below=3pt] {$ \bar{t}_{n,\beta'} \, t_{n,\beta'} \, \bar{s}_{n,\beta'}$} node[above=3pt] {$  $};
	\end{tikzpicture}
	\caption{{Time representations of $\{ \overline{t}_{q, \beta} \}_{q, \beta}$. We represent in blue the support of the solution to the advection equation with velocity field $b_n$ with an initial datum with support in $Q_0$. We fix  the notation $(q,\beta) = (q,(\beta_1,...,\beta_{q+1}))$. We fix $q=1$  in (A), (B), (C), (D); $q=2$ in (E), (F), (G), (H). We show the transition of the times $\overline{t}_{q, \beta} $ from $q$  to $q+1$ in (I), (J), (K), (L). Finally, in (N) we show the arrival time to the middle part of the support of the gluing velocity field $t_{n, \beta}$.  
	Notice that the magnitude of the velocity field in $w_{q+1}$ is extremely larger than the magnitude of the velocity field $w_q$, which implies that the solution of the advection equation travels much faster inside the $q+1$ pipe, while the support of the solution of the advection equation remains almost constant inside the $q$ pipe. In these figures we keep the support of the solution inside $q$ pipe constant while the one in the $q+1$ pipe is travelling. We consequently represent the  order of events as timeline below the figures.
	In the last timeline interval we use the notation $(q,\beta) = (q,\beta_1,...,\beta_{q+1})$ and $\beta' = (\beta_1,...,\beta_{q+1}+1)$.}}
	\label{fig:times}
\end{figure}

 {We then} define a proper subset of these times, i.e.
  \begin{subequations}\label{eq:J-n}
\begin{align} 
B_n & = \left \{ \beta = (\beta_1, \ldots, \beta_{n+1}): 4 \frac{a_{k-2}^{1+ \varepsilon}}{a_{k}} \leq \beta_{k} \leq  \frac{a_{k-2}}{a_{k}} - 4 \frac{a_{k-2}^{1+\varepsilon}}{a_{k}} \qquad  \forall k \in \{1, \ldots , n +1\}  \right \}\,,
\\
J_n & = \left \{ t_{n, \beta} : \beta \in B_n \right \} \,.
\end{align}
\end{subequations}
Thanks to \eqref{eq:sum-a_k} the cardinality of these sets are bounded from above and below by 
\begin{equation} \label{eq:cardinality}
\frac{a_0^2}{a_n a_{n+1}} \geq \#(J_n)= \#(B_n) \geq \frac{9}{10} \frac{a_0^2}{a_n a_{n+1}} \,.
\end{equation}
{
Indeed
\[ \# (B_n) \le \prod_{k=1}^{n+1} \frac{a_{k-2}}{a_k}\]
while
\begin{align*}
	\# (B_n) &= \prod_{k=1}^{n+1} \left(\frac{a_{k-2}}{a_k}-8\frac{a_{k-2}^{1+\varepsilon}}{a_k}\right) = \frac{a_0^2}{a_n a_{n+1}}\prod_{k=1}^{n+1} (1-8a_{k-2}^\varepsilon) 
	 \geq  \frac{a_0^2}{a_n a_{n+1}}\prod_{k=10}^{n+1} (1-\frac{1}{k^2})  \geq \frac{9}{10} \frac{a_0^2}{a_n a_{n+1}}  \,,
\end{align*}
where in the second to last inequality we used the 
 smallness of $a_0$ and in the last inequality we used that $\prod_{k=10}^{n+1} (1-\frac{1}{k^2}) = \frac{9}{10} \frac{n+2}{n+1}$.
}

{Lastly, we will need a starting zone $R_0$ that is a square of width $a_0$ centered in a suitable point $x_0 \in \T^2$. More precisely, there exists $x_0 \in \T^2$ such that
	\begin{equation}\label{eq:R0}
		R_0 = x_0 + [0,a_0]^2 \,,
	\end{equation} 
	where $x_0$ is chosen so that 
	$$  \overline{\bX}^0_{ k \frac{a_1}{v_0}} (R_0) \cap S_0 = S_0 \text{ for any } k \in \{1 , \ldots, \frac{a_0}{a_1} \}\,,$$
 where $\overline{\bX}^0$ is the forward flow map of $w_0$. This set is represented in Figure \ref{fig:Ssets}.}
{Recall the that for any set $A \in \T^d$, $A[\varepsilon]$ denotes its $\varepsilon$-restriction as given in \eqref{eq:restriction}. We are now in position of proving the key properties we need on the velocity field $b$ for the advection equation. }

\begin{prop}\label{prop:basicprop}
For any $p \in (1,2)$, there exist a sequence  of divergence free velocity fields $\left \{ \bb_n \right \}_{n \in \N} \subset L^\infty ([0,1]; L^\infty (\R^2) \cap BV (\R^2)) $ with the following properties
\begin{enumerate}
\item \label{eq:prop:item1} $\bb_n \to \bb $ in $L^1_{t,x}$ and $\sup_{n \in \N} \| \bb_n \|_{L^\infty ((0,1); L^p (\T^2))} < \infty  $.
\item (Backward property). \label{eq:prop:item2}  For any $n \in \N$ there exists a discrete subset of times $J_n \subset [0,1]$ and a subset  $S_{n, \text{glue} } \subset \R^2$ such that
$$ S_{n, \text{glue} } = \begin{cases}
 \{ (x,y) \in \supp (w_{n, \text{glue}}):   \frac{-a_n + 8 a_n^{1+ \varepsilon}}{2} < x < \frac{a_n - 8 a_n^{1+ \varepsilon}}{2}  \,,  y < O_{n, \text{glue}, y}    \} &  n \in 2 \N \,,
\\
 \{ (x,y) \in \supp (w_{n, \text{glue}}):  \frac{-a_n + 8 a_n^{1+ \varepsilon}}{2}  < y <\frac{a_n - 8 a_n^{1+ \varepsilon}}{2}    \,, x < O_{n, \text{glue}, x}\} & n \in 2\N +1 \,,
\end{cases}$$
 such that for any $t_{n, \beta} \in J_n$ there exists a sequence $\{ \bar{t}_{q, \beta} \}_{q=-1}^n$ such that for any $x \in S_{n, \text{glue}} [2a_{n+1}^{1+ \varepsilon}]$ the backward flow $\bX^n_{t_{n , \beta}, \cdot }$ of $\bb_n$ satisfies
$$ \bX^n_{t_{n , \beta}, \bar{t}_{q, \beta}} (x) \in  Q_q [2 a_q^{1+ \varepsilon}] \cap S_q \left  [a_{q+1}^{1+ \varepsilon} \right ] \qquad \forall \, q \in \{0, 1, \ldots, n \} \,,$$
and  $\bar{t}_{q+1, \beta} \geq \bar{t}_{q, \beta}$ for any $q$ with $|\bar{t}_{q+1, \beta} - \bar{t}_{q, \beta}| \leq 6 \frac{a_{q}}{v_{q+1}}$.
Finally, $\bX^n_{{t}_{n, \beta}, t} (x) \in Q_q [a_q^{1+ \varepsilon}]$ for any $t \in [\bar{t}_{q-1, \beta}, \bar{t}_{q, \beta}]$ for any $q$, with the notation $\bar{t}_{-1, \beta} =0$. {In particular}, for any $n \in \N$, $x \in S_{n, \text{glue}}[2a_{n+1}^{1+\varepsilon}]$ and $t_{n, \beta} \in J_n$ we have 
\begin{equation} \label{eq:flow-backward-R-0}
 \bX^{n}_{t_{n, \beta} , 0} (x) \in R_0 [a_0^{1 + \varepsilon}] \,.
\end{equation}
{We recall $R_0$ is defined in \eqref{eq:R0}.}
\item (Forward property). \label{eq:prop:item3}  For any ${t}_{n , \beta} \in J_n$ there exists
  a sequence $\{ \bar{s}_{q, \beta} \}_{q=1}^{n}$ such that for any $x \in S_{n, \text{glue}} [2a_{n+1}^{1+ \varepsilon}]$ the forward flow $\bX^n_{{t}_{n , \beta}, \cdot }$ of $\bb_n$ satisfies
$$ \bX^n_{{t}_{n , \beta}, \bar{s}_{q, \beta}} (x) \in P_n ( Q_q [2 a_q^{1+ \varepsilon}] \cap S_q \left  [a_{q+1}^{1+ \varepsilon} \right ]) \qquad \forall \, q \in \{0, 1, \ldots, n \} \,,$$
where $P_n: \R^2 \to \R^2$ is the reflection on $y$-axis for $n$ even and is the reflection on $x$-axis for $n $ odd. Furthermore, $\bX^n_{{t}_{n , \beta}, s} (x) \in P_n (Q_q [a_q^{1+ \varepsilon}])$ for any $s \in [\bar{s}_{q, \beta}, \bar{s}_{q-1, \beta}]$ for any $q$. Finally, the times $ \{\bar{s}_{q, \beta} \}_{q}$ satisfy  $\bar{s}_{q, \beta} \geq \bar{s}_{q+1, \beta}$ for any $q$ with $|\bar{s}_{q+1, \beta} - \bar{s}_{q, \beta}| \leq 6 \frac{a_{q}}{v_{q+1}}$.  In particular, for any $x \in S_{n, \text{glue}}[2a_{n+1}^{1+\varepsilon}]$ and $t \geq T_n$ we have $\bX^{n}_{t_{n, \beta}, t} (x) \in  P_n (R_0)$ for any $n \in \N$, which implies
\begin{subequations} \label{eq:flow-odd-even}
\begin{align}
\bX^{n}_{t_{n, \beta}, t} (x)  \in    \left \{ (x,y): x,y > a_0 \right \} \quad \text{for any  } n \in 2 \N \,,
\\
\bX^{n}_{t_{n, \beta}, t} (x)  \in \left \{ (x,y): x,y < - a_0 \right \} \quad \text{for any  } n \in 2\N +1  \,.
\end{align}
\end{subequations}
 \item\label{prop:item:nablab} There exists a constant $C>0$ independent on all the parameters, such that  for any $x \in S_{n, \text{glue}} [a_{n+1}^{1+\varepsilon}] $ and for any $  1 \le q \le n,$ we have 
        \begin{equation} \label{eq:gronwall}
               \int_{\overline{s}_{q+1 , \beta}}^{\overline{s}_{q, \beta}} \|\nabla w_{q+1} (s,\cdot)\|_{L^\infty (B_{a_{q+1}^{1+\varepsilon}} (\gamma (s) ) )} ds +  \int_{\overline{t}_{q, \beta}}^{\overline{t}_{q+1, \beta}} \|\nabla w_{q+1} (s,\cdot)\|_{L^\infty (B_{a_{q+1}^{1+\varepsilon}} (\gamma (s) ) )} ds \le C 
\end{equation} 
        where $\gamma(s)$ is a shorthand notation for $\bX^n_{{t}_{n, \beta},s}(x)$ (forward or backward flow respectively if $s \geq t_{n, \beta}$ or $s \leq t_{n, \beta}$), and $B_{a_{q+1}^{1+\varepsilon}} (\gamma (s) )$ is a ball of radius $a_{q+1}^{1+ \varepsilon}$ around $\gamma(s)$.
\end{enumerate}
\end{prop}

\begin{figure}[!htb]
	\centering
	\begin{tikzpicture}[scale=5]
		\draw[->] (-1,0) -- (1,0);
		\draw[->] (0,-0.1) -- (0,0.6);
		\draw (-0.3,0.6) -- (-0.3, 0.4) arc (0:-90:0.1) -- (-0.63, 0.3) arc(270:180:0.1)-- (-0.73,0.6);
		\draw (-0.2,0.6) -- (-0.2, 0.4) arc(360:270:0.2)--(-0.63, 0.2) arc(270:180:0.2) -- (-0.83,0.6);
		\draw (0.3,0.6) -- (0.3, 0.4) arc (180:270:0.1) -- (0.63, 0.3) arc(-90:0:0.1)-- (0.73,0.6);
		\draw (0.2,0.6) -- (0.2, 0.4) arc(180:270:0.2)--(0.63, 0.2) arc(-90:0:0.2) -- (0.83,0.6);
		\draw (-0.5,0.47) arc(180:90:0.03) -- (-0.13,0.5) arc(90:0:0.03) -- (-0.1, 0.13) arc(360:270:0.03) -- (-0.47, 0.1) arc(270:180:0.03) -- (-0.5,0.47);
		\draw (-0.53,0.47) arc(180:90:0.06) -- (-0.13,0.53) arc(90:0:0.06) -- (-0.07, 0.13) arc(360:270:0.06)--(-0.47, 0.07) arc(270:180:0.06) -- (-0.53,0.47);
		\draw (0.5,0.47) arc(0:90:0.03) -- (0.13,0.5) arc(90:180:0.03) -- (0.1, 0.13) arc(180:270:0.03) -- (0.47, 0.1) arc(270:360:0.03) -- (0.5,0.47);
		\draw (0.53,0.47) arc(0:90:0.06) -- (0.13,0.53) arc(90:180:0.06) -- (0.07, 0.13) arc(180:270:0.06)--(0.47, 0.07) arc(270:360:0.06) -- (0.53,0.47);
		\draw (-0.16,0.15) arc(180:90:0.01) -- (0.15,0.16) arc(90:0:0.01) -- (0.16, 0.05) arc(360:270:0.01) -- (-0.15, 0.04) arc(270:180:0.01) -- (-0.16,0.15);
		\draw (-0.17,0.15) arc(180:90:0.02) -- (0.15,0.17) arc(90:0:0.02) -- (0.17, 0.05) arc(360:270:0.02)--(-0.15, 0.03) arc(270:180:0.02) -- (-0.17,0.15);
		\filldraw[fill=red!30!white, fill opacity=0.5, draw=red](-0.63,0.3)--(-0.53,0.3)--(-0.53,0.2)--(-0.63,0.2)--(-0.63,0.3);
		\filldraw[fill=blue!30!white, fill opacity=0.5, draw=blue](-0.53,0.3)--(-0.5,0.3)--(-0.5,0.2)--(-0.53,0.2)--(-0.53,0.3);
		\filldraw[fill=blue!30!white, fill opacity=0.5, draw=blue](-0.17,0.1)--(-0.16,0.1)--(-0.16,0.07)--(-0.17,0.07)--(-0.17,0.1);
		\filldraw[fill=yellow, fill opacity=0.2, draw=yellow!60!black](0.53,0.3)--(0.5,0.3)--(0.5,0.2)--(0.53,0.2)--(0.53,0.3);
		\filldraw[fill=yellow, fill opacity=0.2, draw=yellow!60!black](0.17,0.1)--(0.16,0.1)--(0.16,0.07)--(0.17,0.07)--(0.17,0.1);
		\filldraw[fill=vert!30!white, fill opacity=0.8,
		draw=vert](-0.015,0.04)--(0.015,0.04)--(0.015,0.03)--(-0.015,0.03)--(-0.015,0.04);
	\end{tikzpicture}
	\caption{The starting zone $R_0$ in red, the intersection zones $S_q$ {respectively $P_n(S_q)$} for $0 \le q \le n-1$ in blue {resp. yellow} and the middle zone $S_{n, \text{glue}}$ in green for $n \in \N$ even (here $n=2$).
		}
	\label{fig:Ssets}
\end{figure}

\begin{remark}
    The time dependency is necessary for the stability between the stochastic and the Lagrangian flow. Otherwise, the {Gr\"onwall} type estimate \eqref{eq:gronwall} would have been heuristically 
    $$\int_{\bar{t}_q}^{\bar{t}_{q+1}} \|\nabla \bb_n (s,\cdot)\|_{L^\infty (B_{a_{q}^{1+\varepsilon}} (\gamma (s) ) )} ds \approx \| \nabla b_n \|_{L^\infty (S_{q})} \Delta t_{S_{q}} $$
    where $\Delta t_{S_{q}} $ is the time spent in the set $S_{q}$ which is approximately the width of the $q-$th pipe over the velocity of the $q+1$-th pipe $ \approx \frac{a_{q}}{v_{q+1}}$ and 
   assuming that $b_n$ is smoothed out with the width parameter of the pipes we have 
    $$  \| \nabla b_n \|_{L^\infty (S_{q})} \approx \frac{v_{q+1}}{a_{q+1}}$$
and therefore the estimate
$$\int_{\bar{t}_q}^{\bar{t}_{q+1}} \|\nabla \bb_n (s,\cdot)\|_{L^\infty (B_{a_{q}^{1+\varepsilon}} (\gamma (s) ) )} ds \approx \frac{a_{q}}{a_{q+1}} \,. $$

The error $\frac{a_q}{a_{q+1}}$ is too large to have a good control as in Lemma \ref{lem:loopstabback}. The time dependence in our case seems crucial and the non-uniqueness with a two dimensional autonomous velocity field remains open. 
\end{remark}

\begin{proof}[Proof of Proposition \ref{prop:basicprop}]
We start by proving \eqref{eq:prop:item1}.
    Let us check $\bb$ is in $L^\infty_tL^p_x$. We notice there is no spot where more than two loops overlap. { We also notice that $\| w_{n, \text{glue}} \|_{L^\infty_t L^p_x} \leq 10 \| w_{n+1} \|_{L^\infty_t L^p_x}$.} Therefore, we compute up to constants that do only depend on $p$ :
    \begin{align*}
        \sup_{n\ge0} \|\bb_n \|_{L^\infty([0,T],L^p(\T^2))} & \lesssim \sup_{t \in (0,1)} \sum_{k\ge 0} \int_{Q_k} | w_k (t, x)|^p dx
        \lesssim \sum_{k \ge 0} v_k^p \L^2(Q_k) 
        \\
        & \lesssim \sum_{k \ge 0} a_k^{-\alpha p}a_{k-1}a_k =  \sum_{k\ge -1} a_k^{2+ \delta -\alpha p(1+\delta)} < \infty
    \end{align*}
    where the last holds thanks to \eqref{hyp:ha}.
    Next, we prove that $\bb_n$ is a Cauchy sequence in  $L^1([0,1]\times\T^2)$. 
    From \eqref{d:T_n} we have 
    \begin{align*}
     T_{n+1} - T_n & \leq 2 t_{n+1, n} +  \frac{a_{n}}{a_{n+2}} \overline{\tau}_{n+2} + t_{\text{glue}, n+1} - t_{\text{glue}, n}
    \\
    & \lesssim \frac{a_{n}}{v_{n+1}} + \frac{a_{n}}{a_{n+2}} \frac{a_{n+{2}}}{v_{n+1}} + \frac{a_{n-1}}{v_{n+1}} + \frac{a_{n}}{v_{n+2}}    
    \\
    & \leq a_{n-1}^2\,.
    \end{align*}
   Then we compute by \eqref{d:symmetry:VF} and the fact that $T<1$ and that $w_k$ will switch on and of at most $1/\overline{\tau}_k$ times over the time interval $[0,1]$,
   \begin{align*}
   \| \bb_{n+1} - \bb_n\|_{L^1([0,T]\times \T^2)} &\lesssim \| w_{n, \text{glue}} \|_{L^1} +\| w_{n+1, \text{glue}} \|_{L^1}  + \| w_{n+1} \|_{L^1}  
   \\
   & \quad +  a_{n-2}^2  \sum_{k=1}^n \frac{1}{\overline{\tau}_k} \sup_{t \in [0,1]}  \int_{\R^2}  | w_{k} (t, x)  | dx 
   \\
   & \lesssim v_{n+1} a_{n+1} a_{n-1} + v_{n+2} a_{n+2} a_{n} + v_{n+1} a_{n+1} a_n
   + a_{n-2}^2 \sum_{k=1}^n \frac{v_{k-1}}{a_k} v_k  a_{k-1}  a_k\\
   & { \lesssim a_{n-2}^{(1+\delta) + (1+\delta)^3(1-\alpha)} + a_{n-2}^{2+(1+\delta) - \alpha (2-\delta)(1+\delta)}} \\
   & \leq  a_{n-2}^{1/2} 
      \end{align*}
      where in the last {passage} we have used 
      {
      \begin{equation*}
      	\begin{cases}
      		(1+\delta) + (1+\delta)^3(1-\alpha) > 1/2\,,\\
      		2 + (1+\delta) - \alpha(2-\delta)(1+\delta) > 1/2
      	\end{cases}
      \end{equation*}
      }
      {as} $\alpha < 11/10$ and $\delta \in (0, 1/10)$. Therefore, {$\{b_n\}_n$} is a Cauchy sequence in $L^1$.
It is straightforward from the definition that $\bb_n$ is divergence free, which implies that the $L^1$ limit $\bb$ is divergence free.

We now prove \eqref{eq:prop:item2}.  Fix $x \in S_{n, \text{glue}}$ and $t_{n, \beta} \in J_n$, and we use the shorthand notation
$$\overline{t} = t_{n, \beta}= \sum_{k=1}^n t_{k, k-1} +  \sum_{k=1}^{n+1} \beta_k \overline{\tau}_k + \frac{t_{\text{glue} , n}}{2} $$
for some $\beta_k$ in the range given by \eqref{eq:J-n},
and we define the sequence $\{\overline{t}_{q,\beta} \}_{q}$ with the shorthand notation cutting the $\beta$ as
\begin{equation} \label{eq:times-t_q}
\overline{t}_q = \overline{t}_{q, \beta} =  \sum_{k=1}^q t_{k, k-1}  + \sum_{k=1}^{q+1} \beta_k \overline{\tau}_k \,, \qquad \forall q \in \{ 0, 1, \ldots , n \}  
\end{equation} 
with the same set of $\beta_k$ for $k \leq q+1$ given by $\bar{t}$ and we set $t_{-1, \beta} \equiv 0$ for any $\beta$.
We consider the backward flow $\overline{\bX}^{n, \text{glue}}_{\overline{t}, \cdot }$ of $w_{n, \text{glue}}$ and  we have 
\begin{equation} \label{eq:proof:S_nglue}
x \in S_{n, \text{glue}} \left  [2 a_{n+1}^{1+ \varepsilon} \right ] \Rightarrow  \overline{\bX}^{n, \text{glue}}_{\overline{t}, \overline{t}_{n}} (x) \in  Q_{n} [3 a_{n}^{1+ \varepsilon}] \cap S_{n} \left  [ 2 a_{n+1}^{1+ \varepsilon} \right ] \,,
\end{equation}
thanks to the definition of $S_{n, \text{glue}}$ and $t_{\text{glue}, n}$, see \eqref{eq:prop:item2} and \eqref{eq:table}. Considering the backward flow $\bX_{\overline{t} , \cdot} (x)$ of $b_n$ we define 
$$ D_{n, \text{glue}}(x) = \{ s \in (\overline{t}_{n} ,\overline{t}):\bX_{\overline{t}, \overline{t} - s}^n (x) \in S_{n} \} \,,$$
and we have $|D_{n, \text{glue}}| \leq \frac{a_{n+1}^{1+ \varepsilon}}{v_n}$ for all $x \in S_n$. Indeed, $D_{n, \text{glue}}$ can be estimated as the least time needed to $w_{n, \text{glue}}$ to travel a length $a_{n}$ which is 
$$|D_{n, \text{glue}}| \leq \frac{a_n}{v_{n+1}} \leq  \frac{a_{n+1}^{1+ \varepsilon}}{v_n} $$
thanks to \eqref{hyp:ha}. Therefore, we have 
$$ |  \overline{\bX}^{n, \text{glue} }_{\overline{t}, \overline{t}_{n}} (x) - \bX^n_{\overline{t}, \overline{t}_{n}} (x)| \leq \int_{D_{n, \text{glue}}} \| w_{n} (s, \cdot )\|_{L^\infty} ds \leq a_{n+1}^{1+ \varepsilon} \,,$$
thanks to \eqref{eq:Linfty-w_q}. From this estimate and \eqref{eq:proof:S_nglue} we get 
$$x \in S_{n, \text{glue}} \left  [2 a_{n+1}^{1+ \varepsilon} \right ] \Rightarrow  {\bX}^{n}_{\overline{t}, \overline{t}_{n}} (x) \in  Q_{n} [2 a_{n}^{1+ \varepsilon}] \cap S_{n} \left  [  a_{n+1}^{1+ \varepsilon} \right ] \,.$$
 The proof of \eqref{eq:prop:item2} follows iterating this argument. 
 Let $\overline{\bX}^q_{\overline{t}_q , \cdot} (y)$ be the backward flow of $w_q$ starting at $y$ at time $\overline{t}_q$, then it satisfies
\begin{equation} \label{eq:proof:iterate-n}
 y \in Q_q [2 a_q^{1+ \varepsilon}] \cap S_q \left  [a_{q+1}^{1+ \varepsilon} \right ] \Rightarrow  \overline{\bX}^q_{\overline{t}_q, \overline{t}_{q-1}} (y) \in  Q_{q-1} [3 a_{q-1}^{1+ \varepsilon}] \cap S_{q-1} \left  [  2 a_{q}^{1+ \varepsilon} \right ] \,,
\end{equation}
for any $q \in \{1, \ldots, n \}$.  Indeed, using that $\overline{\tau}_{q-1} = M \overline{\tau}_{q}$ for some $M \in \N$ and that there exists $j \in \N$ for which we have $\overline{t}_q = t_q + j \overline{\tau}_q + \beta_{q+1} \overline{\tau}_{q+1}$, we get that $| \overline{t}_q -  j \overline{\tau}_q - t_{q-1}|  = t_{q, q-1} + \beta_{q+1} \overline{\tau}_{q+1} $ and thanks  to \eqref{d:w_q} and the bound 
$$ 4 a_{q-1}^{1+ \varepsilon}  \leq  \beta_{q+1} \overline{\tau}_{q+1} v_q \leq a_{q-1} - 4 a_{q-1}^{1+ \varepsilon} \,, $$
we have   $ \overline{\bX}^q_{\overline{t}_q, \overline{t}_{q-1}} (y) \in  Q_{q-1} [3 a_{q-1}^{1+ \varepsilon}]$. The  property $\overline{\bX}^q_{\overline{t}_q, \overline{t}_{q-1}} (y) \in  S_{q-1} [ 2 a_{q}^{1+ \varepsilon}]$ follows from $y \in Q_q [2 a_q^{1+ \varepsilon}]$ and the streamline property of preserving this distance to the boundary of $w_q$.
 We now want to prove that   
$$\bX^n_{\overline{t}, \overline{t}_q} (x) \in  Q_q [2 a_q^{1+ \varepsilon}] \cap S_q \left  [a_{q+1}^{1+ \varepsilon} \right ] \Rightarrow  \bX^n_{\overline{t}, \overline{t}_{q-1}} (x) \in  Q_{q-1} [2 a_{q-1}^{1+ \varepsilon}] \cap S_{q-1} \left  [a_{q}^{1+ \varepsilon} \right ] \,.$$

We study the subset  of times $D_q \subset (0,1)$ defined as 
$$ D_q = \{ s \in (\overline{t}_{q-1} ,\overline{t}_{q}):\bX_{\overline{t}, \overline{t}_{q} - s}^n (x) \in S_{q-1} \} \,,$$
and we notice that  it is enough to prove that $|D_q| \leq \frac{a_q^{1+ \varepsilon}}{ v_{q-1}}$, indeed from this we deduce 
\begin{equation} \label{eq:stability:X-overlineX}
 \sup_{t \in [\overline{t}_{q-1} , \overline{t}_{q}]} |  \overline{\bX}^q_{\overline{t}_q, t} (y) - \bX^n_{\overline{t}, t } (x)| \leq \int_{D_q} \| w_{q-1} (s, \cdot )\|_{L^\infty} ds \leq a_q^{1+ \varepsilon} \,,
\end{equation}
where we considered $y = \bX_{\overline{t}, \overline{t}_q}^n (x)$, which 
concludes the proof thanks to \eqref{eq:proof:iterate-n}.
The estimate on $D_q$ can be proved computing  the least time needed to $w_q$ travelling a length $a_{q-1}$, which is similar as above
$$ |D_q| \leq \frac{a_{q-1}}{v_q} \leq \frac{a_q^{1+ \varepsilon}}{ v_{q-1}}$$
where the last inequality follows from \eqref{hyp:ha}.
Finally, we notice that  $\overline{\bX}^q_{\overline{t}_q , \cdot} (y)$ the flow of $w_q$ for $t \in [\overline{t}_{q-1}, \overline{t}_q]$ satisfies
$$  y \in Q_q [2 a_q^{1+ \varepsilon}] \Rightarrow \overline{\bX}^q_{\overline{t}_q, t} (y) \in Q_q [2 a_q^{1+ \varepsilon}] $$
for any $t \in [\overline{t}_{q-1}, \overline{t}_q]$.  Therefore, the property $\bX^n_{\overline{t}, t} (x) \in Q_q [a_q^{1+ \varepsilon}]$ for any $t \in [\overline{t}_{q-1} , \overline{t}_{q}]$ can be proved similarly as before studying the difference between $\bX^n_{\overline{t}, t} (x) - \overline{\bX}^n_{\overline{t}_q, t} (y) $ choosing $y = \bX^n_{\overline{t}, \overline{t}_q} (x)$. 

To prove \eqref{eq:prop:item3} we define
$$ \overline{s}_q = \overline{s}_{q, \beta} = \overline{t} + \sum_{k=q}^n t_{k, k-1} + \sum_{k=q+1}^{n+1}  \left ( \frac{a_{k-2}}{a_{k}} - \beta_k \right ) \overline{\tau}_k + \frac{t_{\text{glue}, n}}{2}$$
where $\overline{t} \in J_n$ and $\beta_k$ are the corresponding values of $\overline{t}$ and we set $s_{-1, \beta} \equiv T$ for any $\beta$. Notice that $\overline{s}_{q+1, \beta} \leq \overline{s}_{q, \beta}  $ for any $q$. By \eqref{d:T_n} we observe that 
$$ T_n - \overline{s}_q - s= \sum_{k=1}^{q-1} t_{k, k-1} + \sum_{k=1}^{q} \left ( \frac{a_{k-2}}{a_{k}} - \beta_k \right ) \overline{\tau}_k - s$$
for any $s \geq 0$ and $q \leq n$ and observing that the previous can be rewritten as 
$$ T_n - \overline{s}_q - s = \sum_{k=1}^{q-1} t_{k, k-1} + \sum_{k=1}^{q} \tilde{\beta}_k \overline{\tau}_k - s $$
for some $\tilde{\beta}_k$ with $4 \frac{a_{k-2}^{1+ \varepsilon}}{a_{k}} \leq \tilde{\beta}_{k} \leq  \frac{a_{k-2}}{a_{k}} - 4 \frac{a_{k-2}^{1+\varepsilon}}{a_{k}}$, we can repeat the previous proof substituting $\overline{t}_q$ with $\overline{s}_q$ using the symmetry of the velocity field \eqref{d:symmetry:VF}.
%

Now we prove \eqref{prop:item:nablab}. Thanks to the streamline property of $w_q$ we have that 
$$ y \in Q_q [2 a_q^{1+ \varepsilon}] \Rightarrow  \overline{\bX}^q_{\overline{t}_q, t} (y) \in  Q_{q} [2 a_{q}^{1+ \varepsilon}] \qquad \text{for any } t \in [\overline{t}_{q-1}, \overline{t}_q] \,.$$
 {Then} thanks to \eqref{eq:stability:X-overlineX} we have that 
$\bX^n_{\overline{t}, t} (x) \in Q_q [a_q^{1+ \varepsilon}]$  for any $t \in [\overline{t}_{q-1}, \overline{t}_{q}]$ and for any $q$. Therefore, 
 thanks  to  
Lemma \ref{lemma:building-block} and \eqref{eq:stability:X-overlineX} we have
$$\int_{\overline{t}_{q-1}}^{\overline{t}_q} \| \nabla w_q  \|_{L^\infty (B_{a_q^{1+ \varepsilon}} (\gamma(s)))} ds \leq \int_{\overline{t}_{q-1}}^{\overline{t}_q} \| \nabla w_q  \|_{L^\infty (B_{2 a_q^{1+ \varepsilon}} (\overline{\gamma}(s)))} ds \leq 
 C\,,$$
  where we used the notation $\gamma(s) = \bX^n_{\overline{t}, s} (x)$ and $\overline{\gamma} (s) =  \overline{\bX}^q_{\overline{t}_q, s} (y) $ with $y =  \bX^n_{\overline{t}, \overline{t}_q} (x)$. This concludes the proof, since the constant $C>0$ is given in Lemma \ref{lemma:building-block} and is independent on all the parameters.
\end{proof}

\section{Proof of Theorem \ref{thm:NU_ADE_loops}}

{Let $\solTE_n$ and $\solAD_n$ denote the solution to \eqref{eq:TE}, respectively \eqref{eq:ADE} with velocity field $b_n$.}
To show that there are at least two different solutions of \eqref{eq:ADE} with $b$, 
and we show that the sequence $\theta_n$ is close (in a suitable sense) to $\rho_n$ and using that $\{ \rho_{2n} \}_{n \in \N}$ and $\{ \rho_{2n +1} \}_{n \in \N}$ have two different limit points in the weak* topology we deduce that $\{ \theta_{2n} \}_{n \in \N}$ and $\{ \theta_{2n +1} \}_{n \in \N}$ also have. The property of being a parabolic solution follows from the weak* lower semi-continuity of the norms {and the uniform boundedness of the sequence $\{\solAD_n \}_n$ in $L^1$, as will be shown in the proof of Theorem \ref{thm:NU_ADE_loops}}.

More precisely, to show that the solution of the \eqref{eq:ADE} is close to the solution of the \eqref{eq:TE} with the regularized velocity field $b_n$ it is enough
 to show that the backward (or forward) stochastic flow is close to the backward (or forward) regular Lagrangian  flow. It is possible to prove this closeness result between the regular Lagrangian flow and the stochastic flow when the flows go from small scales to big scales, but it is not clear whether a similar closeness result can be proved when the flows go from large scales to small scales.
Therefore, in Section \ref{sec:stability} we prove that the backward and the forward regular Lagrangian flows are close to the backward and forward stochastic  flows respectively, starting from any $x \in S_{n+1} [a_{n+1}^{1+ \varepsilon}]$, i.e. the smallest pipe of $b_n$ (corresponding to the smallest scale of the velocity field). 
In Section \ref{sec:distinct} we use the previous closeness result of the flows to prove quantitative ``separation of supports'' (see $(iv)$ in Lemma \ref{lem:claims}) of $\{ \theta_{2n} \}_{n \in \N}$ and $\{ \theta_{2n +1} \}_{n \in \N}$. 
Finally, in Section \ref{sec:concluding-proof} we use all the previous properties to conclude the proof.

\subsection{Stability between the Stochastic and the Lagrangian flow} \label{sec:stability}

%

\begin{lem}[Backward stability]\label{lem:loopstabback}
For any $n \in \N$ and $\overline{t}\in (0, T)$, let  $\bX^n_{\bar{t}, \cdot }$ and $\bY^n_{\bar{t}, \cdot }$ be the backward Lagrangian flow and the backward stochastic flow  of $\bb_n$ and recall the set $S_{n, \text{glue}}$ defined in \eqref{eq:prop:item2} of Proposition \ref{prop:basicprop}. Then, there exists $\Omega_{n} \subset \Omega$ such that 
$\mathbb{P} (\Omega_{n}) \geq 1 - a_0^{1+ \varepsilon}$ with the following property.  For any $t_{n, \beta} \in J_n$ defined in \eqref{eq:J-n} and for any $\omega \in \Omega_{n}$ and $x \in S_{n, \text{glue}} [ a_{n+1}^{1+ \varepsilon}]  $ we have 
$$ | \bX^n_{{t}_{n, \beta},t}(x)-\bY^n_{{t}_{n, \beta},t}(x,\omega) | \leq a_{q}^{1+ \varepsilon} \qquad \forall t \in [\bar{t}_{q-1, \beta}, \bar{t}_{q, \beta}]$$
where $\{ \bar{t}_{q, \beta} \}_{q}$ is the sequence given in \eqref{eq:prop:item2} of Proposition \ref{prop:basicprop}.
In particular, the following holds
 \begin{equation*}
        \sup_{t \in [0,{t}_{n, \beta}]}  \sup_{x \in S_{n, glue} [ a_{n+1}^{1+ \varepsilon}] } \E \big[ |\bX^n_{{t}_{n, \beta},t}(x)-\bY^n_{{t}_{n, \beta},t}(x,\omega)| \big] \le 2 a_0^{1+\varepsilon}  \,.
    \end{equation*}
\end{lem}

\begin{proof}

The proof is by iteration and for convenience we abuse the notation and set $S_{n+1} = S_{n, glue}$. Firstly, we define the set $\Omega_n \subset \Omega$. We recall the set $B_n$ for any $n \in \N$ of multi-indices given in \eqref{eq:J-n} and  we define for any $q \leq n$
and we slightly abuse the notation with $\beta = \pi_i (\beta) \in B_i$ for any $\beta \in B_j$ with $i \leq j$ where $\pi_i (\beta) = (\beta_1, \ldots, \beta_{i+1}) $ for any $\beta = (\beta_1, \ldots, \beta_{j+1}) \in B_j$. Finally, as in Proposition \ref{prop:basicprop} for any ${t}_{n, \beta} \in J_n$  there exists $\{ \overline{t}_{q, \beta} \}_{q}$ and we define  
\begin{align} \label{d:t-beta-max}
\bar{t}_{q, \beta, \text{max}} = \bar{t}_{q, \beta} + \sum_{k = q+2}^{n+1}  \frac{a_{k-2}}{a_{k}} \bar{\tau}_{k} = \sum_{k=1}^q t_{k, k-1}  + \sum_{k=1}^{q+1} \beta_k \overline{\tau}_k  + \sum_{k = q+2}^{n+1}  \frac{a_{k-2}}{a_{k}} \bar{\tau}_{k}   \,,
\end{align}
satisfying 
\begin{align} \label{eq:bound-difference-time}
| \bar{t}_{q, \beta, \text{max}} - \bar{t}_{q-1, \beta}  | \leq  |\bar{t}_{q, \beta} - \bar{t}_{q-1, \beta}| +\sum_{k= q+2}^{n+1}  \frac{a_{k-2}}{a_{k}} \bar{\tau}_{k}  \leq 6\frac{a_{q-1}}{v_{q}} + \sum_{k= q+2}^{n+1}  \frac{a_{k-2}}{a_{k}} \frac{a_k}{v_{k-1}} \leq  8 \frac{a_{q-1}}{v_q} \,,
\end{align} 
where we used \eqref{eq:sum-a_k}, \eqref{d:v-q} with the fact that $\alpha >0$ together with \eqref{d:a-q}. Furthermore, we have $|t_{n, \beta} - \bar{t}_{n, \beta}| \leq 8 \frac{ a_{n}}{v_{n+1}}$, from the definition of $t_{n, \text{glue}}$ \eqref{eq:table}.

Finally, we set 
\begin{equation}\label{eq:goodproba}
\Omega_n  = \bigcap_{q =0}^{n+1} \bigcap_{\beta \in B_q} \left  \{\omega \in \Omega :  \, \sup_{t\in [\bar{t}_{q-1, \beta}, \bar{t}_{q, \beta, \text{max}}]} \sqrt{2} \vert \bW_t - \bW_{\bar{t}_{q}}\vert < \frac{a_{q}^{1+\varepsilon}}{4K} \right \}  \,,
\end{equation}
with the notation $\overline{t}_{n+1, \beta , \text{max}} = t_{n, \beta}$ and  $K = \exp (C)$ and $C>0$ is the constant given in \eqref{eq:gronwall}. 
 By Doob inequality, \eqref{eq:cardinality} and \eqref{eq:bound-difference-time} we can estimate 
\begin{align*}
\mathbb{P} (\Omega_{n}^c) & \leq \sum_{q =0}^{n+1} \sum_{\beta \in B_q} \exp \left (- \frac{a_q^{2+ 2 \varepsilon}}{16 K^2 (\bar{t}_{q,\beta , \text{max}} - \bar{t}_{q-1 , \beta})} \right ) 
\\
& \leq  \sum_{q=0}^{n+1} a_{q+1}^{-2}   \exp \left (- \frac{a_q^{2+ 2 \varepsilon} v_q}{128 K^2 a_{q-1} } \right )  
{= \sum_{q=0}^{n+1} a_{q+1}^{-2}   \exp \left (- \frac{a_{q-1}^{(2+2\varepsilon)(1+\delta)-\alpha(1+\delta)-1}}{128 K^2} \right )}\\
& \leq {C\sum_{q =0 }^{n+1} a_{q+1}^{-2} a_{q+1}^{3 + 2\varepsilon} 
=  C}\sum_{q =0 }^{n+1} a_{q+1}^{1+ {2}\varepsilon} 
{\leq C a_0 \sum_{ q \ge 0} a_0^{(1 + 2\varepsilon)(1+ \delta)^q} 
\leq C a_0^{2+ \varepsilon} \sum_{ q \ge 0} a_0^{\varepsilon(1+ \delta)^q}  }\\
& \leq  \frac{a_0^{1+ \varepsilon}}{2} \,. \numberthis \label{eq:POmega_n_c}
\end{align*}
{In} the last inequality we used \eqref{eq:sum-a_k}{. Going from the second to the third line,} we used
  $\exp (-x) \leq \frac{k!}{x^k}$ and with $k \in \N $ {large enough} so that
  $$ k ( \alpha  (1 +\delta)  + 1- (2+ 2 \varepsilon ) (1+ \delta)) > (3+ 2 \varepsilon) (1+ \delta)^2{.}$$
  {This} inequality is possible thanks to \eqref{hyp:ha} and  \eqref{hyp:epsilon} which imply $\alpha  (1 +\delta)  + 1- (2+ 2 \varepsilon ) (1+ \delta) >0$ and we used $a_0$ {small enough} to reabsorb the constant ${4C = 4 k!(128 K^2)^k}$.

From now on in the proof, we use 
the shorthand notation 
$$\bar{t}_{q} = \bar{t}_{q, \beta} \,, \quad \text{ and } \quad  \bar{t}_{q, \text{max}} = \bar{t}_{q, \beta, \text{max}}\,. $$
Let us fix $\overline{t}= t_{n, \beta} \in J_n$ and fix use the short hand notation $\overline{t}_q = \overline{t}_{q, \beta}$ for the sequence
$\{ \overline{t}_{q, \beta} \}_{q =1}^n$ given by \eqref{eq:prop:item2} of Proposition \ref{prop:basicprop} and set $\overline{t}_{n+1} = \overline{t}$ and $\overline{t}_0 = 0$.  We claim that  
$$| \bX^n_{\bar{t}, \bar{t}_q} (x) - \bY^n_{\bar{t}, \bar{t}_q}(x, \omega) | \leq a_{q+1}^{1+ \varepsilon}\,, \quad \forall \omega \in {\Omega}_{\bar{t}} \,,$$
for any $x \in S_{n+1}[ a_{n+1}^{1+\varepsilon}]$  and  $q \leq n$,
where 
\begin{equation*}
        \Omega_{\overline{t}} = \left \{\omega \in \Omega : \forall \quad 0 \le q \le n+1, \quad \sup_{t\in [\bar{t}_{q-1},\bar{t}_{q}]} \sqrt{2} \vert \bW_t - \bW_{\bar{t}_{q}}\vert < \frac{a_{q}^{1+\varepsilon}}{4K} \right \} \supset \Omega_n \,,
    \end{equation*}
    and the last inclusion holds from the definitions of \eqref{d:t-beta-max} and \eqref{eq:goodproba}. 
Suppose the property is true up to some $q \geq 1$, then we want to prove the property for $q-1$. 
We define 
$$\tau (\omega) = \sup \{ \bar{t}_{q-1} \le  s \le \bar{t}_q : | \bX^n_{\bar{t},s} (x) - \bY^n_{\bar{t},s}  (x, \omega) | >  {a_{q}^{1+ \varepsilon}} \} \le \bar{t}_q \,.$$
The proof follows by proving that $\tau(\omega) \le \overline{t}_{q-1}$ for any $\omega \in {\Omega}_{\overline{t}} $. Suppose $\tau(\omega) >\overline{ t}_{q-1}$. As $a_{q+1}^{1+\varepsilon} < a_q^{1+\varepsilon}/4K$, we have that 
\begin{align*}
| \bY^n_{\bar{t}, \tau } (x, \omega) - \bX^n_{\bar{t} , \tau} (x) | & \leq
| \bY^n_{\bar{t}, \bar{t}_q} (x, \omega) - \bX^n_{\bar{t} , \bar{t}_q} (x) | + \left | \int_{\tau(\omega)}^{\bar{t}_q} \bb_n (\bY^n_{\bar{t}, s} (x, \omega)) - \bb_n (\bX^n_{\bar{t}, s} (x)) ds \right | 
\\
& \quad +   \sqrt{2} \sup_{t \in [\bar{t}_{q-1},\bar{t}_q]} |\bW_t (\omega) - \bW_{\bar{t}_q}(\omega)|
\\
& \leq  \frac{2 a_{q}^{1+ \varepsilon} }{4K} +
  \int_{\tau(\omega)}^{t_{q}} \| \nabla w_{q} \|_{L^\infty (B_{a_{q}^{1+ \varepsilon}} (\gamma (s)) )} | \bY^n_{\bar{t}, s } (x, \omega) - \bX^n_{\bar{t}, s} (x) | ds 
  \\
  & \quad + (\overline{t}_{q} - \overline{t}_{q-1} ) v_{q-1}   
  \\
  & \leq \frac{3a_q^{1+ \varepsilon}}{4K} + \int_{\tau(\omega)}^{t_{q}} \| \nabla w_{q} \|_{L^\infty (B_{a_{q}^{1+ \varepsilon}} (\gamma (s)) )} | \bY^n_{\bar{t}, s } (x, \omega) - \bX^n_{\bar{t}, s} (x) | ds  \,,
\end{align*}
for any $\omega \in \Omega_{\bar{t}}$, where in the last we used $ (\overline{t}_{q} - \overline{t}_{q-1} ) v_{q-1}   \leq 6 \frac{a_{q-1}}{v_q} v_{q-1} \leq  \frac{a_q^{1 + \varepsilon}}{4K} $ thanks to property \ref{eq:prop:item2} of Proposition \ref{prop:basicprop} and \eqref{hyp:epsilon}  and  we used the notation $\gamma (s ) =\bX^n_{\bar{t} , s} $. By {Gr\"onwall's} lemma and property  \ref{prop:item:nablab} of Proposition \ref{prop:basicprop} we have 
\begin{align*}
| \bY^n_{\bar{t}, \tau } (x, \omega) - \bX^n_{\bar{t} , \tau} (x) | \leq \frac{3a_{q}^{1 + \varepsilon}}{4} 
\end{align*}
contradicting $\tau (\omega) > \overline{t}_{q-1}$. 
Finally, {from} this property and the boundedness of the trajectories we get the thesis.
\end{proof}

\begin{remark}\label{rmk:large_small}
The same proof does not hold from large scales to small scales, more precisely it is not possible to fix $x \in S_0$ and prove that for the forward flow $\bX_{0,  \cdot}^n$ the following holds
$$| \bX^n_{0, \bar{t}_q} (x) - \bY^n_{0, \bar{t}_q}(x, \omega) | \leq \frac{a_q^{1+ \varepsilon}}{2 K} \,, \quad \forall \omega \in {\Omega}_{\overline{t}} \,,$$
since already in the 0-th pipe $Q_0$ the error of the Brownian motion can be bounded just by $\frac{a_0^{1 + \varepsilon}}{4K} \gg a_q^{1+ \varepsilon}$. Going from small scales to large scales is better since the error made by the Brownian motion in the small scales is negligible in the large scales.
\end{remark}

Similarly, using property \eqref{eq:prop:item3} of Proposition \ref{prop:basicprop} and replacing $\Omega_{n}$ in the proof with  
    \begin{equation*}
        \tilde{\Omega}_{n}= \bigcap_{q =0}^{n+1} \bigcap_{\beta \in B_q} \left  \{\omega \in \Omega : \, \sup_{s \in [\bar{s}_{q, \beta, \text{min}},\bar{s}_{q-1, \beta}]} \sqrt{2} \vert \bW_s - \bW_{\overline{s}_{q}}\vert < \frac{a_q^{1+\varepsilon}}{4K} \right  \},
    \end{equation*} 
      where $\{ \bar{s}_{ q, \beta} \}_q$ is the one given  in property \eqref{eq:prop:item3} of Proposition \ref{prop:basicprop}, with the notation $\bar{s}_{n+1, \beta, \text{min}} = t_{n, \beta}$ and 
   $$\bar{s}_{q, \beta, \text{min}} = t_{n, \beta} + \sum_{k=q+1}^n t_{k, k-1} + \sum_{k=q+2}^n  \left ( \frac{a_{k-2}}{a_{k}} - \beta_k \right ) \overline{\tau}_k + \frac{t_{\text{glue}, n}}{2}$$ 
    and  replacing $\Omega_{\overline{t}}$ with 
\begin{equation*}
        \Omega_{\overline{s}} = \left  \{\omega \in \Omega : \forall \quad 0 \le q \le n+1, \quad \sup_{s\in [\bar{s}_{q, \beta},\bar{s}_{q-1, \beta}]} \sqrt{2} \vert \bW_s - \bW_{\bar{s}_{q}}\vert < \frac{a_{q}^{1+\varepsilon}}{4K} \right \} \supset \Omega_n \,,
    \end{equation*}
     then
one can prove the following.
\begin{lem}(Forward stability)\label{lem:loopstabfor}
    For any $n \in \N$, $\bar{t} \in {[0,T]}$, let $\bX^n_{\bar{t}, \cdot}$ and $\bY^n_{\bar{t}, \cdot }$ be the forward Lagrangian flow and the forward stochastic flow  of $\bb_n$ and recall the set $S_{n, \text{ glue}}$ defined in \eqref{eq:prop:item2} of Proposition \ref{prop:basicprop}. 
    Then, there exists $\tilde{\Omega}_{n} \subset \Omega$ such that 
$\mathbb{P} (\tilde{\Omega}_{n}) \geq 1 - a_0^{1+ \varepsilon}$ with the following property. For any $t_{n, \beta} \in J_n$ and  $\omega \in \tilde{\Omega}_{n}$ and $x \in S_{n, glue} [ a_{n+1}^{1+ \varepsilon}]  $ we have 
$$ | \bX^n_{{t}_{n, \beta},s}(x)-\bY^n_{{t}_{n, \beta},s}(x,\omega) | \leq a_{q}^{1+ \varepsilon} \qquad \forall s \in [\bar{s}_{q, \beta}, \bar{s}_{q-1, \beta}]$$
where $\{ \bar{s}_{q, \beta} \}_{q}$ is the sequence given in \eqref{eq:prop:item3} of Proposition \ref{prop:basicprop}.
 In particular, the following holds
 \begin{equation*}
        \sup_{s \in [t_{n, \beta}, T_n]}  \sup_{x \in S_{n, glue} [ a_{n+1}^{1+ \varepsilon}] } \E \big[ |\bX^n_{\bar{t},s}(x)-\bY^n_{\bar{t},s}(x,\omega)| \big] \le  a_0^{1+\varepsilon}  \,.
    \end{equation*}
    
\end{lem}

\subsection{Distinct solutions to the Advection--Diffusion equation} \label{sec:distinct}

The goal of this section is to use the backward stability of Lemma \ref{lem:loopstabback} and the forward flow stability of Lemma \ref{lem:loopstabfor} together both from the smallest scale $S_{n, \text{glue}}$ of the velocity field $\bb_n$ to get a full behavior of the solution $\solAD_n $ of \eqref{eq:ADE} with { a positive} initial datum 
$\ADin$ and velocity field $\bb_n$ in the time interval $[0, T_n]$.
{We choose the initial datum $\ADin$ such that 
\begin{equation}\label{eq:initial}
		\begin{cases}
			\supp (\ADin) \subset R_0 + B_{a_0^{1+\varepsilon/2}}(0)\,, \text{ where } R_0 \text{ is the starting zone given in \eqref{eq:R0}}\,,\\
			\ADin \ge 0 \text{ on } \T^2, \qquad \ADin \equiv 1 \text{ on } R_0,  \qquad \|\nabla \ADin \|_{L^\infty }\le a_0^{-1-\varepsilon/2} \text{ and } \| \ADin \|_{L^\infty} \leq 1.
		\end{cases}
\end{equation}
}
 To prove the result we  decompose the solution $\solAD_n$  in a suitable way, and do the gluing on each piece.

\begin{lem}\label{lem:claims}
Let $b_n \in L^\infty ((0,1) ; BV (\T^2)) \cap L^\infty$ be the divergence free velocity field defined in \eqref{d:symmetry:VF} and $\ADin \in C^\infty$ be the initial datum and $\solAD_n$ be the corresponding solution to the advection{--}diffusion.
    For every $n \ge 0$, there is a sequence $(\solAD_{n,k})_{k=1,\ldots, N_n}$ with $\frac{9 a_0^2}{10 a_n a_{n+1}}  \leq N_n \leq \frac{a_0^2}{a_n a_{n+1}}$ and a time $T_{n} < T$ such that  the following hold true
    \begin{align*}
        (i) \quad &\solAD_{n,k} \ge 0 &&\forall k=1,...,N_n\,.\\
        (ii) \quad & \sum_{k=1}^{N_n} \solAD_{n,k} (t,x) \le \solAD_{n} (t,x) &&\forall (t,x) \in [0,T]\times \T^2\,.\\
        (iii) \quad & \sum_{k=1}^{N_n} \int_{\T^2}\solAD_{n,k} (t,x) dx \ge \frac{4}{5} \int_{\T^2} \ADin (x) dx &&\forall t \in ( T_{n}, T)\,.\\
        (iv) \quad & \sum_{k=1}^{N_n} \int_{\{x,y {\ge} 0\}} \solAD_{n,k}(x,t) dx \ge \frac{2}{3}\int_{\T^2}\ADin(x) dx &&\forall t \in ( {T}_{n}, T)\,, \qquad n \text{ even}\,,\\
        & \sum_{k=1}^{N_n} \int_{\{x,y {\le} 0\}} \solAD_{n,k}(x,t) dx \ge \frac{2}{3}\int_{\T^2}\ADin(x) dx &&\forall t \in ( T_n , T)\,, \qquad n \text{ odd}\,.
    \end{align*}
\end{lem}

\begin{proof}
From \eqref{eq:J-n}, setting $N_n := \# B_n = \#J_n$, we identify the sequence $\{ t_{n , \beta} \}_{\beta \in B_n}$ with $\{ {t}_{k} \}_{k =1}^{N_n}$, namely there exists a bijective map 
$$T:  \{ t_{n , \beta} \}_{\beta \in B_n} \to \{ {t}_{k} \}_{k =1}^{N_n}$$
 so that $ {t}_{k+1} >  {t}_{k} $ for any $k$. From \eqref{eq:cardinality} we get
 \begin{align} \label{eq:bound-N-n}
 \frac{9 a_0^2}{10 a_n a_{n+1}}  \leq N_n \leq \frac{a_0^2}{a_n a_{n+1}} \,.
\end{align}  
    
     We define by induction  
    \begin{equation}\label{eq:chunkzero}
        \solAD_{n,1} =
        \begin{cases}
            \text{the solution of }
            \begin{cases}
                \partial_t \solAD_{n,1} + b_n\cdot \nabla \solAD_{n,1} = \Delta \solAD_{n,1}\,,\\
                \solAD_{n,1} (t_1,\cdot) = \one_{\SS} \cdot \solAD_{n}(t_1,\cdot)\,.
            \end{cases}
            &\text{ if } t\ge t_1\,,\\
            0 &\text{ if } t < t_1 \,,
        \end{cases}
    \end{equation}
    where we use the shorthand notation $\SS:=S_{n , \text{glue}}[2a_n^{1+\varepsilon}]$, with $S_{n , \text{glue}}$ defined in Proposition \ref{prop:basicprop}. Then 
    $$0 \le \solbAD[n,1] := \solAD_n -\solAD_{n,1} \le \solAD_n \le 1\,$$
    {thanks to the maximum principle for parabolic PDEs (or by the Feynman-Kac representation formula \eqref{eq:FKbackward}) and $\ADin \le 1$.}
    Now for every $k = 2,\ldots,N_{n-1}$ we set
    \begin{equation} \label{eq:chunks}
        \solAD_{n,k+1} = 
        \begin{cases}
            \text{the solution of }
            \begin{cases}
                \partial_t \solAD_{n,k+1} + \bb_n\cdot \nabla \solAD_{n,k+1} = \Delta \solAD_{n,k+1}\,,\\
                \solAD_{n,k+1} (t_{k+1},\cdot) = \one_{\SS} \cdot \solbAD[n,k](t_{k+1},\cdot)\,.
            \end{cases}
            &\text{ if } t\ge t_{k+1}\,,\\
            0 &\text{ if } t < t_{k+1}\,,
        \end{cases}
    \end{equation}
    and we define  $\solbAD[n,k] = \solbAD[n,k-1] - \solAD_{n,k} $, from which we get
    \begin{equation} \label{eq:sum-solrest}
    \solAD_n = \sum_{j=1}^k\solAD_{n,j} + \overline{\theta}_{n,k} \,.
    \end{equation}
    {We} deduce 
     the pointwise bound
    \begin{equation}\label{eq:solrest}
        0 \le \solbAD[n,k] = \solbAD[n,k-1] - \solAD_{n,k} = \solAD_n - \sum_{j = 1}^k\solAD_{n,j} \le \solAD_n \le 1 \,.
    \end{equation}
    Notice that the first inequality holds since it holds at the starting time $t=t_{k+1}$ and $\solAD_{n,k}$ solves the advection--diffusion equation.

    \textbf{Proof of $(i)$ and $(ii)$:}
    We prove these two inequalities together by induction. First notice that for every $(t,x) \in [0,T]\times \T^2, \solAD_n (t,x) \ge 0$ as it solves the advection--diffusion equation with positive initial datum $\ADin \geq 0$. Therefore 
    \[0\le \solAD_{n,1}(t,x) \le \solAD_n(t,x) \qquad \forall (t,x) \in [0,T]\times \T^2 \,. \]
    Suppose that for some $1 \le N \le N_n-1$ we have $\solAD_{n,k}\ge 0$ for all $1 \le k \le N$ and
    $\sum_{k=1}^N\solAD_{n,k} \le \solAD_n $ then 
    $$0\le \solbAD[n,N] \le \solAD_{n}\,.$$ Therefore, from \eqref{eq:chunks} we have  
    $$0\le \solAD_{n,N+1} \le \solbAD[n,N]$$ and also
    \[\sum_{k=1}^{N+1}\solAD_{n,k} \le \sum_{k=1}^{N}\solAD_{n,k} + \solbAD[n,N] \overset{\eqref{eq:sum-solrest}}{=} \solAD_n\]
    from which we conclude the proof.
    
     \textbf{Inequality $(iii)$:}
Let $T_n$ be defined in \eqref{d:T_n}    
    Let $t \in ( T_n, T)$. Let us fix $t_k \in J_n$. By the Backward Feynman-Kac formula \eqref{eq:FKbackward} and conservation of the average of solutions to \eqref{eq:ADE} with divergence free velocity fields  we have
    \begin{align*}
        \intTd \solAD_{n,k} (t,x) dx &= \intTd \solAD_{n,k} (t_{k},x) dx = \int_{\SS} \solbAD[n,k] (t_{k},x) dx\\
        &=\int_\Omega\int_{\SS}\ADin\big(\bY_{t_{k},0}^n(x,\omega)\big) dx d\P(\omega) - \int_{\SS} \sum_{j=1}^{k} \solAD_{n,j}(t_{k},x) dx =: A - B
    \end{align*}
    Quantity $A$ represents the mass brought up to the middle by the $k$-th piece of the splitting. Quantity $B$ stands for the mass that goes back into $S_{n , \text{glue}}$ after a certain time. We now lower bound $A$ using the backward stability Lemma \ref{lem:loopstabback}  and upper bound $B$ using the forward stability Lemma \ref{lem:loopstabfor} together with properties of the deterministic flow given by Proposition \ref{prop:basicprop}. 
    \\
    \\
    Using property \eqref{eq:flow-backward-R-0}, property $(ii)$ of this Lemma, the backward stability Lemma \ref{lem:loopstabback}  and the property on the initial datum \eqref{eq:initial} 
    \begin{align*}
        A &= \int_{\SS} \ADin\big(\bX_{t_{k},0}^n \big) dx + \int_\Omega \int_{\SS} \Big( \ADin \big(\bY_{t_{k},0}^n \big) - \ADin \big( \bX_{t_{k},0}^n \big)\Big) dx d\P(\omega)\\
        &\ge \L^2(S_{n , \text{glue}}) - \|\nabla \ADin \|_{L^\infty} \L^2(S_{n , \text{glue}} ) \sup_{\SS} \int_\Omega \big\vert \bX_{t_{k},0}(x) - \bY_{t_{k},0}(x,\omega) \big \vert d\P(\omega)  - 2 a_{n+1}^{2+ 2\varepsilon} \\
        &\ge (1-  a_0^{\varepsilon/2})\L^2(S_{n , \text{glue}})\,.
    \end{align*} 
    
     Recall the set $\tilde{\Omega}_n \subset \Omega$ given in the forward stability Lemma \ref{lem:loopstabfor}.
   From the  forward Feynman-Kac formula \eqref{eq:FKforward} we split
    \begin{align*}
        B  &= \sum_{j=1}^{k} \int_{\T^2} \one_{\SS}(x)\solAD_{n,j}(t_{k},x) dx = \sum_{j=1}^{k} \int_{\Omega} \int_{\T^2} \one_{\SS} \big(\bY_{t_{j},t_{k}}(x,\omega)\big) \solAD_{n,j}(t_{j},x) dx d\P(\omega)
        \\
        & \leq \sum_{j=1}^{k} \int_{\tilde{\Omega}_n} \int_{\SS} \one_{S_{n , \text{glue}}} \big(\bY_{t_{j},t_{k}}(x,\omega)\big)   dxd\P(\omega)  + \int_{\tilde{\Omega}_n^c} \int_{\T^2} \one_{\SS} \big(\bY_{t_{j},t_{k}}(x,\omega)\big) \solAD_{n,j}(t_{j},x)  dxd\P(\omega)
        \\
        & \quad + a_{n+1}^{2+ 2 \varepsilon}
        =: B' + B'' + a_{n+1}^{2+ 2 \varepsilon} \,,
    \end{align*}
   where 
    in the first inequality we used the definition of the initial datum for $\solAD_{n, j}$ at time $t_j$.
   From Lemma \ref{lem:loopstabfor}   and property
   \eqref{eq:prop:item3} of Proposition \ref{prop:basicprop} on $\bX_{t_j, t_i}$ we have that $B'=0$. 
We now estimate $B''$. Thanks to $\P(\tilde{\Omega}_n ^c) \le  a_0^{1+\varepsilon}$ by Lemma \ref{lem:loopstabfor}, properties  $(i), (ii)$ which imply $\solAD_n(t_{k-1},\cdot) - \sum_{j=1}^{k-1} \solAD_{n,j}(t_{k-1},\cdot)  \geq 0$ and the forward Feynman Kac formula \eqref{eq:FKforward},  we have
    \begin{align*}
        B'' &\le \int_{\tilde{\Omega}_n^c} \int_{\T^2}  \one_{\SS} \big(\bY_{t_{k-1},t_k}\big) \left( \solAD_n(t_{k-1},\cdot) - \sum_{j=1}^{k-1} \solAD_{n,j}(t_{k-1},\cdot)\right)
        + \sum_{j=1}^{k-1}\one_{\SS} \big(\bY_{t_j,t_k}\big) \solAD_{n,j}(t_j,\cdot) dx d\P(\omega)\\
        &\le \| \solAD_n (t_k, \cdot ) \|_{L^\infty (\T^2)} \P(\tilde{\Omega}_n^c)\L^2(S_{n , \text{glue}}) 
        \\
        & \quad + \int_{\tilde{\Omega}_n^c} \int_{\T^2}  -\one_{\SS}(x) \sum_{j=1}^{k-1} \solAD_{n,j}(t_{k},\cdot)
        + \sum_{j=1}^{k-1}\one_{\SS}(x)\solAD_{n,j}(t_k,\cdot) dx d\P(\omega)\\
        &\le a_0^{1+\varepsilon} \L^2(S_{n , \text{glue}})
    \end{align*}
    where in the last inequality we used the maximum principle and $\| \theta_{in} \|_{L^\infty} \leq 1$.
%
    Finally, {from \eqref{eq:bound-N-n}, the definition of $S_{n,\text{glue}}$ in Proposition \ref{prop:basicprop} and $R_0 = x_0 + [0,a_0]^2$, we deduce
    \[ N_n \L^2(S_{n , \text{glue}}) \ge \frac{9}{10} \L^2(R^0)(1- 16 a_n^{\varepsilon}).\]}
    {Summarizing the above estimates, we obtain}
    \begin{align*}
        \sum_{k=1}^{N_n} \int_{\T^2}\solAD_{n,k} (t,x) dx 
        &\ge N_n \L^2(S_{n , \text{glue}})(1- a_0^{1+\varepsilon}- a_0^{\varepsilon/2} - a_{n+1}^{2 \varepsilon} )\\
        &\ge \frac{9}{10} \L^2(R_0)(1- a_0^{1+\varepsilon}- a_0^{\varepsilon/2} - a_{n+1}^{2 \varepsilon} ){(1-a_n^\varepsilon)} \\
        &\ge \frac{4}{5} \int_{\T^2} \ADin(x)dx \,,
    \end{align*}
    where we used smallness of $a_0$.

    \textbf{Claim $(iv)$ :} 
    Suppose, that $n$ is even. Let $t \in (T_n, T)$.
    Consider a smooth function $\chi $ supported on $\{x_1,x_2 > 0 \}$, such that 
    \begin{equation*}
          \chi(x)\vert_{\{x_1, x_2 \ge 2 a_0\}} \equiv 1\,, \qquad \|\chi\|_{L^\infty(\T^2)}\le 1 \,, \qquad \|\nabla \chi \|_{L^\infty} \le a_0^{-1-\varepsilon/2}\,.
    \end{equation*}
    Recalling the definition of $ \solAD_{n,i}$ in \eqref{eq:chunks} we decompose 
    \begin{align*}
        \sum_{k=1}^{N_n} \int_{\{x_1,x_2 \ge a_0\}} \solAD_{n,k}(t,x)
        \geq \sum_{k=1}^{N_n} \int_{\T^2} \big(\solAD_{n,k} - \solTE_{n,k}\big) \chi(x) dx 
        +\sum_{k=1}^{N_n} \int_{\T^2} \solTE_{n,k} \chi(x) dx =:A + B
    \end{align*}
    where  $\solTE_{n,k}$ is given as the solution of
    \begin{equation}\label{eq:Lagchunks}
        \begin{cases}
            \partial_t \solTE_{n,k}(t,x) + b_n\cdot \nabla \solTE_{n,k}(t,x) = 0\,, &t \ge t_k\,, \\
            \solTE_{n,k}(t_k,\cdot ) = \solAD_{n,k}(t_k,\cdot)\,,
        \end{cases}
    \end{equation}
    with $\solAD_{n,k}\equiv 0$ for $t < t_k$  for any $k= 1,\ldots,N_n$. Thanks to \eqref{eq:prop:item3} of Proposition \ref{prop:basicprop} and property $(iii)$ we get 
    \begin{align*}
        B = \sum_{k=1}^{N_n} \int_{\T^2} \solTE_{n,k}(t,x) \chi(x) dx
        = \sum_{k=1}^{N_n} \int_{\T^2} \solTE_{n,k}(t,x) dx =    \sum_{k=1}^{N_n} \int_{\T^2} \solAD_{n,k}(t_k,x) dx
        \ge \frac{4}{5} a_0^2 
    \end{align*}
    On the other hand, using the forward Feynmann-Kac formula  \eqref{eq:FKforward}  and stability estimate \ref{lem:loopstabfor}  we bound
    \begin{align*}
        |A| &\le \sum_{k=1}^{N_n}\int_\Omega\intTd \big\vert \chi\big(\bY_{t_k,t}(x,\omega)\big) - \chi\big(\bX_{t_k ,t}(x)\big) \big\vert \underbrace{\solAD_{n,k}(t_k,x)}_{\le \one_{\SS}(x)} dx d\P(\omega) \\
        &\le \sum_{k=1}^{N_n} \|\nabla \chi \|_{L^\infty} \L^2 (\SS) \sup_{x\in\SS} \int_\Omega \big\vert \bY_{t_k,t}(x,\omega) - \bY_{t_{k},t}(x)\big\vert d\P(\omega) \\
        &\le N_n a_0^{-1-\varepsilon/2} \L^2(S_{n , \text{glue}}) a_0^{1+\varepsilon} \le a_0^{2+ \varepsilon/2} \,.
    \end{align*}
    Therefore, using {that $\int_{\T^2}\ADin dx \sim a_0^2$ by the properties of $\ADin$ given in} \eqref{eq:initial} we get  
    $$\sum_{k=1}^{N_n} \int_{\{x,y {\ge} a_0\}} \solAD_{n,k}(x,t) dx \ge \left(\frac{4}{5} -  a_0^{\varepsilon/2}\right)a_0^2\ge \frac{2}{3}\int_{\T^2}\ADin(x) dx \,,$$
    for all $t \in (T_n , T)$ for any $n $ even.
        The proof for $n$ odd is similar, using corresponding property \eqref{eq:prop:item3} of Proposition \ref{prop:basicprop} for $n $ odd we get 
    $$\sum_{k=1}^{N_n} \int_{\{x,y {\leq} - a_0\}} \solAD_{n,k}(x,t) dx \ge  \frac{2}{3}\int_{\T^2}\ADin(x) dx \,.$$
\end{proof}

\subsection{Proof of Theorem \ref{thm:NU_ADE_loops}} \label{sec:concluding-proof}

{
The two distinct solutions to \eqref{eq:ADE} will be weak-$\ast$ limits of the sequence $\{\solAD_{2n}\}_n$, respectively $\{\solAD_{2n+1}\}_n$. Firstly, we ensure that such  solutions are indeed parabolic solutions thanks to an Aubin--Lions type argument, see \cites{A63,L69}.
\begin{lem}\label{lem:parabolic_limit}
Let $\{\solAD_n\}_n$  be the sequence of unique bounded solutions to \eqref{eq:ADE} with divergence-free velocity fields $b_n \in L^\infty_{t,x}$ and initial data $\ADin \in L^\infty$.  Let $\theta \in L^\infty$ be any $L^\infty$-weak-$\ast$ limit of a subsequence of $\{\solAD_n\}_n$. Suppose that $b_n \to b $ in $L^1_{t,x}$. Then, $\theta$ is a parabolic solution of the \eqref{eq:ADE} as defined in Definition \ref{d:parabolic}. 
\end{lem}}
\begin{proof}
     {Recall the initial datum $\theta_{\rm in} \in C^\infty(\T^2)$ and recall that for any $n \geq 0$,  $b_n \in L^2([0,T]\times \T^2))$ is divergence free. Thanks to \cite[Theorem 3.3]{BCC24} we know} that there is a unique solution $\theta_n \in L^\infty_{t,x} \cap L^2_t H^1_x$ to \eqref{eq:ADE} that satisfies the local energy equality in the sense of distributions : 
	\begin{align} \label{eq:energy-equality}
 		\partial_t |\theta_n|^2 + \diver \left (\bb_n \frac{|\theta_n|^2}{2} \right ) = \Delta \frac{|\theta_n|^2}{2} - |\nabla \theta_n|^2 \,.
	\end{align}
    {We do an Aubin-Lions type argument to show} that any $L^\infty_{t,x}$ weak* converging subsequence $\{ \theta_{n_k} \}_{n_k}$ converges strongly in $L^{{s}} ((0, T) \times \T^2)$ for any ${s} < \infty$. Since the sequence is uniformly bounded in $L^\infty$ it is enough to show strong convergence in $L^1$.  We claim that  for any $\varepsilon>0$ there exists $\delta >0$ and $N \in \N$ such that  
    $$ \| \theta_{n_j} - \theta_{n_k} \|_{L^1} \leq \| \theta_{n_k} - \theta_{n_k} \star \varphi_\delta \|_{L^1} + \| \theta_{n_j} \star \varphi_\delta - \theta_{n_k} \star \varphi_\delta  \|_{L^1} + \| \theta_{n_j} - \theta_{n_j} \star \varphi_\delta \|_{L^1} < \varepsilon \,,$$
    for any $n_k, n_j \geq N$,
   where $\varphi_\delta$ is a space-time rescaled mollifier. Thanks to the fact that the sequence $\{\theta_n \}_{n}$ is uniformly bounded in {$L^2 ( (0,T); H^1(\T^2))  $} (with a constant depending only on  $ \| \ADin \|_{ L^2}$  and $T< \infty$), we can now fix $\delta >0$ sufficiently small  so that 
   $$  \| \theta_{n_k} - \theta_{n_k} \star \varphi_\delta \|_{L^1}  +\| \theta_{n_j} - \theta_{n_j} \star \varphi_\delta \|_{L^1} { \leq c \delta ( \| \nabla \theta_{n_k} \|_{L^2}  + \| \nabla \theta_{n_j} \|_{L^2}  ) } < \frac{2 \varepsilon}{3} \,,$$ 
{   for any $n_k, n_j  \in \N$.  From now on $\delta >0$ is fixed. }
   From the equation we have 
   $$\| \partial_t \theta_n  \star \varphi_\delta  \|_{L^1} \leq \| \diver (b_n \theta_n)  \star \varphi_\delta  \|_{L^1}  + \| \Delta \theta_n  \star \varphi_\delta \|_{L^1} \leq C(\delta) \| \theta_n \|_{L^\infty} \| b_n \|_{L^1} \,.$$
   Using also that $\| \nabla \theta_n \star \varphi_{\delta} \|_{L^1_{t,x}} \leq \| \nabla \theta_n \star \varphi_{\delta} \|_{L^2_{t,x}} \leq \| \solAD_{\text{in}} \|_{L^2}$ we deduce by Sobolev embedding that $\theta_{n_k}  \star \varphi_{\delta} \to \theta  \star \varphi_{\delta} $ strongly in $L^1$ as $n_k \to \infty$. Therefore, we can find $N \in \N$ sufficiently large so that
   $$ \| \theta_{n_j} \star \varphi_\delta - \theta_{n_k} \star \varphi_\delta  \|_{L^1} < \frac{\varepsilon}{3} \,,$$
   for any $n_k , n_j \geq N$.
	We now fix a weak* converging subsequence  $\{ \theta_{n_k} \}_{n_k}$ and using that it is strongly converging in $L^{{s}} $ for any ${s}< \infty$ and that  $b_n \to b $ in $L^{{s'}}$ for some  ${s'}>1$ we deduce from \eqref{eq:energy-equality} passing into the distributional limit the local energy equality \eqref{eq:energy-equality}. We notice that the inequality  in the {limit} is due to the weak lower semicontinuity of  $| \nabla \theta_{n_k}|^2$.
\end{proof}

We are now ready to conclude the proof.

\begin{proof}[Proof of Theorem \ref{thm:NU_ADE_loops}]
    We now show that there are at least two distinct solutions.
    By Proposition \ref{prop:basicprop}, $b_n \to b$ in $L^1_{t,x}$ so that $(\solAD_{2n})_n$ and $(\solAD_{2n+1})_n$ converge weakly-$\ast$ (up to some non-relabelled subsequence) to some $\solADe,\solADo \in L^\infty([0,T],\T^2)$ solving the original \ref{eq:ADE}{, which are also parabolic thanks to Lemma \ref{lem:parabolic_limit}.} Thanks to the weak* convergence and to $(ii)$ and $(iv)$ of Lemma \ref{lem:claims} we have 
    \begin{align*}
        \int_{\{x,y \ge a_0\}} \solADe (x,t) dx  \ge  \lim_{n \to \infty} \sum_{k\in J_n}\int_{\{x,y \ge a_0\}} \solAD_{2n,k} (x,t) dx 
        \ge \frac{2}{3}\intTd \ADin (x) dx \,.
    \end{align*}
    for any $t \in (\lim_{n\to \infty}T_n, T)$.
 Similarly,
    \begin{align*}
        \int_{\{x,y \le - a_0\}} \solADo (x,t) dx \ge \frac{2}{3}\intTd \ADin (x) dx\,,
    \end{align*}
    for any $t \in (\lim_{n\to \infty}T_n, T)$, 
    which implies $\solADe \neq \solADo$ by conservation of the average.
\end{proof}

\section{The velocity field in $L^{p}_t L^{\infty}_{x}$ }

{The geometric construction of the velocity field comes directly from \cite[Section~2.1]{CCS23}, with some simplifications.} We avoid the convolution in space and we rescale time to get $\bb \in L^p_tL^\infty_x$ for $p <2$.
{We sketch here its construction and refer the interested reader to \cite{CCS23} for additional details}.

\subsection{{Main ideas of the construction}} \label{subsec:idea-chess}

{ We start by considering the advection equation.
	Consider an initial datum that is a chessboard function of size $a_0$ and $1$-periodic. More precisely, we define the chessboard function of size $a_0$ as the function that takes the value $-1$ on black squares of the chessboard of size $a_0$ and $1$ on white squares of the chessboard of size $a_0$ (see Figure \ref{fig:chess_mixing} for a visualization of a chessboard function). The goal is to use a suitable building block velocity field so that the solution to the advection equation starting at time zero as a  chessboard function of size $a_0$  becomes  a chessboard function of size $a_1 \ll a_0$ in very short time. The quantification of the short time will be crucial in the proof of Theorem \ref{thm:NU_ADE_chess} to treat $\Delta$ as a small perturbation.
	 We will then  iterate a  suitable rescaled versions of the building block in disjoint highly concentrated time intervals to get finer and finer chessboard functions of size\footnote{The sequence $\{ a_q \}_q$ is a super-exponential monotone decreasing sequence that we will use in the proof.} $ a_q $  along an increasing sequence of times $\{t_q\}_q\subset [0,T/2]$.
To achieve a non--uniqueness result we reflect the velocity field defined on $[0, T/2]$ and we define $b(t,x) = - b(T-t, x)$ for any $t \in [T/2, T]$. Furthermore, using the idea in \cite{CCS23} we also add infinitely many swapping velocity fields on the time interval $[T/2, T]$ to prove that there are at least two distinct solutions of the advection equation satisfying $\theta_{\text{even}} (T, \cdot ) = - \theta_{\text{odd}} (T, \cdot) \neq 0$.
	 }
	 
	{ 
	 More precisely, we consider  $\{t_q\}_q\subset [0,T/2]$  a monotone sequence that converges at a super-exponential rate, where $T = 2 \lim_{q\to \infty} t_q$. 
	and  we design our vector field $b$ so that the following holds true.
	\begin{itemize}
		\item The regular Lagrangian flow of the velocity field restricted on the time-interval $[t_q, t_{q+1}[$\footnote{We remark that the velocity field is  divergence free and $L^\infty_{loc} ([0,T] \setminus \{T/2\} ; BV )$, so it admits a regular Lagrangian flow on each connected time interval that does not contains $t=T/2$.} rearranges a chessboard function of size $a_q$ into a chessboard function of size $a_{q+1}$ as sketched in Figure \ref{fig:chess_mixing} (mixing step).  
		\item 
		The regular Lagrangian flow of the velocity field restricted on the time-interval $[T-t_{q+1}, T-t_{q}[$  rearranges a chessboard function of size $a_{q+1}$ into a chessboard function of size $a_{q}$ where black and white are swapped as sketched in Figure \ref{fig:chess_unmixing} (unmixing step). Indeed, within this time interval, we include a swapping velocity field, which is useful for obtaining non-unique solutions to the advection equation that remain qualitatively unchanged under the addition of  $\Delta$ to the equation.
	\end{itemize}
	\begin{remark}
The main difference between our construction and the one in  \cite{CCS23} is the choice of the parameter $\gamma >0$ that quantifies the concentration in time of the building block so that subsequent times $\{ t_q \}_q$ satisfy $t_{q+1} - t_q = 3 a_q^\gamma$. In our case $\gamma \sim 2$ and $\gamma >2$ whereas in \cite{CCS23} the authors choose $\gamma \sim 0$ and $\gamma >0$. We are forced to choose high concentration in time $\gamma >2$ in order to treat the Laplacian as a perturbation in the equation and prove that the heuristics picture we provide here is qualitatively preserved also for the advection diffusion equation solutions.
\end{remark}
}
\begin{figure}[htbp]
	\[
	\begin{tikzpicture}[scale=0.6]
		\draw (0,0) rectangle (4cm, 4cm);
	\end{tikzpicture}
	\quad
	\begin{tikzpicture}[scale=0.6]
		\draw (0,0) rectangle (4cm, 4cm);
		\foreach \y in {1,1.5,3,3.5}{
			\draw (0,\y)--(4,\y);
			\draw[-stealth](1,0.125+\y)--(3,0.125+\y);
			\draw (0,0.25+\y)--(4,0.25+\y);
			\draw[stealth-](1,0.375+\y)--(3,0.375+\y);}
		\foreach \y in {0,0.5,2,2.5}{
			\draw (0,\y)--(4,\y);
			\draw[stealth-](1,0.125+\y)--(3,0.125+\y);
			\draw (0,0.25+\y)--(4,0.25+\y);
			\draw[-stealth](1,0.375+\y)--(3,0.375+\y);}
	\end{tikzpicture}
	\quad
	\begin{tikzpicture}[scale=0.6]
		\draw (0,0) rectangle (4cm, 4cm);
		\foreach \x in {0,0.5,1,1.5,2,2.5,3,3.5}{
			\draw (\x,0)--(\x,4);
			\draw (0.25+\x,0)--(0.25+\x,4);}
		\foreach \x in {0.25,0.5,1,1.75,2.25,2.5,3,3.75}{
			\draw[-stealth](0.125+\x,1)--(0.125+\x,3);}	
	\end{tikzpicture}
	\]
	\[	
	\begin{tikzpicture}[scale=0.6][x=1cm]
		\draw (0,0) rectangle (4cm, 4cm);
		\foreach \y in {0,2}{
			\foreach \x in {0,2}{
				\fill (\x,\y) rectangle (1+\x,1+\y) rectangle (2+\x,2+\y);}}
	\end{tikzpicture}
	\,
	\begin{tikzpicture}[scale=0.6][x=1cm]
		\draw (0,0) rectangle (4cm, 4cm);
		\foreach \y in {0,2}{
			\foreach \x in {0,2}{
				\fill (\x,\y) rectangle (1+\x,1+\y) rectangle (2+\x,2+\y);}}
	\end{tikzpicture}
	\,
	\begin{tikzpicture}[scale=0.6][x=1cm]
		\draw (0,0) rectangle (4cm, 4cm);
		\foreach \y in {0,0.5,1,1.5,2,2.5,3,3.5}{
			\foreach \x in {0}{
				\fill (\x,\y) rectangle (0.5+\x,0.25+\y) rectangle (1.5+\x,0.5+\y);
				\fill (1.5+\x,\y) rectangle (2.5+\x,0.25+\y) rectangle (3.5 +\x,0.5+\y);
				\fill (3.5 + \x,\y) rectangle (4+\x,0.25+\y);}}
	\end{tikzpicture}
	\,
	\begin{tikzpicture}[scale=0.6][x=1cm]
		\draw (0,0) rectangle (4cm, 4cm);
		\foreach \y in {0,0.5,1,1.5,2,2.5,3,3.5}{
			\foreach \x in {0,0.5,1,1.5,2,2.5,3,3.5}{
				\fill (\x,\y) rectangle (0.25+\x,0.25+\y) rectangle (0.5+\x,0.5+\y);}}
	\end{tikzpicture}
	\]
	\caption{{ In the  upper part of the Figure we represent the definition of the vector field in three different time sub-intervals of $[t_q, t_{q+1}]$.  The left lower figure is the solution to the advection equation at time $t= t_q$. 
	In the the lower part of the Figure from the second figure on we represent the associated solution to the advection equation at the end of the three subintervals of $[t_q, t_{q+1}]$
	 under the action of the vector field depicted above.}}
	\label{fig:chess_mixing}
\end{figure}

\begin{figure}[htbp]
	\[
	\begin{tikzpicture}[scale=0.6]
		\draw (0,0) rectangle (4cm, 4cm);
		\foreach \x in {0,0.5,1,1.5,2,2.5,3,3.5}{
			\draw (\x,0)--(\x,4);
			\draw (0.25+\x,0)--(0.25+\x,4);}
		\foreach \x in {0.25,0.5,1,1.75,2.25,2.5,3,3.75}{
			\draw[stealth-](0.125+\x,1)--(0.125+\x,3);}	
	\end{tikzpicture}
	\quad
	\begin{tikzpicture}[scale=0.6]
		\draw (0,0) rectangle (4cm, 4cm);
		\foreach \y in {1,1.5,3,3.5}{
			\draw (0,\y)--(4,\y);
			\draw[stealth-](1,0.125+\y)--(3,0.125+\y);
			\draw (0,0.25+\y)--(4,0.25+\y);
			\draw[-stealth](1,0.375+\y)--(3,0.375+\y);}
		\foreach \y in {0,0.5,2,2.5}{
			\draw (0,\y)--(4,\y);
			\draw[-stealth](1,0.125+\y)--(3,0.125+\y);
			\draw (0,0.25+\y)--(4,0.25+\y);
			\draw[stealth-](1,0.375+\y)--(3,0.375+\y);}
	\end{tikzpicture}
	\quad
	\begin{tikzpicture}[scale=0.6]
		\draw (0,0) rectangle (4cm,4cm);
		\foreach \y in {1,2,3}{
			\draw[-stealth] (1,\y)--(3,\y);}
	\end{tikzpicture}
	\]
	\[
	\begin{tikzpicture}[scale=0.6][x=1cm]
		\draw (0,0) rectangle (4cm, 4cm);
		\foreach \y in {0,0.5,1,1.5,2,2.5,3,3.5}{
			\foreach \x in {0,0.5,1,1.5,2,2.5,3,3.5}{
				\fill (\x,\y) rectangle (0.25+\x,0.25+\y) rectangle (0.5+\x,0.5+\y);}}
	\end{tikzpicture}
	\,
	\begin{tikzpicture}[scale=0.6][x=1cm]
		\draw (0,0) rectangle (4cm, 4cm);
		\foreach \y in {0,0.5,1,1.5,2,2.5,3,3.5}{
			\foreach \x in {0}{
				\fill (\x,\y) rectangle (0.5+\x,0.25+\y) rectangle (1.5+\x,0.5+\y);
				\fill (1.5+\x,\y) rectangle (2.5+\x,0.25+\y) rectangle (3.5 +\x,0.5+\y);
				\fill (3.5 + \x,\y) rectangle (4+\x,0.25+\y);}}
	\end{tikzpicture}
	\,
	\begin{tikzpicture}[scale=0.6][x=1cm]
		\draw (0,0) rectangle (4cm, 4cm);
		\foreach \y in {0,2}{
			\foreach \x in {0,2}{
				\fill (\x,\y) rectangle (1+\x,1+\y) rectangle (2+\x,2+\y);}}
	\end{tikzpicture}
	\,
	\begin{tikzpicture}[scale=0.6][x=1cm]
		\draw (0,0) rectangle (4cm, 4cm);
		\foreach \y in {0,2}{
			\foreach \x in {0,2}{
				\fill (\x,2+\y) rectangle (1+\x,1+\y) rectangle (2+\x,\y);}}
	\end{tikzpicture}
	\]
	\caption{{
	In the  upper part of the Figure we represent the definition of the vector field in three different time sub-intervals of $[T-t_{q+1}, T-t_{q}]$. The third upper figure represents the swapping velocity field, that is just a constant in space velocity field. The left lower figure is the solution to the advection equation at time $t= T-t_{q+1}$. In the the lower part of the Figure from the second figure on we represent the associated solution to the advection equation at the end of the three subintervals of $[T-t_{q+1}, T-t_{q}]$ under the action of the vector field depicted above. }}
	\label{fig:chess_unmixing}
\end{figure}

\begin{figure}
	\[  
	\begin{tikzpicture}[scale=0.013][x=1cm]
		\draw (0,0) rectangle (64cm,64cm);
		\node at (0,32)[left]{$\solAD_0$};
		\foreach \y in {0,32}{
			\foreach \x in {0,32}{
				\fill (\x,\y) rectangle (16+\x,16+\y) rectangle (32+\x,32+\y);}}
	\end{tikzpicture}
	\,
	\begin{tikzpicture}[scale=0.013][x=1cm]
		\draw[opacity=0.2] (0,0) rectangle (64cm,64cm);
		\foreach \y in {0,32}{
			\foreach \x in {0,32}{
				\fill[opacity=0.2] (\x,\y) rectangle (16+\x,16+\y) rectangle (32+\x,32+\y);}}
	\end{tikzpicture}
	\,
	\begin{tikzpicture}[scale=0.013][x=1cm]
		\draw[opacity=0.2] (0,0) rectangle (64cm,64cm);
		\foreach \y in {0,32}{
			\foreach \x in {0,32}{
				\fill[opacity=0.2] (\x,\y) rectangle (16+\x,16+\y) rectangle (32+\x,32+\y);}}
	\end{tikzpicture}
	\,
	\begin{tikzpicture}[scale=0.013][x=1cm]
		\draw[opacity=0.2] (0,0) rectangle (64cm,64cm);
		\foreach \y in {0,32}{
			\foreach \x in {0,32}{
				\fill[opacity=0.2] (\x,\y) rectangle (16+\x,16+\y) rectangle (32+\x,32+\y);}}
	\end{tikzpicture}
	\,
	\begin{tikzpicture}[scale=0.013][x=1cm]
		\draw[opacity=0.2] (0,0) rectangle (64cm,64cm);
		\foreach \y in {0,32}{
			\foreach \x in {0,32}{
				\fill[opacity=0.2] (\x,\y) rectangle (16+\x,16+\y) rectangle (32+\x,32+\y);}}
	\end{tikzpicture}
	\,\cdots\,
	\begin{tikzpicture}[scale=0.013][x=1cm]
		\draw[opacity=0.2] (0,0) rectangle (64cm,64cm);
		\foreach \y in {0,32}{
			\foreach \x in {0,32}{
				\fill[opacity=0.2] (\x,\y) rectangle (16+\x,16+\y) rectangle (32+\x,32+\y);}}
	\end{tikzpicture}
	\,
	\begin{tikzpicture}[scale=0.013][x=1cm]
		\draw[opacity=0.2] (0,0) rectangle (64cm,64cm);
		\foreach \y in {0,32}{
			\foreach \x in {0,32}{
				\fill[opacity=0.2] (\x,\y) rectangle (16+\x,16+\y) rectangle (32+\x,32+\y);}}
	\end{tikzpicture}
	\,
	\begin{tikzpicture}[scale=0.013][x=1cm]
		\draw[opacity=0.2] (0,0) rectangle (64cm,64cm);
		\foreach \y in {0,32}{
			\foreach \x in {0,32}{
				\fill[opacity=0.2] (\x,\y) rectangle (16+\x,16+\y) rectangle (32+\x,32+\y);}}
	\end{tikzpicture}
	\,
	\begin{tikzpicture}[scale=0.013][x=1cm]
		\draw[opacity=0.2] (0,0) rectangle (64cm,64cm);
		\foreach \y in {0,32}{
			\foreach \x in {0,32}{
				\fill[opacity=0.2] (\x,\y) rectangle (16+\x,16+\y) rectangle (32+\x,32+\y);}}
	\end{tikzpicture}
	\,
	\begin{tikzpicture}[scale=0.013][x=1cm]
		\draw (0,0) rectangle (64cm,64cm);
		\foreach \y in {0,32}{
			\foreach \x in {0,32}{
				\fill (\x,\y) rectangle (16+\x,16+\y) rectangle (32+\x,32+\y);}}
	\end{tikzpicture}
	\]
	\[
	\begin{tikzpicture}[scale=0.013][x=1cm]
		\draw (0,0) rectangle (64cm,64cm);
		\node at (0,32)[left]{$\solAD_1$};
		\foreach \y in {0,32}{
			\foreach \x in {0,32}{
				\fill (\x,\y) rectangle (16+\x,16+\y) rectangle (32+\x,32+\y);}}
	\end{tikzpicture}
	\,
	\begin{tikzpicture}[scale=0.013][x=1cm]
		\draw (0,0) rectangle (64cm,64cm);
		\foreach \y in {0,8,16,24,32,40,48,56}{
			\foreach \x in {0,8,16,24,32,40,48,56}{
				\fill (\x,\y) rectangle (4+\x,4+\y) rectangle (8+\x,8+\y);}}
	\end{tikzpicture}
	\,
	\begin{tikzpicture}[scale=0.013][x=1cm]
		\draw[opacity=0.2] (0,0) rectangle (64cm,64cm);
		\foreach \y in {0,8,16,24,32,40,48,56}{
			\foreach \x in {0,8,16,24,32,40,48,56}{
				\fill[opacity=0.2] (\x,\y) rectangle (4+\x,4+\y) rectangle (8+\x,8+\y);}}
	\end{tikzpicture}
	\,
	\begin{tikzpicture}[scale=0.013][x=1cm]
		\draw[opacity=0.2] (0,0) rectangle (64cm,64cm);
		\foreach \y in {0,8,16,24,32,40,48,56}{
			\foreach \x in {0,8,16,24,32,40,48,56}{
				\fill[opacity=0.2] (\x,\y) rectangle (4+\x,4+\y) rectangle (8+\x,8+\y);}}
	\end{tikzpicture}
	\,
	\begin{tikzpicture}[scale=0.013][x=1cm]
		\draw[opacity=0.2] (0,0) rectangle (64cm,64cm);
		\foreach \y in {0,8,16,24,32,40,48,56}{
			\foreach \x in {0,8,16,24,32,40,48,56}{
				\fill[opacity=0.2] (\x,\y) rectangle (4+\x,4+\y) rectangle (8+\x,8+\y);}}
	\end{tikzpicture}
	\,\cdots \,
	\begin{tikzpicture}[scale=0.013][x=1cm]
		\draw[opacity=0.2] (0,0) rectangle (64cm,64cm);
		\foreach \y in {0,8,16,24,32,40,48,56}{
			\foreach \x in {0,8,16,24,32,40,48,56}{
				\fill[opacity=0.2] (\x,\y) rectangle (4+\x,4+\y) rectangle (8+\x,8+\y);}}
	\end{tikzpicture}
	\,
	\begin{tikzpicture}[scale=0.013][x=1cm]
		\draw[opacity=0.2] (0,0) rectangle (64cm,64cm);
		\foreach \y in {0,8,16,24,32,40,48,56}{
			\foreach \x in {0,8,16,24,32,40,48,56}{
				\fill[opacity=0.2] (\x,\y) rectangle (4+\x,4+\y) rectangle (8+\x,8+\y);}}
	\end{tikzpicture}
	\,
	\begin{tikzpicture}[scale=0.013][x=1cm]
		\draw[opacity=0.2] (0,0) rectangle (64cm,64cm);
		\foreach \y in {0,8,16,24,32,40,48,56}{
			\foreach \x in {0,8,16,24,32,40,48,56}{
				\fill[opacity=0.2] (\x,\y) rectangle (4+\x,4+\y) rectangle (8+\x,8+\y);}}
	\end{tikzpicture}
	\,
	\begin{tikzpicture}[scale=0.013][x=1cm]
		\draw (0,0) rectangle (64cm,64cm);
		\foreach \y in {0,8,16,24,32,40,48,56}{
			\foreach \x in {0,8,16,24,32,40,48,56}{
				\fill (\x,\y) rectangle (4+\x,4+\y) rectangle (8+\x,8+\y);}}
	\end{tikzpicture}
	\,
	\begin{tikzpicture}[scale=0.013][x=1cm]
		\draw (0,0) rectangle (64cm,64cm);
		\foreach \y in {0,32}{
			\foreach \x in {0,32}{
				\fill (16+\x,\y) rectangle (32+\x,16+\y);}}
		\foreach \y in {0,32}{
			\foreach \x in {0,32}{
				\fill (\x,16+\y) rectangle (16+\x,32+\y);}}
	\end{tikzpicture}
	\]
	\[
	\begin{tikzpicture}[scale=0.013][x=1cm]
		\draw (0,0) rectangle (64cm,64cm);
		\node at (0,32)[left]{$\solAD_2$};
		\foreach \y in {0,32}{
			\foreach \x in {0,32}{
				\fill (\x,\y) rectangle (16+\x,16+\y) rectangle (32+\x,32+\y);}}
	\end{tikzpicture}
	\,
	\begin{tikzpicture}[scale=0.013][x=1cm]
		\draw (0,0) rectangle (64cm,64cm);
		\foreach \y in {0,8,16,24,32,40,48,56}{
			\foreach \x in {0,8,16,24,32,40,48,56}{
				\fill (\x,\y) rectangle (4+\x,4+\y) rectangle (8+\x,8+\y);}}
	\end{tikzpicture}
	\,
	\begin{tikzpicture}[scale=0.013][x=1cm]
		\draw (0,0) rectangle (64cm,64cm);
		\foreach \y in {0,2,4,6,8,10,12,14,16,18,20,22,24,26,28,30,32,34,36,38,40,42,44,46,48,50,52,54,56,58,60,62}{
			\foreach \x in {0,2,4,6,8,10,12,14,16,18,20,22,24,26,28,30,32,34,36,38,40,42,44,46,48,50,52,54,56,58,60,62}{
				\fill (\x,\y) rectangle (1+\x,1+\y) rectangle (2+\x,2+\y);}}
	\end{tikzpicture}
	\,
	\begin{tikzpicture}[scale=0.013][x=1cm]
		\filldraw[opacity=0.2] (0,0) rectangle (64cm,64cm);
	\end{tikzpicture}
	\,
	\begin{tikzpicture}[scale=0.013][x=1cm]
		\filldraw[opacity=0.2] (0,0) rectangle (64cm,64cm);
	\end{tikzpicture}    
	\,\cdots \,
	\begin{tikzpicture}[scale=0.013][x=1cm]
		\filldraw[opacity=0.2] (0,0) rectangle (64cm,64cm);
	\end{tikzpicture}
	\,
	\begin{tikzpicture}[scale=0.013][x=1cm]
		\filldraw[opacity=0.2] (0,0) rectangle (64cm,64cm);
	\end{tikzpicture}    
	\,
	\begin{tikzpicture}[scale=0.013][x=1cm]
		\draw (0,0) rectangle (64cm,64cm);
		\foreach \y in {0,2,4,6,8,10,12,14,16,18,20,22,24,26,28,30,32,34,36,38,40,42,44,46,48,50,52,54,56,58,60,62}{
			\foreach \x in {0,2,4,6,8,10,12,14,16,18,20,22,24,26,28,30,32,34,36,38,40,42,44,46,48,50,52,54,56,58,60,62}{
				\fill (\x,\y) rectangle (1+\x,1+\y) rectangle (2+\x,2+\y);}}
	\end{tikzpicture}
	\,
	\begin{tikzpicture}[scale=0.013][x=1cm]
		\draw (0,0) rectangle (64cm,64cm);
		\foreach \y in {0,8,16,24,32,40,48,56}{
			\foreach \x in {0,8,16,24,32,40,48,56}{
				\fill (\x,\y) rectangle (4+\x,4+\y) rectangle (8+\x,8+\y);}}
	\end{tikzpicture}
	\,
	\begin{tikzpicture}[scale=0.013][x=1cm]
		\draw (0,0) rectangle (64cm,64cm);
		\foreach \y in {0,32}{
			\foreach \x in {0,32}{
				\fill (\x,\y) rectangle (16+\x,16+\y) rectangle (32+\x,32+\y);}}
	\end{tikzpicture}
	\]
	\[
	\begin{tikzpicture}[scale=0.013][x=1cm]
		\draw (0,0) rectangle (64cm,64cm);
		\node at (0,32)[left]{$\solAD_3$};
		\foreach \y in {0,32}{
			\foreach \x in {0,32}{
				\fill (\x,\y) rectangle (16+\x,16+\y) rectangle (32+\x,32+\y);}}
	\end{tikzpicture}
	\,
	\begin{tikzpicture}[scale=0.013][x=1cm]
		\draw (0,0) rectangle (64cm,64cm);
		\foreach \y in {0,8,16,24,32,40,48,56}{
			\foreach \x in {0,8,16,24,32,40,48,56}{
				\fill (\x,\y) rectangle (4+\x,4+\y) rectangle (8+\x,8+\y);}}
	\end{tikzpicture}
	\,
	\begin{tikzpicture}[scale=0.013][x=1cm]
		\draw (0,0) rectangle (64cm,64cm);
		\foreach \y in {0,2,4,6,8,10,12,14,16,18,20,22,24,26,28,30,32,34,36,38,40,42,44,46,48,50,52,54,56,58,60,62}{
			\foreach \x in {0,2,4,6,8,10,12,14,16,18,20,22,24,26,28,30,32,34,36,38,40,42,44,46,48,50,52,54,56,58,60,62}{
				\fill (\x,\y) rectangle (1+\x,1+\y) rectangle (2+\x,2+\y);}}
	\end{tikzpicture}
	\,
	\begin{tikzpicture}[scale=0.013][x=1cm]
		\filldraw[opacity=0.6] (0,0) rectangle (64cm,64cm);
	\end{tikzpicture}
	\,
	\begin{tikzpicture}[scale=0.013][x=1cm]
		\filldraw[opacity=0.2] (0,0) rectangle (64cm,64cm);
	\end{tikzpicture}    
	\,\cdots \,
	\begin{tikzpicture}[scale=0.013][x=1cm]
		\filldraw[opacity=0.2] (0,0) rectangle (64cm,64cm);
	\end{tikzpicture}
	\,
	\begin{tikzpicture}[scale=0.013][x=1cm]
		\filldraw[opacity=0.6] (0,0) rectangle (64cm,64cm);
	\end{tikzpicture}    
	\,
	\begin{tikzpicture}[scale=0.013][x=1cm]
		\draw (0,0) rectangle (64cm,64cm);
		\foreach \y in {0,2,4,6,8,10,12,14,16,18,20,22,24,26,28,30,32,34,36,38,40,42,44,46,48,50,52,54,56,58,60,62}{
			\foreach \x in {0,2,4,6,8,10,12,14,16,18,20,22,24,26,28,30,32,34,36,38,40,42,44,46,48,50,52,54,56,58,60,62}{
				\fill (\x,\y) rectangle (1+\x,1+\y) rectangle (2+\x,2+\y);}}
	\end{tikzpicture}
	\,
	\begin{tikzpicture}[scale=0.013][x=1cm]
		\draw (0,0) rectangle (64cm,64cm);
		\foreach \y in {0,8,16,24,32,40,48,56}{
			\foreach \x in {0,8,16,24,32,40,48,56}{
				\fill (4+\x,\y) rectangle (8+\x,4+\y);}}
		\foreach \y in {0,8,16,24,32,40,48,56}{
			\foreach \x in {0,8,16,24,32,40,48,56}{
				\fill (\x,4+\y) rectangle (4+\x,8+\y);}}
	\end{tikzpicture}
	\,
	\begin{tikzpicture}[scale=0.013][x=1cm]
		\draw (0,0) rectangle (64cm,64cm);
		\foreach \y in {0,32}{
			\foreach \x in {0,32}{
				\fill (16+\x,\y) rectangle (32+\x,16+\y);}}
		\foreach \y in {0,32}{
			\foreach \x in {0,32}{
				\fill (\x,16+\y) rectangle (16+\x,32+\y);}}
	\end{tikzpicture}
	\]
	\[
	\begin{tikzpicture}[scale=0.88]
		\node at (-0.8,0){};
		\draw (0,0)--(4.1,0);
		\draw[dotted] (4.1,0)--(5.9,0);
		\draw (5.9,0)--(10,0);
		\draw (0 cm,4pt) -- (0 cm,-4pt) node[above=4pt,scale=0.8]{$0$};
		\draw (5 cm,4pt) -- (5 cm,-4pt) node[above=4pt,scale=0.8]{$T/2$};
		\draw (10 cm,4pt) -- (10 cm,-4pt) node[above=4pt,scale=0.8]{$T$};
		\draw (1 cm,2pt) -- (1 cm,-2pt) node[above=3pt,scale=0.6]{$t_1$};
		\draw (2 cm,2pt) -- (2 cm,-2pt) node[above=3pt,scale=0.6]{$t_2$};
		\draw (3.05 cm,2pt) -- (3.05 cm,-2pt) node[above=3pt,scale=0.6]{$t_3$};
		\draw (4.1 cm,2pt) -- (4.1 cm,-2pt) node[above=3pt,scale=0.6]{$t_4$};
		\draw (5.9 cm,2pt) -- (5.9 cm,-2pt) node[above=3pt,scale=0.6]{$T-t_4$};
		\draw (6.95 cm,2pt) -- (6.95 cm,-2pt) node[above=3pt,scale=0.6]{$T-t_3$};
		\draw (8 cm,2pt) -- (8 cm,-2pt) node[above=3pt,scale=0.6]{$T-t_2$};
		\draw (9 cm,2pt) -- (9 cm,-2pt) node[above=3pt,scale=0.6]{$T-t_1$};
	\end{tikzpicture}
	\]
	\caption{{The first few solutions to the advection--diffusion equation with approximated velocity field $b_n$ depicted at the first and last time-steps. The figure is just an approximation of the solutions, since the Laplacian will make some regularization errors that we will prove to be  small.}}
	\label{fig:chess_approx}
\end{figure}

{
	The strategy is to approximate the velocity field $b$ by a sequence $\{b_n\}_n$ with $b_n \equiv 0$ on the middle part $[t_n, T-t_n]$, such that the associated solution to the advection--diffusion equation $\solAD_n$ can be depicted approximately as in Figure \ref{fig:chess_approx}. We highlight that in this figure the behaviour of the sequence $\{ \solAD_{2n} \}_{n}$ and $\{ \solAD_{2n+1} \}_{n}$ is opposite at time $t=T$ thanks to the swapping time velocity fields. This property will be proved up to some small error given by the Laplacian term in the equation.  Thanks to this property,
	we are able to show that the even and the odd sequences converge to distinct limiting solutions qualitatively described above up to further subsequences.
	More precisely, we exhibit two distinct solutions to the advection--diffusion equation $\solADe$ and $\solADo$ with velocity field $b$ and initial datum $\ADin$, such that 
	\begin{align*}
		\solADe (T, \cdot) & \sim \ADin(\cdot),\\
		\solADo (T, \cdot) & \sim -\ADin(\cdot).
	\end{align*}
}

\subsection{{Choice of the parameters}}

{The tile-sizes are given by}
\begin{equation}
	a_{q+1}:=a_q^{1+\delta} \qquad \text{ and } \qquad \lambda_q := \frac{1}{2a_q}
\end{equation}
with (here $p<2$ is {crucial})
\begin{equation}\label{hyp:deltachess}
	2(1+\delta)^2 < \frac{p}{p-1} 
\end{equation}
{ and $a_0$ sufficiently small depending only on $\delta$ so that
\begin{align} \label{eq:a-0-small}
\sum_{j \geq 0} a_j^{\delta} \leq 2 a_{0}^\delta < \frac{1}{10} \,.
\end{align}}
The time scaling is given by the time-steps
\begin{equation}\label{eq:times}
	t_q := \sum_{k < q}3a_k^\gamma \qquad \text{ and } \qquad T:= 2 \lim_{q\to \infty} t_q
\end{equation}
with (here \eqref{hyp:deltachess} is {crucial})
\begin{align}
	2(1+\delta)^2 <\gamma < \frac{p}{p-1}\,, \tag{H$\gamma$}\label{hyp:gamma} 
\end{align}

{
	\subsection{Construction of the velocity field}
}

{
	The building blocks $\dW, \widetilde{\dW}, \overline{\dW}: \T^2 \to \R^2$ are shear flows defined by 
	\[ \dW (x_1, x_2) = (W(x_2),0), \qquad \widetilde{\dW}(x_1,x_2) = \left( 0, \frac{1 + W(x_1)}{2}\right), \qquad \overline{\dW}(x_1, x_2) = \left( 0, \frac{1-W(x_1)}{2}\right) \]
	where $W:\T \to \R$ is defined as follows
	\[ W(z) = 
	\begin{cases}
		1 & \text{if } z \in [0, 1/2 [,\\
		-1 & \text{if } z \in [1/2, 1[,
	\end{cases}\]
	and extended by periodicity.
}

The various time-intervals are given for every $q \ge 0, i=1,2,3$ by
\begin{align*}\label{eq:timeintervals}
	\I_{q,i} & = [t_q + (i-1) a_q^\gamma, t_q + i a_q^\gamma[\,,\\
	\J_{q,i} & = ]T- t_q - i a_q^\gamma, T- t_q - (i-1) a_q^\gamma]\,.
\end{align*}

{
\subsection{Construction of the velocity field}
	Let us denote by $\lfloor \xi \rfloor$ the largest integer smaller or equal than the real number $\xi$. We define
	\begin{equation}
		w_{q+1,2}(x) = 
		\begin{cases}
			- \half a_{q}^{1-\gamma} \dW((2a_{q+1})^{-1} x) & \text{for } x = (x_1, x_2) \text{ with } \lfloor x_2/a_{q} \rfloor \text{ even}\,,\\
			\half a_{q}^{1-\gamma} \dW((2a_{q+1})^{-1} x) & \text{for } x = (x_1, x_2) \text{ with } \lfloor x_2/a_{q} \rfloor \text{ odd}\,,
		\end{cases}
	\end{equation}
	and analogously
	\begin{equation}
		w_{q+1,3}(x) = 
		\begin{cases}
			a_{q+1} a_{q}^{-\gamma} \overline{\dW}((2a_{q+1})^{-1} x) & \text{for } x = (x_1, x_2) \text{ with } \lfloor (x_1+ a_q/2) /a_{q} \rfloor \text{ even}\,,\\
			a_{q+1} a_{q}^{-\gamma} \widetilde{\dW}((2a_{q+1})^{-1} x) & \text{for } x = (x_1, x_2) \text{ with } \lfloor  (x_1+ a_q/2)/a_{q} \rfloor \text{ odd}\,.
		\end{cases}
	\end{equation}
	We also define the swapping velocity field (actually only shifting the whole figure) as 
	\begin{equation}
		w_{q+1, \text{swap}} (x) = (a_q^{1-\gamma}, 0) \,.
	\end{equation}
	The Lagrangian flow of the swapping velocity field at time $a_q^\gamma$ maps a black and white checkerboard structure of size $a_q$ into a white and black checkerboard structure of size $a_q$ as shown in Figure \ref{fig:chess_unmixing}.
}

For all $q \ge 0, x \in \T^2$, the vector field itself is {assembled together in the following way :}
\begin{equation}
	\begin{split}
		b(t,x) = 
		\begin{cases}
			0 & \forall \, t \in \I_{q,1}\,,\\
			w_{q+1,2} & \forall \, t \in \I_{q,1}\,,\\
			w_{q+1,3} & \forall \, t \in \I_{q,1}\,,
		\end{cases}
	\end{split}
	\qquad \text{and} \qquad
	\begin{split}
		b(t,x) = 
		\begin{cases}
			w_{q+1,\text{swap}} & \forall \, t \in \J_{q,1}\,,\\
			-b(T-t,x) & \forall \, t \in \J_{q,1}\,,\\
			-b(T-t,x) & \forall \, t \in \J_{q,1}\,.
		\end{cases}
	\end{split}
\end{equation}

The regularity estimates (compared to estimates (4.20) in \cite{CCS23}) for the velocity field {$b$} are the following
\begin{equation}\label{eq:b_estimates}
	\begin{split}
		\linf{b}{(\J_{q,1})\times \T^2} = a_q^{1-\gamma}\,,\\
		\linf{b}{(\I_{q,2}\cup \J_{q,2})\times \T^2} = \half a_q^{1-\gamma}\,,\\
		\linf{b}{(\I_{q,3}\cup \J_{q,3})\times \T^2} = a_{q+1} a_q^{-\gamma}\,.
	\end{split}
\end{equation}
{Finally we observe that  $\nabla b$ is a measure and its absolutely continuous part satisfies $\nabla b \equiv 0$ almost everywhere. The last property is different compared to \cite{CCS23} since we avoid convolution in space.}

\subsection{Initial datum and even and odd chessboards} \label{sec:initialdatum}
Define the \textit{unit chessboard} $\solAD_0 : \T^2 \to \R$ by
\begin{equation*}
	\solAD_0(x_1,x_2) = 
	\begin{cases}
		1 & \text{if } \{ x_1 \}, \{ x_2 \}\in [0, 1/2[ \text{ or }  \{ x_1 \}, \{ x_2 \} \in [1/2,1[\,,\\
		-1 & \text{else}
	\end{cases}
\end{equation*}
where $\{\xi\} = \xi - \lfloor \xi \rfloor$ for all $\xi \in \R$ and extend $\solAD_0$ periodically to the torus. Then set $\bar{\solAD}_0(x) = \solAD_0(\lambda_0 x)$ and define
\begin{equation*}
	\ADin = (\bar{\solAD}_0 \star \psi)(x)
\end{equation*}
where $\psi$ is a rescaled mollifier with parameter $a_0^{1+ \delta/2}$, i.e. 
\begin{equation}\label{eq:psi}
	\psi(x) = a_0^{-2-\delta}\tilde{\psi}(a_0^{-1-\delta/2}x) \,,
\end{equation}
with $\tilde{\psi} \in C^\infty_c(\R)$ is a smooth bump function such that
\begin{equation*}
	\supp \tilde{\psi} \subseteq [-2,2]\,, \qquad \int_\R |\tilde{\psi}| = 1\,, \qquad \|\tilde{\psi}\|_{L^\infty(\R)} \le 1\,, \qquad \|\nabla \tilde{\psi} \|_{L^\infty (\R)} \le 1\,.
\end{equation*}

Then define for every $q \ge 0$ the $q$-th \textit{even/odd chessboard} of tile-size $a_q$ by 
\begin{equation}\label{eq:AqBq}
	A_q := \supp\{1+\solAD_0(\lambda_q x)\} \qquad \text{respectively} \qquad B_q := \supp\{1-\solAD_0(\lambda_q x)\}.
\end{equation}

Notice
\begin{equation}\label{eq:initialchess}
	\linf{\ADin}{\T^2} \le 1 \qquad \text{ and } \qquad \linf{\nabla \ADin}{\T^2} \le a_0^{-1-\delta/2}.
\end{equation}

\subsection{Main properties}

We collect all the wanted properties in the following proposition.

\begin{prop}\label{prop:chessproperties}
	For every $1 \le p < 2$, there exists a vector field $\bb$, an initial datum $ \rho_{\rm in}$ satisfying \eqref{eq:initialchess} and a sequence $(\bb_n)_n$ with
	\[ \bb_n = \one_{[0,t_n]\cup [T-t_n,T]}\bb\] for the  sequence $t_n \to T/2$ defined {in \eqref{eq:times}} such that the following properties hold:
	\begin{enumerate}
		\item  The vector field $\bb$ is divergence-free and belongs to $L^p([0,T],L^\infty(\T^2))$ and $\bb_n \to \bb$ in $L^1({[0,T]} \times \T^2)$.
		\item  For every $n \ge 0$, there is a unique solution $\solTE_n \in L^\infty([0,T]\times\T^2)$ to the \eqref{eq:TE} with velocity field $\bb_n$ such that 
		$$ \solTE_{2n} (T, x) - \solTE_{2n +1}  (T, x)  = 2 \,, \qquad \text{for all } x \in A_0 [5 a_0^{1+ \delta /2}] $$
		and 
		$$ \solTE_{2n } (T, x) - \solTE_{2n +1}  (T, x)  = - 2 \,, \qquad \text{for all } x \in   B_0 [5 a_0^{1+ \delta/2}] \,.$$
		\item \label{prop:item:3}
		For any $n \geq 0$, there exists a good set $G_n \subset \T^2$ with $\mathcal{L}^2 (G_n) \geq \frac{9}{10}$ such that for all $x \in G_n$, the backward regular Lagrangian flow $\bX^n_{t_n ,s}(x)$ of $\bb_n$ for $s \leq t_n$ satisfies 
		$$\int_{t_q}^{t_{q+1}} \| \nabla b_n  \|_{L^\infty ( B_{a_{q+1}^{1 + \delta}} (\gamma (s)))}  ds  =0 \,,$$
		for any $q \leq n-1$,
		where we used the shorthand notation $ \gamma (s) = \bX^n_{t_n ,s} (x)$. In other words, the velocity field is constant in a ball of a proper radius around $\gamma(s)$. 
		\item  \label{prop:item:4} For any $n \geq 0$, there exists a good set $\tilde{G}_n \subset \T^2$ with $\mathcal{L}^2 (\tilde{G}_n) \geq \frac{9}{10}$ such that for all $x \in \tilde{G}_n$ the forward regular Lagrangian flow $\bX^n_{T- t_n ,s} (x)$ of $\bb_n$ for $s \geq T - t_n$ satisfies 
		$$\int_{T-t_{q+1}}^{T-t_{q}} \| \nabla b_n  \|_{L^\infty ( B_{a_{q+1}^{1 + \delta}} (\gamma (s)))}  ds  =0 \,,$$
		for any $q \leq n-1$,
		where we used the shorthand notation $ \gamma (s) = \bX^n_{T-t_n ,s} (x)$.
	\end{enumerate}
\end{prop}
\begin{proof}
	
	{ Proof of $(1)$.}
	From \eqref{eq:b_estimates} we directly compute
	\begin{align*}
		\|b_n - b\|_{L^1([0,T]\times \T^2)} &= \int_{t_n}^{T-t_n} \intTd[2] |b(t,x)|dx\,dt\\
		&\le \sum_{q \ge n} \int_{\J_{q,1}} a_q^{1-\gamma} dt + \int_{\I_{q,2}\cup \J_{q,2}} \half a_q^{1-\gamma} dt +  \int_{\I_{q,3}\cup \J_{q,3}} a_{q+1} a_q^{-\gamma} dt\\
		&\le \sum_{q\ge n} a_q + a_q + 2a_{q+1} \le 8a_n \overset{n \to \infty}{\to} 0\,.
	\end{align*}
	The vector field is divergence-free as it is an $L^1$ limit of  divergence free velocity fields $\bb_n$.
	From estimates \eqref{eq:b_estimates}, we obtain
	\begin{align*}
		\|\bb_n \|_{L^p_t L^\infty_x}^p \le 2 \sum_{q=0}^\infty \left(a_q^\gamma (a_q^{1-\gamma})^p + a_q^\gamma \left(\half a_q^{1-\gamma}\right)^p + a_q^\gamma(a_{q+1}a_q^{-\gamma})^p\right) < \infty
	\end{align*}
	where the convergence holds if and only if $\gamma + (1-\gamma)p > 0$, which is implied by \eqref{hyp:gamma}. Therefore, $\bb_n$ is uniformly bounded in $L^p_t L^\infty_x$, then $\bb \in L^p_t L^\infty_x$.
	
	{ Property $(2)$} follows directly from  \cite{CCS23}. { Proof of  properties $(3), (4)$. We define }
	$$G_n = \bigcap_{j=1}^n \{ x :  \bX^n_{t_n, t_j} (x) \in A_j [a_j^{1+\delta}]  \cup B_j [a_j^{1+ \delta}] \}$$
	and { using the measure preserving property of the regular Lagrangian flow we can estimate
	$$ \mathcal{L}^2 (G_n^c) \leq \sum_{j=1}^n \mathcal{L}^2 ((A_j [a_j^{1+\delta}])^c \cap  (B_j [a_j^{1+ \delta}] )^c ) \leq 4 \sum_{j=1}^n a_j^{\delta} \leq \frac{1}{10}\,,  $$
	where in the last we have used  \eqref{eq:a-0-small}.}
  One can similarly define $\tilde{G}_n$.

	The last two properties directly follows from the construction, since the velocity field is at each time a vertical or horizontal shear flow which is constant on a small neighbourhood  of the backward flow for points $x \in G_n$ and on the forward flow for points $x \in \tilde{G}_n$.
\end{proof}

\section{Proof of Theorem \ref{thm:NU_ADE_chess}}

Firstly, we prove some stability results between the forward (or backward) regular Lagrangian flow and the forward (or backward) stochastic flow of the approximated velocity field $b_n$.

\subsection{Stability between the Stochastic and the Lagrangian flow}

\begin{lem}(Backward stability)\label{lem:chessstabback}
    For any $n\ge 0$, let $G_n \subset \T^2$ and $t_n \in [0,T]$ as in Proposition \ref{prop:chessproperties}, and let $\bX^n$ and $\bY^n$ denote the backward regular Lagrangian flow and the backward stochastic flow of $b_n$ respectively. Then, the following inequality holds:
    \begin{equation*}
        \sup_{t\in [0,T - t_n [} \sup_{x \in G_n} \int_{\Omega} |\bX^n_{T- {t}_n,t}(x)-\bY^n_{T- t_n ,t}(x,\omega)| d\P(\omega) \le a_0^{1+\delta}\,.
    \end{equation*}
\end{lem}
\begin{proof}
    Fix some $n \ge 0$ and we define 
    and define the $n$-th good probability set
    \begin{align*}
        \Omega_n  & = :\Omega_{n, 1} \cap \Omega_{n, 2} 
        \\
        & := \left\{ \omega : \sqrt{2}\sup_{t\in [t_{k-1}, t_{k}[} \vert \bW_{t}(\omega) -\bW_{t_{k}}(\omega)\vert  < a_{k}^{1+\delta}/ 2 \qquad \forall 1 \le {k} \le n \right\}
        \\
        & \quad  \cap \left\{ \omega : \sqrt{2}\sup_{t\in [t_{n}, T- t_n [} \vert \bW_{t}(\omega) -\bW_{t_{n}}(\omega)\vert  < a_{n}^{1+\delta}/2   \right\}\,.
    \end{align*}
     Thanks to Doob's inequality \eqref{eq:translatedDoob},
    \begin{equation}\label{eq:P_Omegaq}
    \begin{aligned}
        \P(\Omega_{n}^c) &\le \P(\Omega_{n,1}^c) + \P(\Omega_{n,2}^c)\\
        &\le \sum_{k=0}^{n} \exp\left( -\frac{a_{k+1}^{2(1+\delta)}}{16(3a_k^\gamma)}\right) + \exp\left( -\frac{a_{{n+1}}^{2(1+\delta)}}{16(12 a_{{n}}^\gamma)}\right) \\
        &\le \frac{1}{4} \sum_{k=0}^{n-1} a_k^{1 +\delta} \leq  a_0^{1+\delta}/2\,.
    \end{aligned}
    \end{equation}
    In the second inequality, we play the same trick than in {\eqref{eq:POmega_n_c} for} the loop case : as $\gamma - 2(1+\delta)^2 > 0$ by \eqref{hyp:gamma}, apply the fact that $\exp(-x) \le k!/x^k$ for every $k \in \N, x \ge 0$ with some $k$ such that
    \[k(\gamma -2(1+\delta)^2) > 1+ 2 \delta \,,\]
    and $a_0$ is chosen small enough to reabsorb the constant $ {4} k !\cdot (16\cdot 12)^k$ with $a_0^\delta$.
    For any $x \in G_n$ and $1 \leq k \leq n$ we now define 
    $$\tau_{k-1} (\omega) = \sup  \left \{ s \in  [t_{k-1}, t_{k}] : | \bY^n_{T- t_n , s} (x , \omega) - \bX^n_{T- t_n , s} (x ) | >   {a_{k}^{1 + \delta}} \right  \} \,, $$
    and 
    $$ \tau_{n} (\omega) = \sup  \left \{ s \in  [t_{n},  T - t_{n}] : | \bY^n_{T- t_n  , s} (x , \omega) - \bX^n_{T - t_n , s} (x ) | >   {a_{n}^{1 + \delta}} \right  \} \,, $$
    and we claim that $\tau_k \equiv t_{k}$ for any $\omega \in \Omega_n$ and any $1 \leq k \leq n$. From this property we conclude the proof, using also that $\P(\Omega_{n}^c) \leq a_0^{1+\delta}/2 \,. $
    We iteratively prove the claim. For $\tau_n$ we observe that $\bb_n \equiv 0$ on $[t_n, T - t_n]$ and then $\tau_n (\omega) \equiv t_n$ thanks to the definition of $\Omega_{n, 2}$.
    
     Suppose it is true for any $j \geq k$, then we want to prove the claim for $k-1$. We define $\gamma (s) = \bX^n_{T- t_n , s} (x ) $ and we suppose by contradiction that $\tau_{k-1} (\omega) > t_{k-1}$. {For every $\omega \in \Omega_n$}
    \begin{align*}
        |\bX^n_{T- t_n,\tau_{k-1}}(x)& -\bY^n_{T- t_n,\tau_{k-1}}(x,\omega)|  \\
        &\le |  \bX^n_{T- t_n,\tau_{k}}(x)-\bY^n_{t_n,\tau_{k}}(x,\omega) | +  \sup_{s \in [t_{k-1} , t_k[} |\bW_s (\omega) - \bW_{t_k}(\omega)|  
        \\
        &  \quad + \int_{\tau_{k-1}}^{t_k}|\bb_n(s,\bX^n_{T- t_n,s})-\bb_n(s,\bY^n_{T- t_n,s})|ds
        \\
        & \leq a_{k+1}^{1 + \delta} +\frac{ a_k^{1 + \delta}}{2} + \int_{\tau_{k-1}}^{t_k} \|\nabla \bb_n \|_{L^\infty (B_{a_{k}^{1+ \delta}} (\gamma (s) ))} | \bX^n_{T- t_n,s} (x) - \bY^n_{T- t_n,s} (x, \omega)|ds 
        \\
        & = a_{k+1}^{1 + \delta} + a_k^{1 + \delta}/2  < a_k^{1+ \delta}\,,
    \end{align*}
    where {in the second inequality, we used the induction hypothesis and $\omega \in \Omega_n$. In} the equality we used property \eqref{prop:item:3} of Proposition \ref{prop:chessproperties} {together with the bound $| \bX^n_{T- t_n,s} (x) - \bY^n_{T- t_n,s} (x, \omega)| \le a_k^{1+\delta}$ for all $s \in [\tau_{k-1}, t_k]$.} {The result} contradicts {the definition of $\tau_{k-1}$}, concluding the proof.
\end{proof}

For the time-interval $]T-t_n,T]$, by symmetry $b_n (t, x) = - b_n(T- t, x)$ for any $t > \frac{T}{2}$, we have the following forward stability. The proof is similar, using $\tilde{G}_n \subset \T^2$ given by \eqref{prop:item:4} of Proposition \ref{prop:chessproperties}.

\begin{lem}(Forward stability)\label{lem:chesssstabfor}
     For any $n\ge 0$, let $\tilde{G}_n \subset \T^2$ and $t_n \in [0,T]$ as in Proposition \ref{prop:chessproperties}, and let $\bX^n$ and $\bY^n$ denote the regular Lagrangian flow and the stochastic flow of $b_n$ respectively. Then, the following inequality holds
    \begin{equation*}
        \sup_{t\in [T - t_n, T ]} \sup_{x \in \tilde{G}_n} \int_{\Omega} |\bX^n_{ T- t_n, t}(x)-\bY^n_{T- t_n, t}(x,\omega)| d\P(\omega) \le a_0^{1+\delta}\,.
    \end{equation*}
\end{lem}

All in all, thanks to the  backward and the forward stability, we prove some quantitative closeness  of the solutions to the \eqref{eq:TE} to the solutions of the \eqref{eq:ADE} with velocity field $b_n$. More precisely we prove the following quantitative bound.
\begin{lem}\label{lem:TE-ADE_close}
    Let $\solAD_n$ and $\solTE_n$ be the solutions to the advection--diffusion  and advection--equation respectively  with initial datum $\theta_{\rm in}$ and velocity field $b_n$ as in Proposition \ref{prop:chessproperties}. For any $f \in W^{1,\infty}$ we have the following bound for any $n \in \N$
        \begin{align*}
        \left\vert \intTd[2] f(x) (\solAD_n (T,x) - \solTE_n (T,x)) dx\right\vert & \le a_0^{1+ \delta} ( \| \ADin \|_{L^\infty}  \| \nabla f \|_{L^\infty} + \| \nabla \ADin \|_{L^\infty}  \|  f \|_{L^\infty} )  
        \\
        & \quad + 2 \| f \|_{L^\infty} ( \mathcal{L}^2 (G_n^c) +  \mathcal{L}^2  ( \tilde{G}_n^c) ) \,,
        \end{align*}
\end{lem}
\begin{proof}
The idea is to use the backward stability in the time interval $[0, T- t_n]$ of Lemma \ref{lem:chessstabback} using formula \eqref{eq:FKbackward} and the forward stability of Lemma \ref{lem:chesssstabfor} using formula \eqref{eq:FKforward}.

    
    We introduce the solution to the advection--diffusion equation starting from time  $T-t_n$ as the solution to the advection--equation. More precisely,  we define $\solbAD[n]$ as the solution of 
    \begin{equation}\label{eq:solbAD}
        \begin{cases}
            \partial_t \solbAD[n] + \dv(\bb_n \solbAD[n]) = \Delta \solbAD[n]\,,\\
            \solbAD[n](T-t_n,x) = \solTE_n(T-t_n,x)\,.
        \end{cases}
    \end{equation}
   From this definition we split the integral as
    \begin{align*}
        \left\vert \intTd[2] f(x)(\solAD_n(T,x) - \solTE_n(T,x)) dx\right\vert & \leq  \underbrace{\left\vert \int_{\T^2} f(x)(\solAD_n(T,x) - \solbAD[n](T,x)) dx \right\vert}_{=:I} 
        \\
        & + \underbrace{\left\vert \int_{\T^2} f(x)(\solbAD[n](T,x) - \solTE_n(T,x)) dx \right\vert}_{=:II}  \,.
    \end{align*}
    We now split 
    $$ I  \leq \left\vert \int_{G_n} f(x)(\solAD_n(T,x) - \solbAD[n](T,x)) dx \right\vert + \left\vert \int_{G_n^c} f(x)(\solAD_n(T,x) - \solbAD[n](T,x)) dx \right\vert = III + IV\,.$$
   Thanks to  energy estimate we have 
   \begin{align*}
       III &\leq \|f \|_{L^2 (\T^2)}  \| \solAD_n (T-t_n, \cdot ) - \solbAD[n] (T-t_n, \cdot) \|_{L^2 (G_n)}\\
       &\leq \linf{f}{\T^2} \| \solAD_n (T- t_n, \cdot ) - \solbAD[n] (T- t_n , \cdot ) \|_{L^\infty (G_n)}\,, 
   \end{align*}
    and thanks to {\eqref{eq:solbAD},} 
    Lemma \ref{lem:chessstabback} and  the backward Feynman-Kac formula \eqref{eq:FKbackward} we have 
    \begin{equation*}
    \| \solAD_n (T- t_n, \cdot ) - \solbAD[n] (T -t_n , \cdot ) \|_{L^\infty (G_n)}  = \| \solAD_n (T- t_n, \cdot ) - \solTE_n (T- t_n , \cdot ) \|_{L^\infty (G_n)} \leq  \linf{\nabla \ADin}{\T^2} a_0^{1+\delta}.
    \end{equation*}
    We also estimate thanks to the maximum principle and Proposition \ref{prop:chessproperties} for the bound on $G_n^c$
    $$IV \leq  2  \| f \|_{L^\infty} \mathcal{L}^2 (G_n^c ) \,.$$
    
     Similarly, we employ  the forward Feynman-Kac formula \eqref{eq:FKforward} and Lemma \ref{lem:chesssstabfor} to get
     \begin{equation*}
         II \le  \linf{\nabla f}{\T^2} \linf{\ADin}{\T^2} a_0^{1+\delta} + 2 \| f \|_{L^\infty} \mathcal{L}^2 (\tilde{G}_n^c) \,.
     \end{equation*} 
    from which we conclude the proof.
\end{proof}

We are now ready to prove the main theorem.

\subsection{Proof of Theorem \ref{thm:NU_ADE_chess}}


%

Firstly, we need to prove that there are at least two distinct solutions, more precisely we will prove, up to not relabelled subsequences,
$$ \theta_{2n} \overset{*}{\rightharpoonup} \theta_{\rm even} \neq \theta_{\rm odd} \overset{*}{\leftharpoonup} \theta_{2n +1} \,. $$
More precisely, we claim that  there exists $f \in C^\infty$ such that 
$$ \lim_{n \to \infty } \intTd[2] f(x) ( \solAD_{2n} (T,x) - \solAD_{2n +1} (T,x) ) dx > 0 \,. $$
 Recalling \eqref{eq:AqBq} for $A_0, B_0$ and \eqref{eq:psi} for $\psi$, we define $$f(x)= \left ( \one_{A_0[5 a_0^{1+\delta/2}]   }  - \one_{ B_0[5 a_0^{1+\delta/2}]} \right ) \star \psi(x)$$
    then  $f$ satisfies
    \begin{equation}\label{eq:f_properties}
      \|f\|_{L^\infty(\T^2)}\le 1 \qquad \text{and} \qquad \|\nabla f\|_{L^\infty(\T^2)} \le a_0^{-1-\delta/2}. 
    \end{equation} 
    
    We now split the integral 
    \begin{align}\label{eq:3pieces}
        \intTd[2] f(x) ( \solAD_{2n} (T,x) - \solAD_{2n +1} (T,x) ) dx
        &\ge  \intTd[2] f(x) ( \solTE_{2n} (T,x) - \solTE_{2n +1} (T,x)) dx 
        \\ \nonumber
        \quad & -\intTd[2] |f (x)  ( \solAD_{2n}(T,x) - \solTE_{2n} (T,x)) | dx 
        \\
        \quad & -  \intTd[2] |f (x) ( \solAD_{2n+1}(T,x) - \solTE_{2n+1} (T,x) )| dx\,. \notag
    \end{align}
    
    Using Proposition  \ref{prop:chessproperties} we have 
$$\intTd[2] f(x) ( \solTE_{2n} (T,x) - \solTE_{2n +1} (T,x)) dx \geq 2 \mathcal{L}^2 ( A_0[10 a_0^{1+\delta/2}]  \cup B_0[10 a_0^{1+\delta/2}] ) - 100 a_0^{\delta /2} \geq 1 \,.$$

 By Lemma \ref{lem:TE-ADE_close} and estimates \eqref{eq:f_properties} and \eqref{eq:initialchess} and estimates $\mathcal{L}^2 ({G}_n^c \cup \tilde{G}_n^c ) \leq 1/5 $ thanks to Proposition \ref{prop:chessproperties}, for every $n \ge 0$ we can bound 
    \begin{align*}
        \intTd[2] | f(x) ( \solAD_n(T,x) - &\solTE_n(T,x) )| dx \\
        &\le  (\linf{\ADin}{\T^2}\linf{\nabla f}{\T^2}+ \linf{\nabla\ADin}{\T^2}\linf{f}{\T^2}) a_0^{1+\delta} + \frac{2}{5} \\
        &\le  2a_0^{\delta/2} + \frac{2}{5} < \frac{1}{2} \,,
    \end{align*}
    which concludes the proof of the existence of at least two {distinct} solutions.
    Finally, arguing as in the proof of Theorem \ref{thm:NU_ADE_loops} (see Section \ref{sec:concluding-proof}), we can prove that such solutions satisfies the local energy inequality.

\bibliographystyle{alpha}
 \bibliography{biblio}

\end{document}

%% file: analysis_preamble.tex
\usepackage{graphicx}
\usepackage{bbm}

\usepackage{amsmath}
\usepackage{dsfont}
\usepackage{amssymb}
\usepackage{mathtools}

\usepackage{amsthm}
\newtheorem{thm}{Theorem}[section]
\newtheorem{prop}[thm]{Proposition}
\newtheorem{lem}[thm]{Lemma}

\theoremstyle{definition}
\newtheorem{df}{Definition}

\theoremstyle{remark}

\newcommand\numberthis{\addtocounter{equation}{1}\tag{\theequation}}

\def\R{\mathbb{R}}

\def\N{\mathbb{N}}
\def\Z{\mathbb{Z}}

\def\T{\mathbb{T}}

\def\E{\mathbb{E}}
\def\P{\mathbb{P}}

\def\dP{\mathbb{P}}
\def\dW{\mathbb{W}}

\def\F{\mathcal{F}}
\def\I{\mathcal{I}}
\def\J{\mathcal{J}}
\def\L{\mathcal{L}}

\def\U{\mathcal{U}}

\def\one{\mathds{1}}

\def\bb{\textit{b}}

\def\bW{\textbf{W}}
\def\bX{\textbf{X}}
\def\bY{\textbf{Y}}

\newcommand\linf[2]{\lVert#1\rVert_{L^\infty(#2)}}

\newcommand{\intTd}[1][d]{\int_{\T^{#1}}}

\def\half{\frac{1}{2}}

\DeclareMathOperator{\dv}{div}
\DeclareMathOperator{\dist}{dist}

\DeclareMathOperator{\supp}{supp}

%% file: nonuniADE_preamble.tex
\def\solTE{\rho}

\def\solAD{\theta}
\newcommand{\solbAD}[1][]{\bar{\theta}_{#1}}

\def\solADe{\theta_{even}}
\def\solADo{\theta_{odd}}

\def\TEin{\rho_{\text{in}}}
\def\ADin{\theta_{\text{in}}}

\def\SS{S}